\documentclass[12pt,reqno]{amsart}

\usepackage{graphicx}
\usepackage{amsmath,amsthm,amscd,amssymb}
\usepackage{amsfonts}
\usepackage{latexsym}
\usepackage{epsfig,epstopdf}
\usepackage{enumerate}

\newtheorem*{thm*}{Theorem \ref{main-general}*}
\newtheorem*{thm1*}{Proposition \ref{small data}*}
\newtheorem*{thm2*}{Proposition \ref{longperturbation}*}
\newtheorem*{thm3*}{Proposition \ref{H^1 Scattering}*}

\numberwithin{equation}{section}
\newtheorem{thm}{Theorem}
\numberwithin{thm}{section}
\newtheorem*{thmA*}{Theorem A}
\newtheorem*{thmB*}{Theorem A*}
\newtheorem{lemma}[thm]{Lemma}
\newtheorem{prop}[thm]{Proposition}
\newtheorem{cor}[thm]{Corollary}

\theoremstyle{definition}
\newtheorem{definition}[thm]{Definition}

\theoremstyle{remark}
\newtheorem{rmk}[thm]{Remark}
\newtheorem{claim}[thm]{Claim}

\newcommand{\N}{\mathbb{N}}
\newcommand{\Z}{\mathbb{Z}}
\newcommand{\R}{\mathbb{R}}
\newcommand{\C}{\mathbb{C}}
\newcommand{\at}{\mathcal{T}}

\newcommand{\IR}{\Rn \times I}
\newcommand{\Rn}{\R^{d}}
\newcommand{\dH}{\dot{H}}
\newcommand{\Hs}{\dH^{s}}
\newcommand{\dB}{\dot{\beta}}

\newcommand{\tM}{\tilde M}

\newcommand{\NLS}{\textrm{NLS}}

\newcommand{\NLSf}{\NLS_{ p}(\Rn)}

\newcommand{\lap}{\Delta}

\newcommand{\GB}{\text{GBG}}
\newcommand{\g}{\mathcal{G}}
\newcommand{\Pu}{\mathcal{P}}
\newcommand{\ME}{\mathcal{ME}}
\newcommand{\uQ}{u_{_Q}}

\newcommand{\SHsj}{S(\dH^{s},I_j)}
\newcommand{\SHdsj}{S^{\prime}(\dH^{-s},I_j)}

\newcommand{\dHs}{\dot{H}^{s}}
\newcommand{\dHds}{\dot{H}^{-s}}

\newcommand{\Lt}{L^2}
\newcommand{\SdHs}{S^{\prime}(\dHds)}
\newcommand{\SdLt}{S^{\prime}(\Lt)}

\newcommand{\SHs}{S(\dHs)}
\newcommand{\SLt}{S(\Lt)}

\newcommand{\wt}{\widetilde{W}}
\newcommand{\tpsi}{\widetilde{\psi}}
\newcommand{\te}{\tilde{e}}
\newcommand{\tu}{\tilde{u}}

\newcommand{\z}{\bar{z}}
\newcommand{\w}{\bar{w}}

\newcommand{\ds}{\displaystyle}

\newcommand{\qsd}{L^{\frac {d}{2 s}}_t} 
\newcommand{\qsdI}{L^{\frac {d}{2 s}}_{I_j}} 
\newcommand{\qsdtail}{L^{\frac {d}{2 s}}_{[t,\infty)}} 
\newcommand{\rsd}{L^{\frac{2 d^2(p-1)}{d^2 (p - 1) + 16}}_x}
\newcommand{\qs}{L^{\frac{d p}{2s}}_t} 
\newcommand{\qsI}{L^{\frac{d p}{2s}}_{I_j}}
\newcommand{\qstail}{L^{\frac{d p}{2s}}_{[t,\infty)}} 
\newcommand{\rs}{L^{\frac {2d^2 p}{d^2p-8s}}_x}
\newcommand{\rsh}{L^{\frac{d^2p(p-1)}{2(d+4)}}_x}

\newcommand{\rua}{L^{\frac{4 d^2 ( p-1)}{(d + 4) (d - d p + 8) + d^2 p (p - 1)}}_x}
\newcommand{\rda}{L^{\frac{8 d^2 p}{(p - 1)^2 ((d^2 - 3 d s + 2 s^2) (d + 4) + 8 s^2)}}_x}
\newcommand{\rta}{L^{\frac{ d^2 p (p - 1)}{2 (d + 4)}}_x}
\newcommand{\qda}{L^{\frac{ d  p}{2s (p-1)}}_t} 
 
\newcommand{\qca}{L^{\frac{d (p - 1)}{s}}_t} 
\newcommand{\rca}{L^{\frac{d^2 (d - 1)}{2 (d - s)}}_x}
\newcommand{\rcia}{L^{\frac{d^2 (d - 1)}{2 d + s^2 (p - 1)^2}}_x}
\newcommand{\qcia}{L^{\frac{d}{s}}_t}
\newcommand{\roa}{L^{\frac{2 d^2}{d^2 - 4 s}}_x}

\newcommand{\qsdq}{\frac {d}{2 s}}

\newcommand{\rsdr}{\frac{2 d^2(p-1)}{d^2 (p - 1) + 16}}
\newcommand{\qsq}{\frac{d p}{2s}}

\newcommand{\rsr}{\frac {2d^2 p}{d^2p-8s}}
\newcommand{\rshr}{\frac{d^2p(p-1)}{2(d+4)}}

\newcommand{\ruar}{\frac{4 d^2 ( p-1)}{(d + 4) (d - d p + 8) + d^2 p (p - 1)}}

\newcommand{\qcaq}{\frac{d (p - 1)}{s}} 
\newcommand{\rcar}{\frac{d^2 (d - 1)}{2 (d - s)}}

\newcommand{\qciaq}{\frac{d}{s}}
\newcommand{\roar}{\frac{2 d^2}{d^2 - 4 s}}

\newcommand{\rubr}{\frac{2 (d + 4) + d (p - 1)^3}{ d^2 (p - 1)}}

\newcommand{\Ds}{D^{s}}
\newcommand{\Dsa}{D^{\frac{s(p-1)}{2}}}

\newcommand{\expa}{\frac {p-1}{2}}

\newcommand{\rub}{L^{\frac{2 (d + 4) + d (p - 1)^3}{ d^2 (p - 1)}}_x}
\newcommand{\rdb}{L^{\frac{d^2p}{2 (d + 4) + d p(p - 1)^2 }}_x}
\newcommand{\qdb}{L^{\frac{d  p}{2s (p-1)}}_t}
\newcommand{\rtb}{L^{\frac{ d^2 p (p - 1)}{2 (d + 4)}}_x}
 
\newcommand{\qcb}{L^{\frac{d (p - 1)}{s}}_t} 
\newcommand{\rcb}{L^{\frac{d^2 (p - 1)}{2 (d - s)}}_x}
\newcommand{\rcib}{L^{\frac{2d^2}{d^2 + 2 d (p - 1)^2 - 2 s (d + 2)}}_x}%
\newcommand{\qcib}{L^{\frac{d}{s}}_t}
\newcommand{\rob}{L^{\frac{2 d^2}{d^2 - 4 s}}_x}
\newcommand{\Dsb}{D^{(p-1)^{2}}}

\newcommand{\expbu}{\frac {(p - 1) (1 + s - p)}{s}}
\newcommand{\expbd}{\frac {(p - 1)^2 }{s}}

\newcommand{\qzh}{L^{\frac{12 (d - 2 s)}{(8 + 3 d - 6 s) (1 - s)}}_{I_j}}
\newcommand{\qzht}{L^{\frac{12 (d - 2 s)}{(8 + 3 d - 6 s) (1 - s)}}_{t}}

\newcommand{\rzh}{L^{\frac{6 d (d - 2 s)}{  3 (d^2+ 2s^2)+ 9 d (1 - s) - 2(5 s + 4)}}_x}

\newcommand{\qdht}{L^{\frac{4}{1- s}}_{t}}
\newcommand{\qdh}{L^{\frac{4}{1- s}}_{I_j}}
\newcommand{\rdh}{L^{\frac{2 d}{d -s-1}}_x}
\newcommand{\rdhr}{\frac{2 d}{d -s-1}}
\newcommand{\qdhtq}{\frac{4}{1- s}}

\newcommand{\quht}{L^{\frac{6}{1-s}}_{t}}
\newcommand{\quh}{L^{\frac{6}{1-s}}_{I_j}}
\newcommand{\ruh}{L^{\frac{6 d}{3 d-4s-2}}_x}
\newcommand{\ruhr}{\frac{6 d}{3 d-4s-2}}

\newcommand{\quhtq}{\frac{6}{1-s}}

\newcommand{\profq}{\frac{2 ( d+2)}{d - 2 s}}

\newcommand{\crit}{{\textnormal{c}}}

\DeclareMathOperator{\im}{Im}

\begin{document}

\pagestyle{headings}

\title[Scattering and Blow-up]{Global Behavior Of Finite Energy Solutions To The $d-$ Dimensional Focusing Nonlinear Schr\"odinger Equation }
\author[C. Guevara]{Cristi Guevara}
\address{School of Mathematical and Statistical Sciences, Arizona State University, Tempe, Arizona, 85287}

\email{Cristi.guevara@asu.edu}

\begin{abstract}
We study the global behavior of  finite energy solutions to the $d$-dimensional focusing nonlinear Schr\"odinger equation (NLS), $i \partial _t u+\Delta u+ |u|^{p-1}u=0, $
with initial data $u_0\in H^1,\; x \in \Rn$.  The nonlinearity power $p$ and the dimension $d$ are such that the scaling index $s=\frac{d}2-\frac2{p-1}$ is between 0 and 1, thus, the NLS is mass-supercritical $(s>0)$ and energy-subcritical $(s<1).$ 

For solutions with $\ME[u_0]<1$ ($\ME[u_0]$ stands for an invariant and conserved quantity in terms of the mass and energy of $u_0$), a sharp threshold for scattering and blowup is given. Namely, if the renormalized gradient $\g_u$ of a solution $u$ to NLS is initially less than 1, i.e., $\g_u(0)<1,$ then the solution exists globally in time and scatters in $H^1$ (approaches some linear Schr\"odinger  evolution as $t\to\pm\infty$); if the renormalized gradient $\g_u(0)>1,$ then the solution exhibits a blowup behavior, that is, either a finite time blowup occurs, or there is a divergence of $H^1$ norm in infinite time.  

This work generalizes the results for the 3d cubic NLS obtained in a series of papers by  Holmer-Roudenko and Duyckaerts-Holmer-Roudenko with the key ingredients,  the concentration compactness and localized variance, developed in the context of the  energy-critical NLS and Nonlinear Wave equations by Kenig and Merle.

\end{abstract}

\maketitle

\section{Introduction}\label{Background}

In this paper, we consider the focusing  Cauchy problem for the nonlinear Schr\"odinger equation (NLS), denoted by $\NLSf$, with finite energy initial data  (i.e., $u_0 \in H^1(\Rn)$),
\begin{align}
	\left\{\begin{array}{ccc}i \partial _t u+\lap u+ |u|^{p-1}u=0  
		\\u(x,0)=u_0(x) \in H^1(\Rn),
		\end{array}\right.
	\label{eq:NLS}
\end{align}
where $u=u(x,t)$ is a complex-valued function in space-time $\Rn_x\times \R_t$,  $p\geq 1$.

For a fixed $\lambda \in (0,\infty)$, the rescaled function $u_{\lambda}(x,t):=\lambda^{\frac{2}{p-1}}u(\lambda x, \lambda^2 t)$ is a solution of $\NLSf$ in \eqref{eq:NLS} if and only if $u(x, t)$ is. This scaling property gives rise to the scale-invariant norms. Sobolev norm  $\dH^{s_c}(\Rn)$ with $s_c:=\frac d2-\frac2{p-1}$.Ginibre-Velo \cite{GiVe79a,GiVe79b}  showed that the initial-value problem $\NLSf$ with initial data $u(x,0)=u_0(x)\in H^1(\Rn)$, $1\leq p<1+\frac{4}{d-2}$  is locally well-posed in $H^s(\Rn)$ with $s\geq1$. Later, Cazenave-Weissler \cite{CaWe90} showed that for small initial data in $\dH^{s}(\Rn)$, with $0\leq s<\frac{d}{2}$ and $0<p\leq \frac{d+2}{d-2},$  there exists a unique solution to $\NLSf$ defined for all times. If the data is not small, we can define the  maximal interval of existence of solutions to $\NLSf$ and denote it   by  $(T_*,T^*)$. We say a solution is global in forward time if $T^*=+\infty.\;$ Similarly, if   $T_*=-\infty,$ the solution is global in backward time. A solution is global if ($T_*,T^*)=\R.$

On their maximal interval of existence solutions to \eqref{eq:NLS} have three conserved quantities:  mass, energy  and momentum, where
\begin{align*}
	M[u](t)&=\int_{\Rn} |u(x,t)|^2dx=M[u_0], 
	\\
	E[u](t)&=\dfrac 12 \int_{\Rn}|\nabla u(x,t)|^2dx-\dfrac {\mu}{p+1} \int_{\Rn}|u(x,t)|^{p+1} dx=E[u_0],  
	\\
 	P[u](t)&=\mbox{Im}\int_{\Rn}\bar{u}(x,t)\nabla u(x,t)dx=P[u_0]. 
	\end{align*}
The following quantities are  scaling invariant:
$$
E[u]^{s_c}M[u]^{1-s_c},\quad \mbox{ and }
\quad \|u\|_{L^{2}(\Rn)}^{1-s_c}\|\nabla u\|_{L^{2}(\Rn)}^{s_c}.$$
These quantities were first introduced in   \cite{HoRo07} in the context of mass-supercritical $\NLS$ $(0<s<1)$ and used to classify the global behavior of solutions.

We say that a global solution $u(t)$ to  $\NLSf$  \emph{scatters in $H^s(\Rn)$} as $t\to +\infty$   if
there exists $\psi^+\in H^s(\Rn)$ such that 
\begin{align}\label{eq:scatter}
\lim_{t\to +\infty}\|u(t)-e^{it\lap}\psi^+\|_{H^s(\Rn)}=0.
\end{align}
Similarly, we can define scattering in $H^s(\Rn)$ for $t\to-\infty.$

For the $L^2$-critical NLS equation (i.e. $s=0$)  with $u_0\in H^1(\Rn)$, Weinstein in \cite{We82} established a sharp threshold for global existence, namely, the condition $\|u_0\|_{L^2(\Rn)} < \|Q\|_{L^2(\Rn)}$, where $Q$ is the ground state solution (see  Section \ref{SEandGS}), guarantees a global existence of evolution $u_0\leadsto u(t)$.  Solutions at the threshold mass, i.e., when $\|u_0\|_{L^2(\Rn)} = \|Q\|_{L^2(\Rn)}$, may   blowup in finite time. Such solutions are called the minimal mass   blowup solutions. Merle in \cite{Me93} characterized  the minimal mass   blowup $H^1$ solutions showing that all such solutions are pseudo-conformal transformations of the ground state (up to $H^1$ symmetries), that is, $$u_{_T}(x,t)=\frac{e^{i/(T-t)}e^{i|x|^2/(T-t)}}{T-t}Q\left(\frac x{T-t}\right).$$

In the energy-critical case $s=1,$  Kenig-Merle \cite{KeMe06} studied global behavior of solutions with $u_0\in \dH^1(\Rn)$ in dimensions $d=3, 4$, and $5$ and showed that under a certain energy threshold (namely, $E[u_0]<E[W]$, where $W$ is the positive solution of $\Delta W+W^p=0$, decaying at $\infty$), it is possible to characterize global existence versus finite  blowup depending on the size of $\|\nabla u_0\|_{L^2(\Rn)}$, and also prove scattering for globally existing solutions. To obtain the last property, they  applied the concentration-compactness and rigidity technique.
The concentration-compactness method appears in the context of wave equation in G\'{e}rard \cite{Ge96} and NLS 
in Merle-Vega \cite{MeVe98}, which was later followed by Keraani \cite{Ke01}, and dates back to works of  P-L.  Lions \cite{Lions84} and Brezis-Coron \cite{BrCo85}. The rigidity argument (estimates on a localized variance) is the technique of Merle from mid 1980's. 

The mass-supercritical and energy-subcritical case ($0<s<1$) is discussed in detail in the next section, and the energy-supercritical case ($s>1$) is largely open. 

In the mass- supercritical and energy-subcritical case ($0<s_c<1$) the 3d cubic NLS equation with $u_0 \in H^1$ was studied in a series of papers \cite{HoRo08}, \cite{DuHoRo08}, \cite{DuRo08}, \cite{HoRo09} and \cite{HoPlRo09}. The authors obtained a sharp scattering threshold for radial initial data in \cite{HoRo08}, under a so called \emph{mass-energy threshold}
$M[u]E[u]<M[Q]E[Q],$ where $Q$ is the ground state solution. The extension of these results to the nonradial data is in \cite{DuHoRo08}.  Behavior of solutions and characterization of all solutions at the \emph{mass-energy threshold} $M[u]E[u]=M[Q]E[Q]$ is in \cite{DuRo08}. For infinite variance nonradial solutions  Holmer-Roudenko in \cite{HoRo09} introduced a first application of  concentration-compactness and rigidity arguments to prove the existence of a ``weak   blowup"\footnote{  See Section \ref{wbup} for exact formulation and discussion.}. In addition, Holmer-Platte-Roudenko \cite{HoPlRo09} consider (both theoretically and numerically) solutions to the 3d cubic NLS above the \emph{mass-energy threshold} and give new   blowup criteria in that region. They also predict the asymptotic behavior of solutions for different classes of initial data (modulated ground state, Gaussian, super-Gaussian, off-centered Gaussian, and oscillatory Gaussian) and provide several conjectures in relation to the threshold for scattering.

In the spirit of  \cite{DuHoRo08}, \cite{HoRo08},  \cite{HoRo09},Carreon-Guevara \cite{CaGu11}  study  the long-term behavior of solutions for the 2d quintic NLS equation with $u_0\in H^1(s=\frac12)$.  This equation is  important to study since  it has a higher power of nonlinearity (higher than cubic), and recently  a nontrivial   blowup result (a standing ring) was exhibited by Rapha\"el in \cite{Ra06}  (there are further extensions of \cite{Ra06} to higher dimensions and different nonlinearities in \cite{RaSz09}, \cite{HoRo10a}, \cite{HoRo10b}).

\subsection{Statement of the results}\label{overview}
Throughout this document, unless otherwise specified, we assume that $0<s<1$ and $s=\frac d2-\frac 2{p-1}$, $\alpha := \frac{\sqrt{d(p-1)}}{2}$, and $\beta := 1-\frac{(d-2)(p-1)}4.$ Let
\begin{align}\label{eq:ground}
\uQ(x,t):= e^{i\beta t} Q(\alpha x).
\end{align}  
Then $\uQ(x,t)$ solves the equation \eqref{eq:NLS},  provided $Q$ solves\footnote{  Here, in the equation \eqref{eq:Qground} and definition of Q, we use the notation from Weinstein \cite{We82}.  Rescaling $Q(x) \mapsto \beta^{\frac{1}{p-1}}Q \Big(\sqrt{\frac{\beta}{\alpha}}x\Big)$ will solve a  the nonlinear elliptic equation $ -Q+\lap Q +Q^p = 0$. \\ }  
\begin{align}
 \label{eq:Qground}
- \beta \; Q +\alpha^2 \; \lap Q + Q^{p} = 0, \quad Q=Q(x), \quad x\in{\Rn}.
\end{align}
The theory of  nonlinear elliptic equations (Berestycki-Lions \cite{BeLi83a,BeLi83b})
shows that \eqref{eq:Qground} has an infinite number of solutions in $H^1(\Rn)$, but a unique solution of minimal $\Lt$-norm, which we denote by $Q(x)$. It is positive, radial, exponentially decaying (for example, \cite[Appendix B] {Tao06NLD}) and is called the \emph{ground state} solution. 

We introduce the following notation:
 \begin{align}
 \label{eq:grad}
 &\bullet\text{the~renormalized~gradient} 
&\g_{u}(t)&:=\dfrac{\|u\|_{\Lt(\Rn)}^{1-s}\|\nabla u(t)\|_{\Lt(\Rn)}^{s}}{\|\uQ\|_{\Lt(\Rn)}^{1-s}\|\nabla \uQ\|_{\Lt(\Rn)}^{s}},\\
 &\bullet\text{the~renormalized~momentum} 
 &
 \Pu[u]&:=\dfrac {P[u]^s\|u\|_{\Lt(\Rn)}^{1-2s}}{\|\uQ\|^{{1-s}}_{\Lt(\Rn)}\|\nabla \uQ\|^s_{\Lt(\Rn)}},\label{eq:mom}\\
 &\bullet\text{the~renormalized~Mass-Energy}
 &
\ME[u]&:=\dfrac{M[u]^{1-s}E[u]^{s}}{M[\uQ]^{1-s}E[\uQ]^{s}} \label{eq:ME}\quad\quad(\text{for  }E[\uQ]>0).
\end{align}

Note that \eqref{eq:ME} we only consider $E[u]>0$, since for  $E[u]<0$ the blowup is known (see \cite{VlPeTa71}, \cite{Za72}, \cite{Gl77}, \cite{GlMe}) 

\begin{rmk}[Negative energy] \label{d:negative energy}
Note that  it is possible to have initial data with $E[u]<0$ and the blowup from the dichotomy in Theorem A  Part II (a)  below applies. (It
follows from the standard convexity blow up argument and the work of
Glangetas-Merle \cite{GlMe}). Therefore, we only consider $E[u]\geq0$ in the rest of the paper.
\end{rmk}

The main result of this paper is 

 \begin{thmA*}\label{main-general}
 Consider $\NLSf$ such that $0<s<1,$  $u_0 \in {H}^1(\Rn)$, $d\ge1$  and let $u(t)$ be the corresponding solution on its  maximal time interval of existence $(T_*,T^*)$. Assume 
\begin{align}\label{mass-energy}
\left(\ME[u]\right)^{\frac1s}-\frac d{2s}\left(\Pu[u]\right)^{\frac2s}<1.
\end{align} 
\begin{enumerate}
\item[I.] If  
\begin{align}\label{ground-momentum1}
\left[\g_u(0)\right]^{\frac2s}-\left(\Pu[u]\right)^{\frac2s}<1,\;
\end{align}
 then
\item[(a)] $\left[\g_u(t)\right]^{\frac2s}-\left(\Pu[u]\right)^{\frac2s}<1$ for all $t\in \R$, and thus, the solution is global in time  ($T_*=-\infty$, $T^*=+\infty$) and 
\item[(b)] $u$ scatters in $H^1(\Rn)$, i.e., there exists $\phi_\pm\in H^1(\Rn)$ such that 
\begin{align*}
\lim_{t\to \pm\infty}\|u(t)-e^{it\lap}\phi_\pm\|_{H^1(\Rn)}=0.
\end{align*}
\item[II.] If   
\begin{align}\label{ground-momentum2}
\left[\g_u(0)\right]^{\frac2s}-\left(\Pu[u]\right)^{\frac2s}>1,\;
\end{align}
then  $\left[\g_u(t)\right]^{\frac2s}-\left(\Pu[u]\right)^{\frac2s}>1$ for all $t\in(T_*,T^*)$ and
\item[(a)] if $u_0$ is radial (for $d\geq3$ and in $d=2, \; 3<p \leq 5$) or  $u_0$ is of finite variance, i.e., $|x|u_0 \in \Lt(\Rn)$,  then the solution blows up in finite time  ( $T^* < +\infty$, $\;T_*>-\infty$).
\item[(b)] If $u_0$ is non-radial and of infinite variance, then either the solution blows up in finite time ( $T^* < +\infty$, $\;T_*>-\infty$) or there exists a sequence of times $t_n\to+\infty$ (or $t_n\to-\infty$) such that $\|\nabla u(t_n)\|_{\Lt(\Rn)}\to \infty.$ 
\end{enumerate}
\end{thmA*}

We say there is  a ``weak   blowup" occurs if  $\ME[u]<1$, and $u(t)$ exists globally for all positive time (or negative times) and there exists a sequence of times $t_n\to \pm\infty$ such that $\|\nabla u(t_n)\|_{L^2}\to \infty$. In other words,  $\Lt$ norm of the gradient diverges along at least one infinite time sequence.

Our arguments follow \cite{DuHoRo08, HoRo07, HoRo08,HoRo09, CaGu11} which considered the focusing  $\NLS_3(\R^3)$ and $\NLS_5(\R^2)$, i.e., the integer powers of the nonlinearity. However for the general case we need to consider the fractional powers $p$ as well. To deal with them our innovation is to  use Besov spaces to treat the local theory, the long term perturbation and the $H^1$ scattering (see Propositions, \ref{small data}, \ref{longperturbation} and  \ref{H^1 Scattering}). In particular, 
the range of the Strichartz exponents is adjusted for the $d-$dimensional case, as well as the range of admissible pairs for the Kato-type estimates. And using  interpolation tricks on admissible pairs $(p,r)$ with $r < +\infty$ for the Strichartz and Kato-type estimates to avoid the pair$(2, \infty)$ which  is not $\dHs-$admissible.

The key  argument to obtain scattering\footnote{When writing this paper, we got aware of the paper \cite{FaXiCa11} proving the  scattering for general case using the embedding of $\{u\in H^1;\ME[u]<1\, \text{ and }\, \g_u(t)<1\}$to  $L^{\frac{2(p-1)(p+1)}{4-(d-2)(p-1)}}([0,\infty),L^{p+1}(R^d))$ instead of the Besov spaces as we do to treat the nonlinearity. Although, the scattering result in both paper is the same, we provide a different approach (via Besov spaces). Furthermore,  our approach lets us also obtain the ``weak blowup" result.} 
and ``weak blowup"  is the concentration compactness technique together with a rigidity theorem. Note that for $2<q<\frac{2d}{d-2}$ the embedding $H^1(\Rn)\hookrightarrow L^q(\Rn)$ is not compact\footnote{  In fact, given any $f\in H^1(\Rn)$, the sequence $f_n(x)=f(x-x_n)$, where the sequence $x_n\to \infty$ in $\Rn$, is uniformly bounded in $ H^1(\Rn)$, but has no convergent sequence on $L^q$.}; however, a profile decomposition allows to manage this lack of compactness and to produce a ``critical element". Then a localization principle proves scattering or ``weak   blowup", depending on the initial assumptions. 

The structure of this paper is as follows: Section 2 reviews the local theory,  the properties of the ground state and  reduction of the problem with nonzero momentum  to the case $P[u]=0$ via Galilean transformation for the equation \eqref{eq:NLS}. In Section 3 we present the outline of  concentration compactness machinery and localized virial identity. We include the detailed proofs for the linear and nonlinear profile decompositions, these are the key to prove scattering  and to obtain  the ``weak blowup" in Section 4.

\subsection{Notation} \label{notation}
The \emph{space-time} norms are
\begin{align*}
\|u\|_{L^q_tL^r_x(\R\times\Rn)}=\|u\|_{L^q_tL^r_x} :=\left(\int_{\R}\left(\int_{\Rn} |u(x,t)|^r dx\right)^{\frac{q}{r}}dt\right)^{\frac{1}{q}}, 
\end{align*}
with the corresponding changes when either $q=\infty$ or $r=\infty.$  

Consider the \emph{Littlewood-Paley} projection operators: if $\varphi \in C_{comp}^{\infty}(\Rn)$ be such that

\noindent$
\varphi(\xi)=\left\{\begin{array}{cc}1 &\qquad |\xi|\leq 1 \\0 &\qquad |\xi|\geq 2\end{array}\right.
$.
For each dyadic number $N\in 2^{\Z}$ and a Schwartz function $f$, define the Littlewood-Paley operators\quad
$\widehat{P_{\leq N}f}(\xi)
:=\varphi\Big(\frac{\xi}{N} \Big)\hat f(\xi)$,\quad
$\widehat{P_{> N}f}(\xi):=\Bigg(1-\varphi\Big(\frac{\xi}{N} \Big)\Bigg)\hat f(\xi)$,
\quad $\widehat{P_{ N}f}(\xi):=\Bigg(\varphi\Big(\frac{\xi}{N} \Big)-\varphi\Big(\frac{2\xi}{N} \Big)\Bigg)\hat f(\xi).$

 For $1\leq p,q\leq \infty$ and ${\sigma}>\frac{d}{p},\;$   the \emph{inhomogeneous Besov space} 
$\beta^{\sigma}_{p,q}(\Rn)=\big\{u\in S'(\Rn): \|u\|_{B^{\sigma}_{p,q}(\Rn)}<\infty\big\},$ 
where
\begin{align*}
\big\|u\big\|_{\beta^{\sigma}_{p,q}(\Rn)}
:=\|P_{ \leq N}u\|_{L^p_x}+\Bigg(\sum_{j=1}^{\infty} \Big(2^{j{\sigma}}\|P_{2^j}u\|_{L^p_x}\Big)^q\Bigg)^{\frac1q}
=\|P_{ \leq N}u\|_{L^p}+\Big(\sum_{N\in 2^{\Z}} \big(N^{\sigma}\|P_{N}u\|_{L^p_x}\big)^q\Big)^{\frac1q},
\end{align*}
 and  the \emph{homogeneous Besov space}
$\dB^{\sigma}_{p,q}(\Rn)=\big\{u\in S'(\Rn): \|u\|_{\dB^{\sigma}_{p,q}(\Rn)}<\infty\big\},$
where
$$\big\|u\big\|_{\dB^{\sigma}_{p,q}(\Rn)}
:=\Big(\sum_{N\in 2^{\Z}} \big(N^{\sigma}\|P_{N}u\|_{L^p_x}\big)^q\Big)^{\frac1q}.$$

Note that most of the $L^p$, $H^s$, $\dHs$, $\beta^{\sigma}_{p,q}$ and $\dB^{\sigma}_{p,q}$ norms are defined on $\Rn$, thus, we will omit the symbol $\Rn$ unless  we need a  specific  space dimension. 
\subsection{Acknowlegmets} This project was as a part of doctoral research of the author \cite{guevara} and was partially supported by grants from the National Science Foundation (NSF - Grant DMS - 080808; PI Roudenko), the Alfred P. Sloan Foundation. The author would like to thank Gustavo Ponce for discussions on the subject and Svetlana Roudenko for guidance on this topic.

\section{Preliminaries}\label{SG}

In this section, we review the  Strichartz estimates (e.g., see Cazenave \cite{Ca03}, Keel-Tao \cite{KeTa98}, Foschi \cite{Fos05}), fractional calculus tools  and local theory; these are the instruments to treat the nonlinearity $F(u)=|u|^{p-1}u$, in particular, when $p$ is fractional.  
 In addition,  we survey the ground state properties and the reduction to the zero momentum which allows us to restate Theorem A into a simpler form.

\subsection{Fractional calculus tools}

For Lemmas  \ref{chain}, \ref{Holderfd}, \ref{leibniz},  assume $p,p_i\in (1,\infty),\; \frac{1}{p}=\frac{1}{p_{i}}+\frac{1}{p_{i+1}}
$,  with $i=1,2,3.$ 

\begin{lemma}[Chain rule \cite{KePoVe93}] \label{chain}
Suppose $F \in C^1(\mathbb{C})$. Let $\sigma \in  (0,1),$  then
$$
\| D^{\sigma} F(u) \|_{L^p} \lesssim \| F'(u) \|_{L^{p_{_1}}} \| D^{\sigma} f \|_{L^{p_{_2}} }.
$$
\end{lemma}

\begin{lemma}[Leibniz rule \cite{KePoVe93}]\label{leibniz}
 Let $\sigma \in  (0,1),\;$ then
$$
\| D^{\sigma} (fg) \|_{L^p } \lesssim \bigg (\| f \|_{ L^{p_{_1}} } \| D^{\sigma} g \|_{ L^{p_{_2}}}+\| g\|_{L^{p_{_3}} } \| D^{\sigma} f \|_{ L^{p_{_4}}}\bigg).$$
\end{lemma}

\begin{lemma}[Chain rule for H\"older-continuous functions \cite{Vi07}] \label{Holderfd}
Let $F$ be a H\"older-continuous function of order $0<\rho<1$, then for every $0<\sigma<\rho$,
and $\tfrac{\sigma}{\rho}<\nu<1$ we have
\begin{align*}
\bigl\| D^\sigma F(u)\bigr\|_{L^p}
\lesssim \bigl\||u|^{\rho-\frac{\sigma}{\nu}}\bigr\|_{ L^{p_{_1}}} \bigl\|D^{\nu} u\bigr\|^{\frac{\sigma}{\nu}}_{L^{\frac{\sigma}{\nu}p_{_2}}},
\end{align*}
provided  $(1-\frac\sigma{\rho \nu})p_{_1}>1$.
\end{lemma}

\subsection{Strichartz  type estimates}

We say the pair $(q,r)$ is $\dHs-${\it admissible} if 
\begin{align*}
\frac{2}{q}+\frac{d}{r}=\frac{d}{2}-s, \qquad \mbox{ with } \quad 2 \leq q,r \leq \infty  \quad \mbox{ and }
  \quad(q,r,d) \neq (2,\infty,2);
\end{align*}
and  the pair $(q,r)$ is $\frac{d}{2}-${\it acceptable} if 
\begin{align*}
1\leq q,r \leq \infty, \qquad \frac{1}{q}<d\Big(\frac12-\frac1r\Big), \quad  \mbox{ or } \quad (q,r)=(\infty,2).
\end{align*}

As usual we denote by  $q'$ and $r'$ the H\"older conjugates of $q$ and $r$,  respectively (i.e., $\frac1r+\frac1{r'}=1$). Note that any $L^2-$admissible pair is also a $\frac{d}{2}-$acceptable, but not vice versa.

\subsubsection{Strichartz estimates}

The Strichartz estimates (e.g., see Cazenave \cite{Ca03}, \cite{KeTa98}, Foschi \cite{Fos05}) are
\begin{align}
\Big\|e^{it\triangle}\phi\Big\|_{L_t^qL^r_x}&\lesssim \big\|\phi\big\|_{\Lt},
\qquad\qquad
\Bigg\|\int e^{-i\tau\triangle}f(\tau)d\tau\Bigg\|_{\Lt}\lesssim\bigl\|\phi\|_{L_t^{q'}L^{r'}_x},
\label{eq:linear and inhomogeneous}\\&
\Bigg\|\int_{\tau<t} e^{i(t-\tau)\triangle}f(\tau)d\tau\Bigg\|_{L_t^qL^r_x}\lesssim\bigl \|f\|_{L_t^{q'}L^{r'}_x},
\label{eq:retarded}
\end{align}
where $(q,r)$ is an $L^2-$admissible pair. The  retarded estimate \eqref{eq:retarded}  have a wider range of admissibility (not only $L^2-$admissible) and holds when the pair $(q,r)$ is $\frac{d}{2}-$acceptable \cite{Kato}. 

In order upgrade the estimates \eqref{eq:linear and inhomogeneous} and  \eqref{eq:retarded} to the $\dHs$ level,  define the \emph{Strichartz space}  $\SHs=S(\dHs(\IR))$ 
as the closure of all test functions under the norm $\|\cdot\|_{\SHs}$ with

{\small\begin{align*}
	\| u \|_{\SHs} =  \left\{\begin{array}{ccc}
	 \sup  \left\{\| u \|_{L^q_t L^r_x} \; \begin{array}{|c}(q,r) \; \dHs-\text{admissible~with~ } \\ 
	 \big(\frac{2}{1-s}\big)^+ \leq q \leq\infty, \quad\frac{2d}{d-2s}\leq r \leq\big(\frac{2d}{d-2}\big)^- 
	 \end{array}\right\}
	 & \text{if~} d \geq 3
	   \\ &    \\
	   \sup  \left\{\| u \|_{L^q_t L^r_x} \; \begin{array}{|c}(q,r) \; \dHs-\text{admissible~with~ } \\ 
	  \big(\frac{2}{1-s}\big)^+ \leq q \leq\infty,\quad\frac{2}{1-s}\leq r \leq\big( \big(\frac{2}{1-s}\big)^+\big)'
	 \end{array}\right\}
	 & \text{if~} d =2
	   \\ &    \\
	   \sup  \left\{\| u \|_{L^q_t L^r_x} \; \begin{array}{|c}(q,r) \; \dHs-\text{admissible~with~ } \\ 
	  \frac{4}{1-2s} \leq q \leq\infty,\quad\frac{2}{1-2s}\leq r \leq\infty
	 \end{array}\right\}
	 & \text{if~} d =1.
	 \end{array}\right.
\end{align*}}
Here, $ (a^+ )'$ is defined as $ (a^+ )' :=\frac{a^+\cdot a}{a^+-a},$  so that $\frac1a=\frac1{(a^+ )'} +\frac1{a^+ }$ for any positive real value $a$, with $a^+$ being a fixed number slightly larger than $a$. Likewise, $a^-$ is a fixed number slightly smaller than $a$.

{\rmk Note that $\frac{2d}{d-2s}<\big(\frac{2d}{d-2}\big)^- <\frac{2d}{d-2}$,  if $d\ge3$. Additionally,  when $d=2$ and $s\neq\frac12,$ the quantity $r=\frac{2d}{d-2s} $ might be very large, but $\frac{2d}{d-2s}<\big( \big(\frac{2}{1-s}\big)^+\big)'$.\label{extremos}}

 Similarly, define the \emph{dual  Strichartz space}  $\SdHs=S'(\dHds(\IR))$  as the closure of all test functions under the norm $\|\cdot\|_{\SdHs}$  with
{\small
\begin{align*}
	\| u \|_{\SdHs} =  \left\{\begin{array}{ccc}
	 \inf  \left\{\| u \|_{L^{q'}_t L^{r'}_x} \; \begin{array}{|c}(q,r) \; \dHds-\text{admissible~with~ } \\ 
	 \big(\frac{2}{1+s}\big)^+ \leq q \leq\big(\frac{1}{s}\big)^-,\; \big(\frac{2d}{d-2s}\big)^+\leq r \leq\big(\frac{2d}{d-2}\big)^-
	 \end{array}\right\}
	 & \text{if~} d \geq 3
	   \\ &    \\
	   \inf  \left\{\| u \|_{L^{q'}_t L^{r'}_x} \; \begin{array}{|c}(q,r) \; \dHds-\text{admissible~with~ } \\ 
	 \big(\frac{2}{1+s}\big)^+ \leq q \leq\big(\frac{1}{s}\big)^-, \; \big(\frac{2}{1-s}\big)^+\leq r \leq\big( \big(\frac{2}{1+s}\big)^+\big)'
	 \end{array}\right\}
	 & \text{if~} d =2
	   \\ &    \\
	   \inf  \left\{\| u \|_{L^{q'}_t L^{r'}_x} \; \begin{array}{|c}(q,r) \; \dHds-\text{admissible~with~ } \\ 
	 \frac{2}{1+2s} \leq q \leq\big(\frac{1}{s}\big)^-, \; \big(\frac{2}{1-s}\big)^+\leq r \leq\infty
	 \end{array}\right\}
	 & \text{if~} d =1.
	 \end{array}\right.
\end{align*}}

{\rmk Note that  $S(\Lt)= S(\dH^0)$ and $S'(\Lt)= S'(\dH^{-0}).$ In this dissertation, if $(q,r)$ is $\dH^{-0}$ admissible we say a pair $(q',r')$ is  \emph{$L^2-$dual admissible}.  }

Under the above definitions, the Strichartz estimates \eqref{eq:linear and inhomogeneous} become 
 \begin{align}
\| e^{it \Delta} \phi \|_{\SLt} \leq c \| \phi \|_{\Lt }
\quad\mbox{and}\quad
\Big\|\int_{s<t} e^{i(t-s) \Delta} f(s) ds \Big\|_{\SLt} \leq c \| f \|_{\SdLt} \label{eq:stri}
\end{align} 
and in this paper, we refer to them as \emph{the (standard) Strichartz} estimates. 

Combining \eqref{eq:stri}  with the  Sobolev embedding $W_x^{s,r}(\Rn) \hookrightarrow L_x^{\frac{nr}{n-sr}}(\Rn)$ for $s < \frac{n}{r}$ and interpolating yields \emph{the Sobolev Strichartz} estimates
 \begin{align}
\| e^{it \Delta} \phi \|_{\SHs} \leq c \| \phi \|_{\dHs }
\quad\mbox{and}\quad
\Big\|\int^t_0 e^{i(t-s) \Delta} f(s) ds\Big \|_{\SHs} \leq c \|\Ds  f \|_{\SdLt}  \label{eq:strisob},
\end{align} 
and in similar fashion \eqref{eq:retarded} leads to the \emph{Kato's Strichartz} estimate \cite{Kato87, Fos05}
\begin{align}
\label{eq:Kato-Strichartz}
\Big\|\int^t_0 e^{i(t-s) \Delta} f(s) ds\Big \|_{\SHs} \leq c \| f \|_{\SdHs}.
\end{align}  

Kato's Strichartz estimate  along with the  Sobolev embedding imply the inhomogeneous estimate (second estimate in \eqref{eq:strisob} but not vice versa) and it is the key estimate in the long term perturbation argument (Proposition \ref{longperturbation}).

\subsubsection{Besov Strichartz estimates}  We address the question of non-integer nonlinearities for the $\NLSf$. The following remark is due

{\rmk \label{re:nonlineal} 
The complex derivative of the nonlinearity $F(u)=|u|^{p-1}u$     
is $ F_z (z)=\frac{p+1}{2} |z|^{p-1}\;$ and $\;F_{\z} (z)=\frac{p-1}{2} |z|^{p-1}\frac{z}\z$. They are H\"older-continuous functions of order $p$, and  for any $u,v \in\C$, we have
\begin{align}
\label{eq:dif-int}
F(u)-F(v)=\int_0^1\Big[ F_z (v+t(u-v))(u-v)+ F_{\z} (v+t(u-v))\overline{(u-v)}\Big]dt,
\end{align}
thus,
\begin{align}
\label{eq:Holder-continuous}
|F(u)-F(v)|\lesssim|u-v|\bigl( |u|^{p-1}+|v|^{p-1}\bigr).
\end{align} 
Hence, the nonlinearity  $F(u)$ satisfies
\begin{enumerate}
\item[(a)] $F \in C^2(\C),$ if $2\leq d< 5,$ or  $d=5$ when $\frac{1}2<s<1$,
\item[(b)] $F \in C^1(\C),$ if $d\geq 6,$ or $d=5$ when  $0< s \leq \frac12$.   
\end{enumerate}

When estimating the fractional derivatives of  \eqref{eq:dif-int}, 
in the case (b), there is a lack of smoothness. This issue is resolved by using the Besov spaces.
}

Define the \emph{Besov Strichartz} space   $\dB^{\sigma}_{\SHs}=\dB^{\sigma}_{\SHs}(\IR)$
as the closure of all test functions under the semi-norm $\|\cdot\|_{\dB^{\sigma}_{\SHs}}$
with
{\small\begin{align*}
	\| u \|_{\dB^{\sigma}_{\SHs}} =  \left\{\begin{array}{ccc}
	 \sup  \left\{ \| u \|_{L^q_t \dB^{\sigma}_{r,2}}
	  \; \begin{array}{|c}(q,r) \; \dHs-\text{admissible~with~ } \\ 
	 \big(\frac{2}{1-s}\big)^+ \leq q \leq\infty, \quad\frac{2d}{d-2s}\leq r \leq\big(\frac{2d}{d-2}\big)^- 
	 \end{array}\right\}
	 & \text{if~} d \geq 3
	   \\ &    \\
	   \sup  \left\{ \| u \|_{L^q_t \dB^{\sigma}_{r,2}}
	    \; \begin{array}{|c}(q,r) \; \dHs-\text{admissible~with~ } \\ 
	  \big(\frac{2}{1-s}\big)^+ \leq q \leq\infty,\quad\frac{2}{1-s}\leq r \leq\big( \big(\frac{2}{1-s}\big)^+\big)'
	 \end{array}\right\}
	 & \text{if~} d =2.
	   \\ &    \\
	   \sup  \left\{\| u \|_{L^q_t \dB^{\sigma}_{r,2}} \; \begin{array}{|c}(q,r) -\dHs \; \text{admissible~with~ } \\ 
	  \frac{4}{1-2s} \leq q \leq\infty,\quad\frac{2}{1-2s}\leq r \leq\infty
	 \end{array}\right\}
	 & \text{if~} d =1.
	 \end{array}\right.
\end{align*}}

Similary, define the \emph{dual Besov Strichartz} space  $\dB^{\sigma}_{\SdHs}=\dB^{\sigma}_{\SdHs}(\IR)$ as the closure of all test functions under the semi-norm $\|\cdot\|_{\dB^{\sigma}_{\SdHs}}$ with

{\small
\begin{align*}
	\| u \|_{\dB^{\sigma}_{\SdHs}} =  \left\{\begin{array}{ccc}
	 \inf  \left\{\| u \|_{L^{q'}_t \dB^{\sigma}_{r',2}} 
	 \; \begin{array}{|c}(q,r) \; \dHds-\text{admissible~with~ } \\ 
	 \big(\frac{2}{1+s}\big)^+ \leq q \leq\big(\frac{1}{s}\big)^-,\; \big(\frac{2d}{d-2s}\big)^+\leq r \leq\big(\frac{2d}{d-2}\big)^-
	 \end{array}\right\}
	 & \text{if~} d \geq 3
	   \\ &    \\
	   \inf  \left\{\| u \|_{L^{q'}_t \dB^{\sigma}_{r',2}} 
	  \; \begin{array}{|c}(q,r) \; \dHds-\text{admissible~with~ } \\ 
	 \big(\frac{2}{1+s}\big)^+ \leq q \leq\big(\frac{1}{s}\big)^-, \; \big(\frac{2}{1-s}\big)^+\leq r \leq\big( \big(\frac{2}{1+s}\big)^+\big)'
	 \end{array}\right\}
	 & \text{if~} d =2
	 	   \\ &    \\
	   \inf  \left\{\| u \|_{L^{q'}_t \dB^{\sigma}_{r',2}}
	    \; \begin{array}{|c}(q,r) \; \dHds-\text{admissible~with~ } \\ 
	 \frac{2}{1+2s} \leq q \leq\big(\frac{1}{s}\big)^-, \; \big(\frac{2}{1-s}\big)^+\leq r \leq\infty
	 \end{array}\right\}
	 & \text{if~} d =1.
	 \end{array}\right.
\end{align*}}

\begin{lemma}\label{Besov-sigma}
If $u\in \dB^{\sigma}_{\SHs}$ and $\sigma\geq0, \; s \in \R,\;$ then
\begin{equation*}\label{eq:Besov-sigma}
\big\|D^{\sigma}u\big\|_{\SHs}
\lesssim
\|u\|_{\dB^{\sigma}_{\SHs}}.
\end{equation*}
\end{lemma}

\begin{proof}
Let $(q,r)$ be $\dHs-$admissible pair, then
\begin{align*}
\big\|D^{\sigma}u\big\|_{L_t^qL_x^r}
&\lesssim
 \bigg\|\Big(\sum_{N\in 2^{\Z}} \big |P_{N}D^{\sigma} u\big|^2\Big)^{\frac12}\bigg\|_{L_t^qL_x^r}
\lesssim
\Big\| \Big(\sum_{N\in 2^{\Z}} \|P_{N}D^{\sigma}u\|_{L_x^r}^2\Big)^{\frac12} \Big\|_{L_t^q}
 \\&
\approx
\Big\| \Big(\sum_{N\in 2^{\Z}} N^{2\sigma} \|P_{N}u\|_{L_x^r}^2\Big)^{\frac12}\Big\|_{L_t^q}
 \lesssim
\|u\|_{\dB^{\sigma}_{L_t^q\dB^{\sigma}_{r,2}}}.
\end{align*}
Taking $\sup$ over all $(q,r)\; \dHs-$admissible pairs, yields the claim.
\end{proof}

\begin{lemma}[Embedding] \label{interpolacion}For any compact time interval $I$, assume $ 0\leq\sigma<\rho$, $1\leq r, r_{_1}, q\leq \infty$. Then
\begin{align}
\label{eq:inter1}
\|D^{\sigma}u\|_{L^{q}_tL^{r}_x}\lesssim \|D^{\rho}u\|_{L^{q}_tL^{r_{_1}}_x},
\end{align}
where $r_{_1}=\frac{r d}{(\rho-\sigma) r+d}$  and $q_{_1}=q_{_2}$.
\end{lemma}
\begin{proof} The Sobolev embedding $\dot W_x^{\rho,r_{_1}}(\Rn) \hookrightarrow \dot W_x^{\sigma,r}(\Rn)$ yields the inequality \eqref{eq:inter1}.
 \end{proof}
{\rmk \label{remark lemma 5}If $q', r'$ and ${r'_{_1}}$ are the H\"older's conjugates of $r, q$ and $r_{_1}$, respectively, then  we have 
\begin{align*}
 \|D^{\rho}u\|_{L^{q'}_tL^{r'_{_1}}_x}\lesssim\|D^{\sigma}u\|_{L^{q'}_tL^{r'}_x} .
\end{align*}
 }

 \begin{lemma} [Linear Besov-Strichartz]\label{linear-Strichartz}
Let $u \in \dB^{\sigma}_{\SLt}$ be a solution to
the forced Schr\"odinger equation
\begin{equation}\label{eqs}
i u_t + \Delta u = \sum_{m=1}^M F_m
\end{equation}
for some functions $F_1 ,\dots,F_M$ and $\sigma=0$ or $\sigma=s$.  Then on $\IR$ we have
\begin{equation}\label{eq:linear-besov-strichartz}
\|u\|_{\dB^{\sigma}_{\SHs}}
\lesssim \|u_0\|_{\dH^{\sigma}} + \sum_{m=1}^M \| 
F_m \|_{\dB^{\sigma}_{\SdLt}}.
\end{equation}
\end{lemma}

 \begin{lemma}[Inhomogeneous Besov Strichartz estimate] If $F \in\dB^{\sigma}_{\SdHs},$ then
 \label{BeStr}
 \begin{equation}
\label{eq:In-Besov-Strichartz}
\Big\|\int^t_0 e^{i(t-\tau) \Delta} F(\tau) d\tau\Big \|_{\dB^{\sigma}_{\SHs}} \lesssim \| F\|_{\dB^{\sigma}_{\SdHs}}.
\end{equation} 
\end{lemma}

Proofs of Lemma \ref{linear-Strichartz} and \ref{BeStr} can be found in \cite{Tao06NLD}.

\begin{lemma}[Interpolation inequalities for Besov spaces \cite{Tri78}] \label{inter-besov}
Let $1\leq p_i, q_i\leq \infty$  and $u \in \beta^{\sigma_i}_{p_i,q_i}(\Rn),$ where $i=1,2,3$. Then
\begin{align*}
\|u\|_{\beta^{\sigma_1}_{p_{_1},q_{_1}}(\Rn)}=\|u\|_{\beta^{\sigma_2}_{p_{_2},q_{_2}}(\Rn)}^{1-\theta}\|u\|_{\beta^{\sigma_3}_{p_{_3},q_{_3}}(\Rn)}^\theta
\end{align*}
provided that 
\begin{align*}
\sigma_1=(1-\theta)\sigma_2+\theta\sigma_3,\qquad \frac1{p_{_1}}=\frac{1-\theta}{p_{_2}}+\frac{\theta}{p_{_3}}\qquad \mbox{and}\qquad
\frac1{q_{_1}}=\frac{1-\theta}{q_{_2}}+\frac{\theta}{q_{_3}}.
\end{align*}
\end{lemma}

\subsection{Local Theory} 
In this subsection the global existence and scattering in $H^1(\Rn)$ for small data in $\dHs$ (Propositions \ref{small data} and \ref{H^1 Scattering}), and a long  perturbation argument (Proposition \ref{longperturbation}) are examined. The proofs rely on  Besov spaces which allow us to treat the lack of smoothness of the nonlinearity $F(u)=|u|^{p-1}u$  
(see Remark \ref{re:nonlineal}). 

Although it may appear that the small data theory (Proposition \ref{small data}) is by now a straight forward argument, we write out its proof carefully to show how we deal with the non-integer nonlinearities. For the same reason we include full proofs of the long-term perturbation (Proposition \ref{longperturbation}) and the $H^1$ scattering (Proposition \ref{H^1 Scattering}).


\begin{prop}[Small data]  \label{small data}
Suppose $\|u_0\|_{\dHs}\lesssim A.$ There exists $\delta_{sd}=\delta_{sd}(A)>0$ such that if $\|e^{it\triangle}u_0\|_{\dB^0_{\SHs}}\lesssim \delta_{sd}$, then u(t) solving the $\NLSf$ is global in $\dHs(\Rn)$ and 
\begin{align*}
\|u\|_{\dB^{0}_{\SHs}}
\lesssim 2 \|e^{it \triangle}u_0\|_{\dB^{0}_{\SHs}},\qquad
\|u\|_{\dB^{s}_{\SLt}}
\lesssim 2c \|u_0\|_{\dHs}.
\end{align*}
\end{prop}

\begin{proof}
Using a fixed point argument in a ball $B$, 
 the existence of solutions to \eqref{eq:NLS} and continuous dependence on the initial data is proven as follows.

Let 
\begin{align*} 
B=\Big\{\|u\|_{\dB^{0}_{\SHs}}
\lesssim 2 \|e^{it \triangle}u_0\|_{\dB^{0}_{\SHs}},\quad
\|u\|_{\dB^{s}_{\SLt}}
\lesssim 2 c\|u_0\|_{\dHs}
\Big\}.
\end{align*}
Assume $F(u)=|u|^{p-1}u$ and  the map $u\mapsto \Phi_{u_0}(u)$  defined via
\begin{equation*}
\Phi_{u_0}(u):=e^{it\triangle}u_0+i \int_0^t e^{i(t-\tau)\triangle}F(u(\tau))d\tau.
\end{equation*}
Combining the triangle inequality and the Linear Besov Strichartz estimates \eqref{eq:linear-besov-strichartz} and the fact that $F(u) \in C^1$, we obtain
\begin{align*}
\|\Phi_{u_0}(u)\|_{\dB^{0}_{\SHs}} &\lesssim \|e^{it\triangle}u_0\|_{\dB^{0}_{\SHs}} +  \| F(u)\|_{\dB^{s}_{\SdLt}},\\
\|\Phi_{u_0}(u)\|_{\dB^{s}_{\SLt}} &
\lesssim \|u_0\|_{\dB^{s}_{\SLt}} +\| F(u)\|_{\dB^{s}_{\SdLt}}. 
\end{align*}
For each dyadic number $N\in 2^{\Z},\;$ the fractional chain rule (Lemma \ref{chain}) and H\"older's inequality lead to
\begin{align*}
\| \Ds F(u)\|_{\SdLt}& 
\lesssim \| \Ds (|u|^{p-1}u)\|_{\qsd \rsd}
\\&\lesssim \|u\|^{p-1}_{\qs \rsh} \|\Ds u\|_{\qs \rs} \lesssim \|u\|^{p-1}_{S(\dH^s)} \|\Ds u\|_{\SLt},
\end{align*}
thus, Littlewood-Paley theory yields 
\begin{align}\label{small besov}
\| |u|^{p-1}u\|_{\dB^{s}_{\SdLt}} 
  \lesssim \|u\|^{p-1}_{\dB^{0}_{S(\dH^s)}} \|u\|_{\dB^{s}_{\SLt}}.
\end{align}
Therefore,
\begin{align*}
\|\Phi_{u_0}(u)\|_{\dB^{0}_{\SHs}} &\lesssim \|e^{it\triangle}u_0\|_{\dB^{0}_{\SHs}}+
  \|u\|^{p-1}_{\dB^{0}_{S(\dH^s)}} \|u\|_{\dB^{s}_{\SLt}},\\
\|\Phi_{u_0}(u)\|_{\dB^{s}_{\SLt}} & 
\lesssim \|u_0\|_{\dB^{s}_{\SLt}} +
 \|u\|^{p-1}_{\dB^{0}_{S(\dH^s)}} \|u\|_{\dB^{s}_{\SLt}}
\end{align*}
and if $\|e^{it\triangle}u_0\|_{\dB^{0}_{\SHs}}\leq\delta_1$ with $\delta_1 = \min\Bigl\{\frac{1}{2^pc_1^{p-1}A^{p-2}},\sqrt[p-1]{\frac{1}{2^pc_2^{p-1}A}}\Bigr\}$ leads to $\Phi_{u_0}(u)\in B$. 

To complete the proof, we need to show that the map $u\mapsto \Phi_{u_0}(u)$ is a contraction. Take  $u,v \in B$,  and note that the triangle inequality and Besov Strichartz estimates  yield
\begin{align*}
\|\Phi_{u_0}(u)-\Phi_{u_0}(v)\|_{\dB^0_{\SHs}} 
&\lesssim
\Big\| \int_0^t e^{i(t-\tau)\triangle}\Big(F\big(u(\tau)\big)-F\big(v(\tau)\big)\Big)d\tau\Big\|_{\dB^0_{\SHs}}\\
&\lesssim
 \| \Ds \big(F(u)-F(v)\big)\|_{\dB^0_{\SdLt}}\approx  \| F(u)-F(v)\|_{\dB^s_{\SdLt}},
\end{align*}
and
\begin{align*}
\|\Phi_{u_0}(u)-\Phi_{u_0}(v)\|_{\dB^s_{\SLt}} &\approx
\|\Ds (\Phi_{u_0}(u)-\Phi_{u_0}(v))\|_{\dB^0_{\SLt}}\\ 
&\lesssim
\Big\| \int_0^t e^{i(t-\tau)\triangle}\Ds \Big(F\big(u(\tau)\big)-F\big(v(\tau)\big)\Big)d\tau\Big\|_{\dB^0_{\SLt}}  \\
&\lesssim
 \| \Ds \big(F(u)-F(v)\big)\|_{\dB^0_{\SdLt}}\approx  \| F(u)-F(v)\|_{\dB^s_{\SdLt}}.
\end{align*}
For each dyadic number $N\in 2^{\Z},\;$ we estimate $ \| \Ds \big(F(u)-F(v)\big)\|_{\SdLt}.\;$ Recall that we are considering the mass-supercritical energy-subcritical NLS, i.e., $0<s<1$ and $p=1+\frac {4}{d-2s}.$ Due to the lack of smoothness of the nonlinearity (Remark \ref{re:nonlineal}), we consider two (complementary) cases:
 
\noindent{\bf (a)} The function $F(u)$ is \textbf{ at least} in  $C^2(\C)$.

\noindent{\bf (b)} The nonlinearity $F(u)$ is \textbf{ at most} in $C^1(\C)$. 
 
In the rest of the proof we examine these cases separately. After the proof we refer to the specific examples to illustrate our approach. 

\noindent {\bf Case  (\emph{a})}.   
$F(u)$ is \textbf{ at least} in  $C^2(\C)$: this case occurs when $1\leq d\leq 4+2s,$ i.e.,  dimensions 2, 3, and 4  for $0<s<1$, or  dimension 5 when $\frac12\leq s<1$.  
Combining \eqref{eq:Holder-continuous}, chain rule  (Lemma \ref{chain})  and H\"older's inequality, gives
\begin{align*}
\|\Ds \big(F(u)-&F(v)\big)\|_{\SdLt} \notag
\lesssim \| \Ds (|u|^{p-1}u-|v|^{p-1}v)\|_{\qsd \rsd}
\\&\lesssim  \|\Ds |u-v|\|_{\qs \rs}
\Big(\|u\|^{p-1}_{\qs \rsh}
+\|v\|^{p-1}_{\qs \rsh}
\Big)\notag\\
& \lesssim  \|\Ds |u-v|\|_{\SLt}\Big(\|u\|^{p-1}_{S(\dH^s)}+\|v\|^{p-1}_{S(\dH^s)}\Big).
\end{align*}
Here, we used  the H\"older split 
\begin{equation}\label{split1}
\rsdr=\frac{d^2p-8s}{2d^2p}+(p-1)\frac{2(d+4)}{d^2p(p-1)}
\end{equation} 
together with the fact that the pair $\left(\qsdq, \rsdr \right)$ is $\Lt-$ dual admissible, the pair $\left(\qsq, \rsr\right)$ is $\Lt-$admissible and the pair $\left(\qsq, \rshr\right)$  is $\dHs-$ admissible.

Therefore, $\| F(u)-F(v)\|_{\dB^s_{\SdLt}}\lesssim  \| u-v\|_{\dB^s_{\SLt}}\Big(\|u\|^{p-1}_{\dB^0_{S(\dH^s)}}+\|v\|^{p-1}_{\dB^0_{S(\dH^s)}}\Big).$
If $\|e^{it\triangle}u_0\|_{\dB^{0}_{\SHs}}\leq\delta_2$ with  
$
\delta_2 =  \min\Big\{  \sqrt[p-1]{ \frac{1}{2^{p}C}},\frac{1}{2^{p}  A^p-2C}  \Big\}
$
implies that $\Phi_{u_0}$ is a contraction.
\bigskip

\noindent {\bf Case (\emph{b})}.  
 $F(u)$ is \textbf{ at most} in $C^1(\C)$:  this corresponds to dimensions higher than $4+2s$, i.e.,  $d=5$ with $0<s< \frac12$ or $d\geq 6$ with $0<s<1$. 
Let $w=u-v$,  therefore \eqref{eq:dif-int} and the triangle inequality imply 
\begin{align}\label{eq:difestimate2}
\|\Ds \big(&F(u)-F(v)\big)\|_{\SdLt} \notag
\lesssim \| \Ds (|u|^{p-1}u-|v|^{p-1}v)\|_{\qsd \rsd}
\\&\lesssim  \|\Ds F_{z}(v+w)w\|_{\qsd \rsd}+\|\Ds F_{\z}(v+\w)\w\|_{\qsd \rsd}.
\end{align}
To  estimate \eqref{eq:difestimate2}, we consider the subcases (i) $s\leq p-1$ and (ii) $s>p-1$.

\noindent (i) 
If dimensions $4+2s<d\leq  \frac{4+2s^2}{s}$, then $s\leq p-1<1$, thus,

\begin{align}
\|\Ds & F_{z}(u)w\|_{\qsd \rsd}\lesssim 
 \|\Dsa F_{z}(u)w\|_{\qsd\rua} \label{1}
 \\
\lesssim& \|\Dsa F_{z}(u)\|_{\qda\rda}
\|w\|_{\qs\rta} \label{2}\\
&+\| u\|^{p-1}_{\qca\rca}
\|\Dsa w\|_{\qcia \rcia} \label{2l}
\\
\lesssim& \|u\|^{\expa}_{\qs \rsh}\|\Ds u\|^{\expa}_{\qs \rs}\|\Ds w\|_{\qs \rs}\label{3}\\
&+\| u\|^{p-1}_{\qca\rca}
\|\Ds w\|_{\qcia \roa}\label{4}\\
\lesssim& \|\Ds w\|_{\SLt}\Big(\|u\|^{\expa}_{\SHs}\|\Ds u\|^{\expa}_{\SLt}
+\| u\|^{p-1}_{\SHs}\Big)\notag,
\end{align}
where,  Remark \ref{remark lemma 5} yields \eqref{1}, since  $\ruar$ and $\rsdr$ are H\"older conjugates and $\frac{s(p-1)}{2}<s$.  Leibniz rule gives \eqref{2} and \eqref{2l}. Then applying  chain rule for H\"older-continuous functions (Lemma \ref{Holderfd}) with $\rho:=p-1,\; \sigma:=\frac{s(p-1)}{2}$  and $\nu:=s$ to \eqref{2},  we obtain  \eqref{3}. Noticing that $\roa\hookrightarrow \rcia,\;$  Lemma \ref{interpolacion} implies \eqref{4}.  The last line comes from the fact that the pairs  
$\left (\qsq, \rshr\right),\; \left(\qcaq,\rcar\right)$ are $\dHs-$admissible, and the pairs $\left(\qsq, \rsr\right)$, $\left({\qciaq, \roar}\right)$  are $\Lt-$admissible.  In a similar fashion, we obtain the estimate for the conjugate
\begin{align*}
\|\Ds F_{\z}(v+\w)&\w\|_{\qsd \rsd} 
\lesssim\|\Ds w\|_{\SLt}\big(\|u\|^{\expa}_{\SHs}\|\Ds u\|^{\expa}_{\SLt}
+\| u\|^{p-1}_{\SHs}\big).
\end{align*}
Thus, Littlewood-Paley theory implies that 
$$\| F(u)-F(v)\|_{\dB^s_{\SdLt}}\lesssim 2  \| u-v\|_{\dB^s_{\SLt}}\Big(\|u\|^{\expa}_{\dB^0_{S(\dH^s)}}\|u\|^{\expa}_{\dB^s_{\SLt}}+\|u\|^{p-1}_{\dB^0_{S(\dH^s)}}\Big),$$
and  letting
$\delta_3 \leq  \sqrt[\frac{p-1}2]{ \frac{1}{2^{(p+2)}CA^{\frac{p-1}2}}}$
gives that $\Phi_{u_0}$ is a contraction.

\noindent(ii) 
If the dimensions $d>  \frac{4+2s^2}{s},$ then $ p-1<s.$  Therefore, we make an estimate for  $ \|\Ds F_{z}(u) w\|_{\qsd \rsd}$, as follows

\begin{align}
\|\Ds  F_{z}&(u)w\|_{\qsd \rsd}\lesssim 
 \|\Dsb F_{z}(u)w\|_{\qsd\rub} \label{1b}
 \\
\lesssim& \|\Dsb F_{z}(u)\|_{\qdb\rdb}
\|w\|_{\qs\rtb} \label{2b}\\
&+\| u\|^{p-1}_{\qcb\rcb}
\|\Dsb w\|_{\qcib \rcib} \label{2bl}
 \\
\lesssim &\|u\|^{\expbu}_{\qs \rsh}
\|\Ds u\|^{\expbd}_{\qs \rs}
\|\Ds w\|_{\qs \rs}\label{3b}
\\
&+\| u\|^{p-1}_{\qcb\rcb}
\|\Ds w\|_{\qcib \rob}\label{4b}\\
\lesssim &\|\Ds w\|_{\SLt}\Big(\|u\|^{\expbu}_{\SHs}\|\Ds u\|^{\expbd}_{\SLt}
+\| u\|^{p-1}_{\SHs}\Big) \notag,
\end{align}
as before in (i), Remark \ref{remark lemma 5} yields \eqref{1b}, since  $\rubr$ and $\rsdr$ are H\"older conjugates and ${(p-1)}^{2}<s$.  
Leibniz rule gives \eqref{2b} and \eqref{2bl}. To obtain \eqref{3b}, we use  the chain rule for H\"older-continuous functions (Lemma \ref{Holderfd}) with $\rho:= (p-1)^2$ and $\nu:=s\;$ in \eqref{2b}. The line \eqref{4b} follows from Lemma \ref{interpolacion},   
 and finally, since the pairs 
$\left (\qsq, \rshr\right)$, $\left(\qcaq,\rcar\right)$ are $\dHs-$admissible, and the pairs $\left(\qsq, \rsr\right), \left({\qciaq, \roar}\right)$  are $\Lt-$admissible, we obtain the last estimate.  Similarly, 
\begin{align*}
\|\Ds F_{\z}(v+\w)&\w\|_{\qsd \rsd} 
\lesssim \|\Ds w\|_{\SLt} \big(\|u\|^{\expbu}_{\SHs}\|\Ds u\|^{\expbd}_{\SLt}
+\| u\|^{p-1}_{\SHs}\big).
\end{align*}

Therefore, Littlewood-Paley theory produces
$$\| F(u)-F(v)\|_{\dB^s_{\SdLt}}\lesssim 2  \| u-v\|_{\dB^s_{\SLt}}\Big(\|u\|^{\expbu}_{\dB^0_{S(\dH^s)}}\|u\|^{\expbd}_{\dB^s_{\SLt}}+\|u\|^{p-1}_{\dB^0_{S(\dH^s)}}\Big),$$
and taking $\delta_4 \leq  \sqrt[\frac{(p-1)(1+s-p)}s]{ \frac{1}{2^{(p+1)}CA^{\frac{(p-1)^2}s}}}$
implies that $\Phi_{u_0}$ is a contraction.

From cases (a) and (b) choosing   
$
\delta_{sd}\leq  \min\big\{\delta_{1},\delta_{2},\delta_{3},\delta_{4}  \big\}
$ implies that the map $u\mapsto \Phi_{u_0}(u)$ is a contraction which concludes the proof.
\end{proof}
To better understand the difference for the above cases (a), (b)(i) and (b)(ii), we refer to the reader to   \cite[Examples 2.14, 2.15 and 2.16]{guevara} were we give examples of  $\dH^{\frac12}-$critical$\NLS_{\frac73}(\R^4)$, $\NLS_{\frac53}(\R^7)$ and  $\NLS_{\frac{13}9}(\R^{10})$ and demonstrate how the estimates work.


\begin{prop}[Long term perturbation]\label{longperturbation}
For each $A>0$, there exist $\epsilon_0=\epsilon_0(A)>0$ and $c=c(A)>0$ such that 
the following holds.
Let  $u=u(x,t) \in H^1(\Rn)$  
solve 
$\; \NLSf.$
 Let $v=v(x,t) \in H^1(\Rn)$ for all $t$ and satisfy $\te=i v_t + \Delta v +|v|^{ p-1} v.$  

\noindent If 
$
\| v \|_{\dB^{0}_{\SHs}} \leq A$, $ \|\te \|_{\dB^{0}_{\SdHs}} \leq \epsilon_0
$
and
$
\| e^{i (t-t_0) \Delta} (u(t_0)-v(t_0)) \|_{\dB^{0}_{\SHs}} \leq \epsilon_0,
$
then $
\| u \|_{\dB^{0}_{\SHs}} \leq c.
$
\end{prop}

\begin{proof}
Let $F(u)= |u|^{p-1}u,\;
w=u - v$, and $
W(v,w) =F(u) -F(v) =F(v+w)-F(v)$. Therefore, $w$ solves the equation
$$
i w_t + \Delta w +W(v,w)  +\te=0.
$$

Since $\| v \| _{\dB^{0}_{S(\dHs)}} \leq A$, split the interval $[t_0 , \infty)$ into $K=K_A$ intervals $I_j = [t_j,t_{j+1}]$ such that for each $j$, 
$\| v \|_{\dB^{0}_{S(\dHs, I_j)}} \leq \delta$ with $\delta$ to be chosen later. Recall that the integral equation of $w$ at time $t_j$ is given by
\begin{equation}
\label{DuhamelsInterval}
w(t)= e^{i (t-t_j) \Delta}w(t_j) + i \int_{t_j}^t  e^{i (t-\tau) \Delta} (W+\te)(\tau) d\tau.
\end{equation}
Applying Kato Besov Strichartz estimate \eqref{eq:In-Besov-Strichartz} on \eqref{DuhamelsInterval} for each $I_j$, we obtain
\begin{align*}
\| w \|_{\dB^{0}_{\SHsj}}& \lesssim \| e^{i (t-t_j) \Delta}w(t_j) \|_{\dB^{0}_{\SHsj}}+
\| \int_{t_j}^t  e^{i (t-\tau) \Delta} (W+\te)(\tau) d\tau \|_{\dB^{0}_{\SHsj}} \\
 &  \lesssim \| e^{i (t-t_j) \Delta}w(t_j) \|_{\dB^{0}_{\SHsj}}
 +c \| W(v,w)\| _{\dB^{0}_{\SHdsj}}
 +c \| \te\| _{\dB^{0}_{\SHdsj}}\\
 &  \lesssim \| e^{i (t-t_j) \Delta}w(t_j) \|_{\SHsj}+c \| W(v,w)\| _{\dB^{0}_{\SHdsj}} 
 +c \epsilon_0.
\end{align*}
Thus, for each dyadic number $N\in 2^{\Z},\;$  the following estimate  holds
\begin{align}
\| W(v,w)&\| _{\SHdsj}  
\lesssim\| F(v+w) -F(v)\| _{\qzh\rzh} \notag\\
&\lesssim 
 \| w \|_{\qdh\rdh} \Big(\| v\|^{p-1}_{\quh\ruh}
+\| w \|^{p-1}_{\quh\ruh}\Big)\label{LP-1}\\
&\lesssim 
 \| w \|_{\SHsj} \Big(\| v\|^{p-1}_{\SHsj}+\| w \|^{p-1}_{\SHsj}\Big)\notag\\
 & \leq \| w \|_{\SHsj} \Big( \delta_N^{p-1}+\| w \|^{p-1}_{\SHsj}\Big),\label{LP-3}
\end{align}
where we first observed that the pairs $(\frac{6}{1-s},\frac{6 d}{3 d-4s-2})$, $(\frac{4}{1- s},\frac{2 d}{d -s-1})$ are $\Hs-$admissible;  the pair $(\frac{12 (d - 2 s)}{(8 + 3 d - 6 s) (1 - s)},\frac{6 d (d - 2 s)}{  3 (d^2+ 2s^2)+ 9 d (1 - s) - 2(5 s + 4)})$ is $\dH^{-s}-$admissible. Thus, we used \eqref{eq:Holder-continuous}  and H\"older's inequality to obtain \eqref{LP-1}. Since $\| v \|_{\dB^{0}_{S(\dHs, I_j)}} \leq \delta$ for each dyadic interval, there exists $\delta_N=\delta(N)$, so we obtain \eqref{LP-3}.
Therefore, 
\begin{align*}
\| F(v+w) -F(v)\|_{\dB^{0}_{\SHdsj}}
&\lesssim 
 \| w \|_{\dB^{0}_{\SHsj}} \Big(\| v\|^{p-1}_{\dB^{0}_{\SHsj}}+\| w \|^{p-1}_{\dB^{0}_{\SHsj}}\Big)\\
 & \leq \| w \|_{\dB^{0}_{\SHsj}}\Big( \delta^{p-1}+\| w \|^{p-1}_{\dB^{0}_{\SHsj}}\Big).
\end{align*}
Choosing 
$ \delta=\sum_{N\in2^{\Z}}\delta_N < \min\Big\{ 1, \frac{1}{
{4 c_1}} \Big\}$ and $\|  e^{i (t-t_j) \Delta} w(t_j) \|_{\dB^{0}_{\SHsj}}+ c_1 \epsilon_0  \leq \min \Big\{1, \frac{1}{2\sqrt[p]{4 c_1}}\Big\}
$, we have
$$
\| w \|_{\dB^{0}_{\SHsj}} \leq 2 \|  e^{i (t-t_j) \Delta} w(t_j) \|_{\dB^{0}_{\SHsj}}+ 2 c_1 \epsilon_0.
$$
Taking  $t=t_{j+1}$, applying $e^{i (t-t_{j+1}) \Delta } $ to both sides of (\ref{DuhamelsInterval}) and repeating the Kato estimates  \eqref{eq:Kato-Strichartz} , we obtain 
\begin{align}
\| e^{i (t-t_{j+1}) \Delta }w(t_{j+1}) \|_{\dB^{0}_{S( \dHs)}} 
\nonumber & \leq 2 \|  e^{i (t-t_j) \Delta} w(t_j) \|_{\dB^{0}_{\SHsj}}+ 2 c_1 \epsilon_0.
\end{align}
Iterating this process until $j=0$, we obtain
\begin{align*}
\| e^{i (t-t_{j+1}) \Delta }w(t_{j+1}) \|_{\dB^{0}_{S( \dHs)}} & \leq 2^j \| e^{i (t-t_{0}) \Delta }w(t_{0}) \|_{\dB^{0}_{\SHsj}}+(2^j-1) 2 c_1\epsilon_0
 \leq 2^{j+2} c_1 \epsilon_0.
 \end{align*}
 These estimates must hold for all intervals $I_j$ for $0 \leq j \leq K-1$, therefore,
 $$
 2^{K+2} c_1 \epsilon_0 \leq \min\Big\{1, \frac{1}{2\sqrt[p]{4 c_1}}\Big\},
$$ 
which determines how small $\epsilon_0$ has to be taken in terms of $K$ (as well as, in terms of $A$).
\end{proof}

An illustration of  specific cases (the nonlinearity $F(u)$ is (a) at least in  $C^2(\C)$
and (b) at most in $C^1(\C)$) of the estimate $\| W(v,w)\| _{\SHdsj}$ is given in \cite[Examples 2.18, 2.19  and 2.20 ]{guevara}.
\begin{prop}[$H^1$ scattering]\label{H^1 Scattering}
Assume $u_0 \in H^1(\Rn)$. Let  $u(t)$ be a global solution to $\NLSf$ 
with the initial condition $u_0$, globally finite $\dHs$ Besov Strichartz norm $\|u\|_{\dB^{0}_{\SHs}} <+\infty$ and uniformly bounded $H^1(\Rn)$ norm $\sup_{t\in[0,+\infty)}\|u(t)\|_{H^1}\leq B$. Then there exists $\phi_+\in H^1(\Rn)$ such that \eqref{eq:scatter} holds, 
 i.e., $u(t)$ scatters in $H^1(\Rn)$ as $t\to +\infty$. Similar statement holds for negative time.
\end{prop}

\begin{proof}
Suppose $u(t)$ solves  $\NLSf$  
with the initial datum $u_0$, and satisfies the integral equation 
\begin{align}
\label{eq:Duhamel}
u(t) = e^{it\lap}u_0+ i\mu \int_0^t  e^{i(t-\tau)\lap}\left( |u|^{p-1}u(\tau)\right) d\tau.
\end{align}

The assumption  $\|u\|_{\dB^{0}_{\SHs}}<+\infty$ implies that for each dyadic $N\in 2^{\Z}$ there exists $M=\|u\|_{\qs \rsh}<\infty$ and let $\tM\sim M^{\frac{np}{2s}}$. Decompose $[0,+\infty)=\cup_{j=1}^{\tM}I_j$, such that  for each $j$, $\|u\|_{\qsI \rsh}<\delta$. Hence, the triangle inequality and Strichartz estimates  yield
\begin{align*}
\|u\|_{\SLt} &\lesssim \|e^{it\triangle}u_0\|_{\SLt} +  \| F(u)\|_{\SdLt},\\
\|\nabla u\|_{\SLt} &\lesssim \|e^{it\triangle}\nabla u_0\|_{\SLt} +  \| \nabla F(u)\|_{\SdLt}.
\end{align*}

Therefore,  the integral equation \eqref{eq:Duhamel}  
 on $I_j$, combined with the above inequalities, leads to 
\begin{align}
\|\nabla u\|_{S(L^2;I_j)}
&\lesssim B+\big\||u|^{p-1}\nabla u\big\|_{S'(L^2;I_j)} 
\lesssim B+\big\||u|^{p-1}\nabla u\big\|_{\qsdI \rsd} \label{estper1a} 
\\&\lesssim B+ \|u\|^{p-1}_{\qsI \rsh} \|\nabla u\|_{\qsI \rs}\label{estper1b} 
\\& \lesssim B+ \delta^{p-1} \|\nabla u\|_{S(\Lt;I_j)}.\label{estper1}
\end{align}
The pairs $\big(\frac {d}{2 s},\frac{d^2p(p-1)}{2(d+4)}\big)$ and $\big(\frac {d}{2 s},\frac {2d^2 p}{d^2p-8s}\big)$ are $\Lt-$admissible and the pair $\big(\frac {d}{2 s},\frac{2 d^2(p-1)}{d^2 (p - 1) + 16}\big)$ is $\Lt-$ dual admissible;  we obtain \eqref{estper1b}  applying  H\"older's inequality  to  \eqref{estper1a}. Similarly, by dropping the gradient, it follows
\begin{align}\label{estper2}
\|u\|_{S(L^2;I_j)}
 \lesssim B+ \delta^{p-1} \| u\|_{S(\Lt;I_j)}.
\end{align}

Combining \eqref{estper1} and \eqref{estper2} and using the fact that $\delta$ can be chosen appropiately small, gives that
$
\|(1+|\nabla|) u\|_{S(L^2;I_j)}
 \lesssim 2 B. \;
$ Summing over the $\tM$ intervals, leads to
 $$\|(1+|\nabla|) u\|_{S(L^2)}\lesssim  BM^{\frac{np}{2s}}.$$

Define the wave operator 
\begin{align*}
\phi_+=u_0+i\int^{+\infty}_0 e^{-i\tau\Delta}F(u(\tau))d\tau,
\end{align*}
note that $\phi_+\in H^1$, thus, Strichartz estimates and the hypothesis  lead to
\begin{align}
\|\phi_+\|_{H^1} 
&\lesssim\|u_0\|_{H^1}+\big\||u|^{p-1}\nabla u\big\|_{S'(L^2)}\notag
 \lesssim\|u_0\|_{H^1}+\big\||u|^{p-1}\nabla u\big\|_{\qsd \rsd}\notag
\\&\lesssim\|u_0\|_{H^1}+ \|u\|^{p-1}_{\qs \rsh} \|\nabla u\|_{\qs \rs}
\lesssim B+ BM^{\frac{p(d+2s)-2s}{2s}}.\label{scat:esti3}
\end{align}
Additionally,
\begin{align}
u(t)-e^{it\Delta}\phi^+=-i\int^{+\infty}_t e^{i(t-\tau)\Delta}F(u(\tau))d\tau \label{scatterH1}.
\end{align}
Therefore, estimating the $\Lt$ norm of \eqref{scatterH1}, Strichartz estimates and H\"older's inequality give
\begin{align}
\|u(t)&-e^{it\Delta}\phi_+\|_{\Lt}\notag
\lesssim \Big\|\int^{+\infty}_t e^{i(t-\tau)\Delta}F(u(\tau))d\tau\Big\|_{\SLt}\\
&\lesssim\big\|F(u(\tau))\big\|_{S'(L^2;[t,+\infty)}
\lesssim\big\||u|^{p-1}\nabla u\big\|_{\qsd \rsd}\label{scat:esti1},
\end{align}
and simillary, estimating the $\dH^1$ norm of  \eqref{scatterH1}, we obtain
\begin{align}
\|\nabla(u(t)&-e^{it\Delta}\phi_+)\|_{\Lt}
\lesssim \Big\|\int^{+\infty}_t e^{i(t-\tau)\Delta}F(u(\tau))d\tau\Big\|_{\SLt}\notag\\
&\lesssim\big\|F(u(\tau))\big\|_{S'(L^2;[t,+\infty))}
\lesssim\big\||u|^{p-1}\nabla u\big\|_{\qsdtail \rsd}.\label{scat:esti2}
\end{align}
Using the Leibniz rule (Lemma \ref{leibniz}) to estimate \eqref{scat:esti1} and \eqref{scat:esti2}, yields
\begin{align*}
\big\||u|^{p-1}\nabla u\big\|_{\qsdtail \rsd}
&\lesssim \|u\|^{p-1}_{\qstail \rsh} \|\nabla u\|_{\qstail \rs}.
\end{align*}
By \eqref{scat:esti3} the term $\|u\|^{p-1}_{\qstail \rsh} \|\nabla u\|_{\qstail \rs}$  is bounded. Then as $t\to\infty$ the term $\|u\|_{\qstail \rsh} \to 0$, thus, summing over all dyadic $N$,  \eqref{eq:scatter}  is obtained.
\end{proof}

\subsection{Properties of the Ground State} \label{SEandGS}
Weinstein \cite{We82} proved the Gagliardo-Nierberg inequality
\begin{equation}
\|u\|^{p+1}_{L^{p+1}}\leq C_{GN} \|\nabla u\|^{\frac{d(p-1)}{2}}_{L^{2}}\| u\|^{2-\frac{(d-2)(p-1)}{2}}_{L^{2}}\label{G-N}
\end{equation}
with the sharp constant 
\begin{equation}\label{const GN}
C_{GN}=\frac{p+1}{2\|Q\|^{p-1}_{L^2}},
\end{equation}
  where $Q$ is as in \eqref{eq:Qground}. 
  
This inequality \eqref{G-N} is optimized by $Q$, i.e.,
$ 
\|Q\|^{p+1}_{L^{p+1}}=\frac{p+1}{2}\| \nabla Q\|^{\frac{d(p-1)}{2}}_{L^2}\| Q\|^{2-\frac{d(p-1)}{2}}_{L^2}.
$ 
 
Multiplying (\ref{eq:Qground}) by $Q$ and integrating, gives
 $\|Q\|^{p+1}_{L^{p+1}}=\alpha^2\|\nabla Q\|^2_{L^2}+\beta\|Q\|^2_{L^2},$
thus, Pohozhaev identities yield
 ${\|\nabla Q\|_{L^2}}={\| Q\|_{L^2}},$ and, ${\|\ Q\|^{p+1}_{L^{p+1}}}=\frac{p+1}{2}{\| Q\|^2_{L^2}}.$ 
 
 In addition,
\begin{align}
\|\uQ\|^2_{L^2}=\alpha^{-d}\|Q\|^2_{L^2},\quad\|\nabla \uQ\|^2_{L^2}=\alpha^{2-d}\|\nabla Q\|^2_{L^2} , \quad\mbox{and} \quad \|\uQ\|^{p+1}_{L^{p+1}}=\alpha^{-d}\|Q\|^{p+1}_{L^{p+1}} ,
\label{poz3}
\end{align}
therefore, the scale invariant quantity becomes
\begin{eqnarray}\label{ground modify}
\|\uQ\|_{L^{2}}^{1-s}\|\nabla \uQ\|_{L^{2}}^{s}
=\alpha^{-\frac 2{p-1}}\|Q\|_{L^{2}},
\end{eqnarray}
and the mass-energy scale invariant quantity is 
\begin{eqnarray}\label{me Q1}
M[\uQ]^{1-s}E[\uQ]^{s}&=&\big(\alpha^{-d}\|Q\|^2_{L^2}\big)^{1-s}\bigg(\frac{\alpha^{2-d}}{2} \|\nabla Q\|^2_{L^2}-\frac{\alpha^{-d}}{p+1}\|Q\|^{p+1}_{L^{p+1}}\bigg)^{s}\\
&=&\frac{\alpha^{-d}}{2^s}\left(\frac{(p-1)s}{2}\right)^s\|Q\|^2_{L^{2}}\label{me Q2}\\
&=&\bigg(\frac{s}{d}\bigg)^{s} \big(\|\uQ\|_{L^{2}}^{1-s}\|\nabla \uQ\|_{L^{2}}^{s}\big)^2 
\label{mass modify},
\end{eqnarray}
since the  energy definition yields \eqref{me Q1}, Pohozhaev identities \eqref{poz3} and \eqref{ground modify} imply \eqref{me Q2} and \eqref{mass modify}. 

Notice that 
\begin{eqnarray*}
M[u]^{1-s}E[u]^{s}&=&(\|u\|^2_{L^2})^{1-s}\bigg(\frac12 \|\nabla u\|^2_{L^2}-\frac1{p+1} \|\ u\|^{p+1}_{L^{p+1}} \bigg)^{s}\\
&\geq&(\|u\|^{1-s}_{L^2}\|\nabla u\|^{s}_{L^2})^2\bigg(\frac12 -\frac{C_{GN}}{p+1} \big( \|u\|^{1-s}_{L^2}\|\nabla u\|^{s}_{L^2}\big)^{p-1} \bigg)^{s}\\
&\geq&\frac1{2^{s}}(\|u\|^{1-s}_{L^2}\|\nabla u\|^{s}_{L^2})^2\bigg(1 -\alpha^{-2}\bigg(\dfrac{ \|u\|^{1-s}_{L^2}\|\nabla u\|^{s}_{L^2}}{ \|\uQ\|^{1-s}_{L^2}\|\nabla \uQ\|^{s}_{L^2}} \bigg)^{p-1} \bigg)^{s},
\end{eqnarray*}
therefore,
\begin{eqnarray}
\label{thrME-G}
\frac{d}{2s} \left[\g_u(t )\right]^{\frac2s}\bigg(1 -\frac{\left[\g_u(t )\right]^{p-1}}{\alpha^{2}} \bigg)
\leq\left(\ME[u]\right)^{\frac1s}\leq\frac{d}{2s}\left[\g_u(t)\right]^{\frac2s}.
\end{eqnarray}

Summarizing,  the upper bound in \eqref{thrME-G} is obtained by bounding the energy $E[u]$ above by the kinetic energy part, and the lower bound is achieved using the definition of energy and the sharp Gagliardo-Nirenberg inequality \eqref{G-N} to bound the potential term.

\subsection{Properties of the Momentum}\label{momentum}
 
Let $u$ be a solution of $\NLSf$ and assume that $P[u]\neq 0$. Let $\xi_0 \in \Rn$ to be chosen later and $w$ be the Galilean transformation of $u$ 
\begin{align*}
 w(x,t)=e^{ix\cdot\xi_0}e^{-it|\xi_0|^2}u(x-2\xi_0t,t).
\end{align*}
Then 
\begin{align*}
	\nabla w(x,t)=i
	\xi_0 \cdot e^{ix \cdot \xi_0}e^{-it|\xi_0|^2}u(x-2\xi_0t,t)+e^{ix\cdot\xi_0}e^{-it|\xi_0|^2}\nabla u(x-2\xi_0t,t),
\end{align*}
therefore,
\begin{align}\label{gradientw}
	\ds	\|\nabla w\|^2_{\Lt}=|\xi_0|^2 M[u]+2\xi_0\cdot P[u]+\|\nabla u\|^2_{\Lt}. 
\end{align}
Observe that $M[w]=M[u]$,  $P[w]=\xi_0M[u]+P[u], $ and 
\begin{align}\label{energyw}
	\ds E[w]=\dfrac12|\xi_0|^2M[u]+\xi_0 \cdot P[u]+E[u].
\end{align}

Note that the value $\xi_0=-\frac{P[u]}{M[u]}$ minimizes the expressions  \eqref{gradientw} and \eqref{energyw}, with $P[w]=0$, that is, 
\begin{align*}\ds
		E[w]=E[u]-\dfrac{(P[u])^2}{2M[u]}
		 \quad\quad \mbox{and}\quad\quad		 
		 \|\nabla w\|^2_{\Lt}=\|\nabla u\|^2_{\Lt}-\dfrac{(P[u])^2}{M[u]}.
\end{align*}
Thus, the conditions \eqref{mass-energy},  \eqref{ground-momentum1} and \eqref{ground-momentum2} in Theorem A become
$$ \left(\ME[w]\right)^{\frac1s}=\left(\ME[u]\right)-\frac d{2s}\left(\Pu[u]\right)^{\frac2s}<1,\quad\quad
 [\g_w(0)]^{\frac2s}=[\g_u(0)]^{\frac2s}-\Pu^{\frac2s}[u]<1$$
and $\;\ds [\g_w(0)]^{\frac2s}>1,$
hence we restate Theorem A 
 as 
 
 \begin{thmB*}[Zero momentum]
 Let $u_0\in H^{1}(\Rn)$, $d\ge1$, with $P[u_0]$ and $u(t)$ be the corresponding solution to \eqref{eq:NLS} in $H^1(\Rn)$ with maximal time interval of existence $(T_*,T^*)$ and $s\in (0,1)$. 
Assume $\ME[u]<1.$ 
\begin{enumerate}
\item[I.] If  
$\g_u(0)<1,$
then
\begin{enumerate}[(a)]
\item $\g_u(t)<1$ for all $t\in \R$, thus, the solution is global in time  ( $T_*=-\infty$, $T^*=+\infty$) and 
\item $u$ scatters in $H^1(\Rn)$, this means, there exists $\phi_\pm\in H^1(\Rn)$ such that  
\begin{align*}
\lim_{t\to \pm\infty}\|u(t)-e^{it\Delta}\phi_\pm\|_{H^1(\Rn)}=0.
\end{align*}
\end{enumerate}
\item[II.] If  $\g_u(0)>1$,  then  $\g_u(t)>1$ for all $t\in(T_*,T^*)$ and if 
\begin{enumerate}[(a)]
\item  $u_0$ is radial (for $d\geq3$ and in $d=2, \; 3<p \leq 5$) or  $u_0$ is of finite variance, i.e., $|x|u_0 \in \Lt(\Rn)$,  then the solution blows up in finite time  ($T^* < +\infty$, $\;T_*>-\infty$).
\item $u_0$ non-radial and of infinite variance, then either the solution blows up in finite time
 ( $T^* < +\infty$, $\;T_*>-\infty$)  or there exists a sequence of times $t_n\to +\infty$ (or $t_n\to-\infty$)  such that $\|\nabla u(t_n)\|_{\Lt(\Rn)}\to \infty$.  
\end{enumerate}
\end{enumerate}
\end{thmB*}

Thus, in the rest of the paper, we will assume that $P[u]=0$ and prove only Theorem  A*.  To illustrate the scenarios for global behavior of solutions  given by  Theorem A* 
we provide Figure \ref{fig1}. 
We plot $y=(\ME[u])^{\frac1{s}}$ vs. $x=[\g_u(t)]^{\frac2{s}}$ using the \eqref{thrME-G} restriction 
in Figure 1. 
\begin{figure}[h]
\includegraphics[width=160mm, height=120mm]{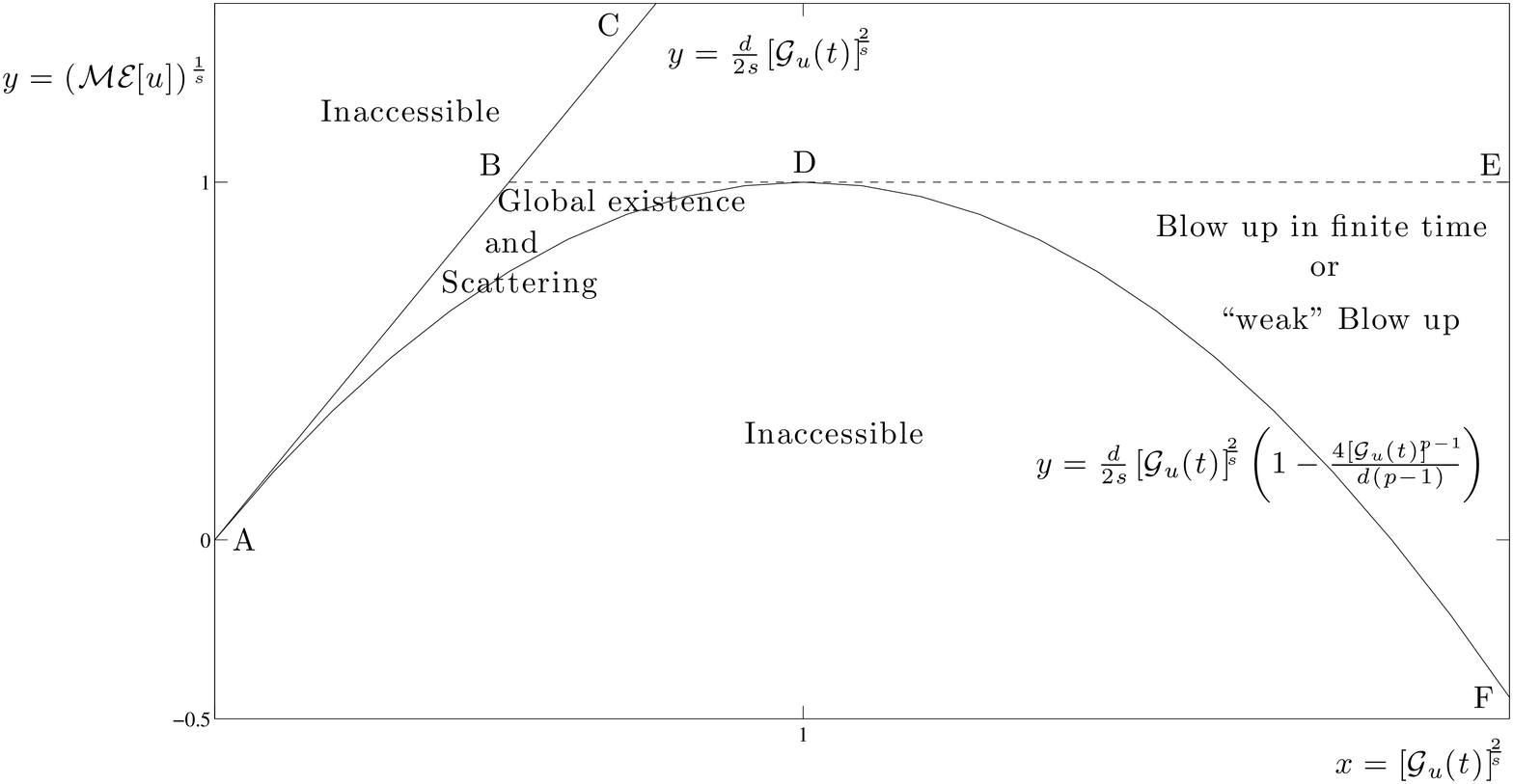}
\caption[Scenarios for global behavior of solutions to the $d$-dimensional focusing critical NLS  with finite energy initial data. ]
{ 
Plot of plot $y=(\ME[u])^{\frac1{s}}$ vs. $x=[\g_u(t)]^{\frac2{s}}$ 
where  $\g_u(t)$ and $\ME[u]$ are defined by \eqref{eq:grad} and \eqref{eq:ME}, respectively.  
The region above the line ABC and below the curve ADF are forbidden regions by \eqref{thrME-G}.  Global existence of solutions and scattering holds in the region ABD, which corresponds to Theorem A* part I and the region EDF explains Theorem A* part II (a), and the ``weak blowup" Theorem A part II  (b). }
 \label{fig1}
\end{figure}

\subsection{ Energy bounds  and Existence of the Wave Operator}
\begin{lemma}[Comparison of Energy and Gradient]\label{equivalence grad and energy} \label{comparison}
Let $u_0 \in H^1(\Rn)$ such that $\g_u(0)<1$ and $\ME[u]<1.$
Then
\begin{equation}\label{eq:ener-grad}
\frac{s}{d}\| \nabla u(t) \|^2_{\Lt} \leq E[u] \leq \frac1 2 \| \nabla u(t) \|^2_{\Lt}.
\end{equation}
\begin{proof}
  The energy definition combined with  $\g(0)<1$ (and thus, by Theorem A* part I (a) $\g_u(t)<1$), the Gagliardo-Nirenberg inequality  \eqref{G-N} and Pohozhaev identities \eqref{poz3}  and \eqref{ground modify} yield
\begin{eqnarray}
E[u]&\geq& \| \nabla u(t) \|^2_{\Lt}\left(\frac12 - \frac{C_{GN}}{p+1} \| \nabla u(t) \|_{L^{2}}^{\frac{d(p-1)}2-2}\|  u \|_{L^{2}}^{2-\frac{(d-2)(p-1)}2}\right)\notag \\
&\geq& \| \nabla u(t) \|^2_{\Lt}\left(\frac12 - \frac{C_{GN}}{p+1} \left(\| \nabla \uQ \|_{L^{2}}^{s}\|  \uQ \|_{L^{2}}^{(1-s)}\right)^{p-1}\right) \notag\\
&=& \| \nabla u(t) \|^2_{\Lt}\left(\frac12 - \frac{C_{GN}}{p+1} \alpha^{-2}\| Q \|_{L^{2}}^{p-1}\right) \notag\\
&=&  \left(\frac{\alpha^{2}-1}{2\alpha^{2}}\right)\| \nabla u(t) \|^2_{\Lt} =
\frac{s}{d}  \| \nabla u(t) \|^2_{\Lt},\label{eq:valor}
\end{eqnarray}
where the equality \eqref{eq:valor} is obtained from combining \eqref{ground modify}, the sharp constant \eqref{const GN} and $\alpha = \frac{\sqrt{d(p-1)}}{2}$.
The second inequality of \eqref{eq:ener-grad} follows directly from the definition of energy.
\end{proof}
\end{lemma}

\begin{lemma}[Lower bound on the convexity of the variance]
\label{Lower-bound-convexity}
Let $u_0 \in H^1(\Rn)$ satisfy  $\g_u(0)<1$ and $\ME[u]<1$. 
Then $\g_u(t)\leq \omega:=\sqrt{\ME[u]}$ for all $t,$ and 
\begin{align}\label{lower}
16(1-\omega^{p-1}) E[u] \leq 8(1-\omega^{p-1})\| \nabla u \|^2_{\Lt}\leq 8 \| \nabla u \|_{\Lt}^2-\frac{4d({p-1})}{p+1}  \|u\|_{L^{p+1}}^{p+1}.
\end{align}
\begin{proof} The first inequality in \eqref {eq:ener-grad} yields $\|\nabla u\|_{\Lt}^2\leq\frac {d}{s}E[u],\;$  
multiplying it by $M^\theta[u]$, where $\theta=\frac{1-s}{s}$, normalizing by  $\|\nabla \uQ\|^2_{\Lt} \| \uQ\|_{\Lt}^{2\theta}$  and using the fact that  $\|\nabla \uQ\|_{\Lt}^2\leq\frac {d}{s}E[\uQ]$ leads to
\begin{eqnarray*}
[\g_u(t)]^2\leq \ME[u], \quad\quad\mbox{i.e,} \quad\quad\g_u(t)\leq\omega.
\end{eqnarray*}
Next, considering the right side of \eqref{lower}, applying Gagliardo-Nirenberg inequality \eqref{G-N}, then the relation \eqref{ground modify} and recalling that $\alpha = \frac{\sqrt{d(p-1)}}{2}$, we obtain
\begin{align}
 8 \|\nabla u\|^2_{\Lt}-\frac{4d(p-1)}{p+1} \|u\|^{p+1}_{L^{p+1}}
\geq \|\nabla u\|^2_{\Lt}\bigg(8- \frac{2d(p-1)}{\alpha^2}[\g_u(t)]^{p-1}\bigg)
\geq8 \|\nabla u\|^2_{\Lt}(1- \omega^{p-1}),\label{lower conv}
\end{align}
which gives the middle inequality in \eqref{lower}. 

Finally, combining \eqref{lower conv} with the second inequality in \eqref{eq:ener-grad},  completes the proof.
\end{proof}

\end{lemma}
\begin{prop}[Existence of  Wave Operators]\label{Existence of wave operator}
Let $\psi \in H^1(\Rn).$
\begin{enumerate}[I.]
\item Then there exists $v_{+}\in H^1$ such that for some $-\infty<T^*<+\infty$ it produces a solution $v(t)$ to $\NLSf$ on time interval $[T^*,\infty)$ such that 
\begin{equation}\label{wave1}
\|v(t)-e^{it\Delta}\psi\|_{H^1}\to0 \quad\quad\text{~~as~~}\quad\quad t\to+\infty
\end{equation}
 Similarly, there exists $v_{-}\in H^1$ such that for some $-\infty<T_*<+\infty$ it produces a solution $v(t)$ to $\NLSf$ on time interval $\;(-\infty,T_*]$ such that 
\begin{equation}\label{wave2}
\|v(-t)-e^{-it\Delta}\psi\|_{H^1}\to0 \quad\quad\text{~~as~~}\quad\quad t\to+\infty
\end{equation}
\item Suppose that for some $0<\sigma\leq\left(\frac{2s}{d}\right)^{\frac s2}<1$
\begin{equation}
\label{psi-identity}
\| \psi \|^{2(1-s)}_{\Lt} \| \nabla \psi \|^{2s}_{\Lt}  < \sigma^2 \left(\frac {d  }{s}\right)^sM[\uQ]^{1-s} E[\uQ]^{s}\;.
\end{equation}
Then there exists $v_0 \in H^1$ such that $v(t)$ solving $\NLSf$ with initial data $v_0$ is global in $H^1$ with 
\begin{eqnarray}
 M[v]= \| \psi \|_{\Lt}^2,\quad\quad
 E[v]= \frac{1}{2} \| \nabla \psi \|_{\Lt}^2, \quad\quad 
 \g_v(t)\leq\sigma< 1\label{wave meg}\\
 \text{and~~}\quad\quad\quad\quad
\| v(t) - e^{i t \triangle} \psi \|_{H^1}\to0 \quad\quad\text{~~as~~}\quad \quad t\to\infty.\label{wave scat}
\end{eqnarray}
Moreover, if $\| e^{i t \triangle} \psi \|_{\dB^0_{S(\dHs)}} \leq \delta_{sd},$ then
$\| v_0 \| _{\dH^{s}} \leq 2 \| \psi \|_{\dH^{s}}$
  and  
  $\| v \|_{\dH^{s}}  \leq 2 \| e^{i t \triangle} \psi \|_{\dB^0_{S(\dHs)}}.$
\end{enumerate}
\end{prop}

\begin{proof}
\noindent I.  This is essentially Theorem 2 part (a) of \cite{St81} adapted to the case $0<s<1$ (see his Remark (36) and \cite[Theorem 17]{St81s}).

\noindent II. For this part, we consider the integral equation
\begin{equation}\label{int-wave}
v(t)= e^{i t \triangle} \psi - i \int_t^{\infty} e^{i (t-t') \triangle} (|v|^{p-1} v) dt'.
\end{equation}
We want to find a solution to \eqref{int-wave} which exists for all t. Note that for $T>0$  from the small data theory (Proposition \ref{small data}) there exists $\delta_{sd}>0$ such that $\| e^{i t \triangle} \psi \|_{\dB^0_{S([T,\infty),\dH^{s})}} \leq  \delta_{sd}$. Thus,  
 repeating the argument of Proposition \ref{small data}, we first show that we can solve the equation \eqref{int-wave} in $\dHs$ for $t\geq T$ with $T$ large. So this solution will estimate $\| \nabla v \|_{S(\Lt;[T, \infty))}, $ which will also show that $v$ is in $H^1.$

Observe that for any $v\in H^1$
\begin{align}
\| \nabla |v|^{p-1}v\|_{\SdLt} \notag
&\lesssim \| \nabla |v|^{p-1}v\|_{\qsd \rsd}
\\&\lesssim \|v\|^{p-1}_{\qs \rsh} \|\nabla v\|_{\qs \rs} 
\lesssim \|v\|^{p-1}_{S(\dH^s)} \|\nabla v\|_{\SLt}.\label{est1-wave}
\end{align}
Note that the pairs $\big(\frac {d}{2 s},\frac{d^2p(p-1)}{2(d+4)}\big)$ and $\big(\frac {d}{2 s},\frac {2d^2 p}{d^2p-8s}\big)$ are $\Lt-$ admissible and the pair $\big(\frac {d}{2 s},\frac{2 d^2(p-1)}{d^2 (p - 1) + 16}\big)$ is $\Lt-$ dual admissible.  Thus,  the H\"older's inequality  yields  \eqref{est1-wave}.
Now, the Strichartz \eqref{eq:stri} and Kato Strichartz \eqref{eq:Kato-Strichartz} estimates  imply
\begin{align*}
\| \nabla v \|_{S([T, \infty), \Lt)} & \lesssim c_1 \|\nabla \psi \|_{S([T, \infty), \Lt)}+  c \| \nabla (|v|^p v )\|_{S'([T, \infty), \Lt)} \\
      & \lesssim c_1\| \psi \|_{\dH^1}+  c_3 \|v\|^{p-1}_{S([T, \infty),\dH^s)} \|\nabla v\|_{S([T, \infty), \Lt)}.
      \end{align*}
Taking $T$ large enough, so that $c_3 \|  v \|_{S([T, \infty),\dH^{s})}\leq \frac 12$, we obtain
\begin{equation*}
\label{nablav}
\| \nabla v \|_{S([T, \infty),\Lt)}  \leq 2 c_1 \| \psi \|_{H^1}.
\end{equation*}
It now follows
\begin{align*}
\| \nabla ( v-e^{i t \Delta} \psi)\|_{S([T, \infty),\Lt)} &\leq \| \nabla( |v|^{p-1} v ) \|_{S'([T, \infty),\Lt)}\\&
 \leq \| \nabla v \|_{S([T, \infty),\Lt)} \| v \|^{p-1}_{S([T, \infty),H^s)} \leq   c \| \psi \|_{H^1},
\end{align*}
hence, $\| \nabla ( v-e^{i t \Delta} \psi)\|_{S(\Lt([T, \infty)))} \to 0$ as $T \to \infty$. 

On the other hand,  Proposition \ref{H^1 Scattering} ($H^1$ Scattering) implies $v(t) \to e^{i t \Delta} \psi$ in $H^1$ as $t \to \infty$, and the decay estimate  together with the embedding 
and $H^1(\Rn)\hookrightarrow  L^q(\Rn)$ for  
$q\leq \frac{2d}{d-2}$ when $3\leq d$, $q<\infty$ when  $d=2$ and $q\leq\infty$ when  $d=1$ 
imply
$$
\| e^{it \Delta} \psi \|_{L_x^{p+1}} \leq |t|^{\frac{-(p-1)d}{2(p+1)}} \| \psi \|_{H^1},$$
thus, $\| e^{it \Delta} \psi \|_{L_x^{p+1}} \to 0\;  \mbox{as}\; t \to \infty.
$ 
Since  $\| \nabla e^{it \Delta} \psi \|_{\Lt}=\| \nabla \psi \|_{\Lt}$, it follows
\begin{align*}
E[v] &=   \frac{1}{2} \| \nabla v \|_{\Lt}^2 - \frac{1}{p+1} \| v \|_{L_x^{p+1}}^{p+1}  \\
&= \lim_{t \to \infty}\Big(\frac{1}{2} \| \nabla  e^{it \Delta} \psi \|_{\Lt}^2 - \frac{1}{p+1} \|  e^{it \Delta} \psi \|_{L_x^{p+1}}^{p+1} \Big) = \frac{1}{2} \| \nabla \psi \|^2_{\Lt}
\end{align*}and
$$
M[v]=\lim_{t \to \infty} \|  v(t) \|_{\Lt}^2=\lim_{t \to \infty}  \| e^{it \Delta} \psi\|_{\Lt}^2=  \| \psi \|_{\Lt}^2.
$$
\noindent From the hypothesis \eqref{psi-identity},  we obtain 
 $$
 M[v]^{1-s}E[v]^s= \frac1{2^s}\| \psi \|^{2(1-s)}_{\Lt} \| \nabla \psi \|^{2s}_{\Lt}   < \sigma^2 \left(\frac {d  }{2s}\right)^sM[\uQ]^{1-s} E[\uQ]^{s}
$$
and thus,
$\ME[v]<1$,  since $\sigma^2<\left(\frac {2s}d\right)^s.$
Furthermore,
\begin{align*}
\lim_{t \to \infty} \| v(t) \|_{\Lt}^{2(1-s)}\| \nabla v(t) \|_{\Lt}^{2s} 
&= \lim_{t \to \infty}   \| e^{it \Delta} \psi \|_{\Lt}^{2(1-s)}\| \nabla(e^{it \Delta} \psi) \|_{\Lt}^{2s}
 =   \| \psi \|_{\Lt}^{2(1-s)}\| \nabla \psi \|_{\Lt}^{2s}\\
  & < \sigma^2 \left(\frac {d  }{s}\right)^s M[\uQ]^{1-s} E[\uQ]^{s} 
  = \sigma^2 \| \uQ \|_{\Lt}^{2(1-s)}  \| \nabla \uQ \|_{\Lt}^{2s},
\end{align*} 
where, the inequality is due to \eqref{psi-identity} and the last equality is obtained using \eqref{mass modify}.  Hence, $$\lim_{t\to\infty}\g_v(t)\leq\sigma<1.$$ 
We can take $T>0$ large so that $\g_v(T)\leq1.$ 
Then applying Theorem A* part I (a) (global existence of solutions with $\ME[v]<1$ 
and $\g_v(t)<1$), we evolve $v$ from time $T$ back to time 0 (we automatically get $\g_v(t)\leq1$ for all $t\in[0,+\infty)$.) Thus, we obtain   $v$ with initial data $v_0 \in H^1$ and properties \eqref{wave meg} and \eqref{wave scat} as desired.
\end{proof}

\section{Scattering via Concentration Compactness}\label{Scattering_cap}

 Theorem 2.1 and Corollary 2.5 of Holmer-Roudenko \cite{HoRo07} proved the general case  for the mass-supercritical and energy-subcritical NLS equations with $H^1$ initial data, thus, establishing Theorem A* I(a) and II(a) for finite variance data. In addition,  \cite{CaGu11} included  the proof of the blow up in finite time when $d=2$ and $p=5$ for the radial initial data (i.e., Theorem A* part II(a)), since it was not include  in \cite{HoRo07}, the authors considered $p<5$.
 
The goal of this section is to prove scattering in $H^1(\Rn)$ of global solutions of  $\NLSf$ 
 from Theorem A* part I (a).

\begin{definition}
Suppose $u_0\in H^1(\Rn)$ and  let $u$ be the corresponding $H^1(\Rn)$ solution to \eqref{eq:NLS} on $[0,T^*)$, the maximal (forward in time) interval of existence. We say that  $SC(u_0)$ holds if $T^*=+\infty$ and $\|u\|_{\dB^0_{\SHs}}<\infty$. 
 \end{definition}

\subsection{ Outline of Scattering via Concentration Compactness}\label{scattering}
Notice that  $H^1-$ scattering of $u(t)=\NLS(t)u_0$ is obtained when  $SC(u_0)$ holds by Proposition \ref{H^1 Scattering}. Therefore, to establish Theorem A* part I (b), it  will be enough  to verify that the global-in-time $\dHs$  Besov-Strichartz norm is finite, i.e., $\|u\|_{\dB^0_{\SHs}}<\infty$,  
since the hypotheses provides an \emph{a priori bound}  for $\|\nabla u(t)\|_{\Lt}$ (by Theorem A* part I a), thus, the maximal forward time of existence is $T=+\infty$. In other words, it  remains to show 
\begin{prop} \label{scatteringcondition}
If $\g_u(0)<1$ and $\ME[u]<1$, then $SC(u_0)$ holds.
\end{prop}

The technique to achieve the scattering property  above (Proposition \ref{scatteringcondition}) is the induction argument on the mass-energy threshold introduced in \cite{KeMe06} and is based on \cite{HoRo08} and \cite{DuHoRo08}. We describe it in  steps 1, 2, 3. 

\noindent \textbf{ Step 1:} \emph{ Small Data.} 

\noindent 
The equivalence of energy with the gradient (Lemma \ref{equivalence grad and energy}) yields 
$$ 
\|u_0\|_{\dHs}^{p+1}\leq(\|u_0\|^{1-s}_{\Lt}\| \nabla u_0\|^{s}_{\Lt})^{\frac{p+1}2} \leq \left(\bigg(\frac {d}{s}\bigg)^sM[u]^{1-s}E[u]^s\right)^{\frac{p+1}4}.
$$
If $\g_u(0)<1$ and $M[u]^{1-s}E[u]^s<\big(\frac {s}{d}\big)^{s}\delta_{sd}^4$, then using the above inequality one obtains
$\|u_0\|_{\dHs}\leq\delta_{sd}$
and  by Strichartz estimates $\|e^{it\Delta}u_0\|_{\dB^0_{\SHs}}\leq c \delta_{sd}$. Hence,  the small data (Proposition \ref{small data}) yields  $SC(u_0)$ property.

Observe that  Step 1 gives the basis for induction:

\noindent 
Assume $\g_u(0)<1.$ Then  for small $\delta>0$ such that $M[u_0]^{1-s}E[u_0]^{s}<\delta$, we have that $SC({u_0})$ holds. 

\noindent Let $(ME)_c$ be the supremum of all such $\delta$ for which $SC(u_0)$ holds, namely,
\begin{align*}
(ME)_c= \sup \big\{\delta \;|&\; u_0\in H^1(\Rn) \mbox{ with the property: } \\
&\g_u(0)<1 \mbox{ and }M[u]^{1-s}E[u]^{s} < \delta \Rightarrow SC(u_0) \mbox{ holds}\big\}.
\end{align*}
Thus,  the goal is to show that $(ME)_c =M[\uQ]^{1-s}E[\uQ]^{s}$. 

{\rmk In the definition of $(ME)_c$,  it should be considered $\g_u(0)\leq1$  instead of the strict inequality $\g_u(0)<1.$ However, $\g_u(0)$=1 only when $\ME[u]=1$ (see Figure \ref{fig1} point D). In other words,  $u_0=\uQ(x)$ is a soliton solution to \eqref{eq:NLS} and does not scatter, thus, it suffices to consider the strict inequality $\g_u(0)<1.$ }
\bigskip

\noindent \textbf{ Step 2:} \emph{ Induction on the scattering threshold and construction of the ``critical'' solution.} 

\noindent
Assume that $(ME)_c<M[\uQ]^{1-s}E[\uQ]^{s}$. This means that, there exists a sequence of initial data $\{u_{n,0}\}$ in $H^1(\Rn)$ which will approach the threshold $(ME)_c$ from above and produce solutions which do not scatter, i.e., there exists a sequence $u_{n,0}\in H^1(\Rn)$ with 
 \begin{align}\label{sequn}
 \g_{u_n}(0)<1\; \text{~~and~~}\;M[u_{n,0}]^{1-s}&E[u_{n,0}]^{s}\searrow(ME)_c\;\mbox{as}\; n\to\infty\\
\notag  \text{~~and~~}& \|u\|_{\dB^0_{\SHs}}=+\infty,
 \end{align}

i.e., $SC(u_{n,0})$ does not hold (this is possible by definition of supremum of $(ME)_c$).

Using a nonlinear profile decomposition on the sequence $\{u_{n,0}\}$ will allow us to construct a ``critical" solution of $\NLSf$, denoted by $u_c(t),$ that will lie exactly at the threshold $(ME)_c$ and will not scatter, see Existence of the Critical solution (Proposition \ref{existence of u_c}).

\noindent \textbf{ Step 3:}  \emph{ Localization properties of the critical solution.} 

\noindent The critical solution $u_c(t)$ will have the property that
$K=\{u_c(t)|t\in[0,+\infty)\}$ is precompact in $H^1(\Rn)$ (Lemma \ref{precompact}).  
Hence, its localization implies that for given $\epsilon >0$, there exists an $R>0$ such that
$
\|\nabla u(x,t)\|_{\Lt(|x+x(t)|>R)}^2\leq\epsilon
$   
uniformly in $t\;$ (Corollary \ref{precompact-localization}); this combined with the zero momentum will give control on the growth of $x(t)$ (Lemma \ref{spacial translation}). 

On the other hand, the rigidity theorem (Theorem \ref{rigidity}) implies that such compact in $H^1$ solutions  with the control on $x(t)$, can only be zero solutions,  which contradicts the fact that $u_c$ does not scatter. As a consequence, such $u_c$ does not exist and the assumption that $(ME)_c<M[\uQ]E[\uQ]$ is not valid. This finishes the proof of scattering in Theorem A*, Part 1(b).

 In the rest of this section we proceed with the linear and nonlinear profile decomposition and the proof of the existence and properties of the critical solution described  in Step 2 and Step 3. 

\subsection{Profile decomposition}\label{profile section}

This section contains the profile decomposition for linear and nonlinear flows for $\NLS^+_{p}(\Rn)$. The important point to make here is that these are general profile decompositions for bounded sequences on $H^1$.

\begin{prop}[Linear Profile decomposition]\label{nonradialPD}
Let $\phi_n(x)$ be a uniformly bounded  sequence in $H^1(\Rn)$. Then for each $M\in\N$ there exists a subsequence of $\phi_n$ (also denoted $\phi_n$), such that, for each $1 \leq j \leq M$, there exist, fixed in $n$, a  profile $\psi^j$ in $H^1(\Rn)$, a sequence  $t^j_n$ of time shifts, a sequence $x^j_n$ of space shifts 
and a sequence $W^M_n(x)$ of remainders\footnote[11]{\normalsize  Here, in Proposition \ref{nonradialPD} and Proposition \ref{nonradialPDNLS}, $W^M_n(x)$ and  $\wt^M_n(x)$ represent the remainders for the linear and nonlinear decompositions, respectively.\\} 
 in $H^1(\Rn)$,  such that  
$$
\phi_n(x)=\sum^{M}_{j=1} e^{-i t_n^j \Delta } \psi^j(x-x^j_n) + {W}^M_n(x)
$$
with the properties:
\begin{itemize}
\item Pairwise divergence for the time and space sequences. 
For $1 \leq k \neq j \leq M$, 
\begin{align}
\label{time-space-seq}
\lim_{n \to \infty} |t^j_n-t^k_n | + |x^j_n - x^k_n|=+ \infty.
\end{align} 
\item  Asymptotic smallness for the remainder sequence
\begin{align}
\label{smallnessProp}
\lim_{M \to \infty} \big( \lim_{n \to \infty} \| e^{i t \Delta}W^M_n \|_{\dB^0_{\SHs}}\big ) = 0.
\end{align}
\item Asymptotic Pythagorean expansion. 
For fixed $M\in\N$ and any $0 \leq s \leq 1$, we have 
\begin{align}
\label{limHs}
\| \phi_n \|^2_{\dot{H}^{s}} = \sum^M_{j=1} \| \psi^j \|^2_{\dot{H}^{s}} + \|W^M_n \|^2_{\dot{H}^{s}} +o_n(1).
\end{align}
\end{itemize}
\end{prop}

\begin{proof}
Let $\phi_n$ be uniformly bounded in $H^1$, and $c_1>0$ such that $\|\phi_n \|_{H^1}\leq c_1$. 

For each dyadic $N\in 2^{\N}$, given $(q,r)$ an $\dot{H}^{s}$ admissible pair, pick   
  $\theta=\frac{4d(d+r)-2d^2r}{r^2(d-2s)(d-2)-2dr(d+2s-4)}$,  so $0<\theta<1$.
  
  Let $r_1=\frac{r (d - 2) + 2 d}{2 (d - 2)}$, and $q_1=\frac{8(d-2)+4dr}{r(d-2s)(d-2)-2d(d+2s-4)}$, so $(q_1,r_1)$ is  $\dot{H}^{s}$ admissible pair, for $0<s<1$ and $d\geq 2$. Interpolation and Strichartz estimates \eqref{eq:strisob} yield
\begin{align}
\| e^{it \Delta} W^M_n \|_{L^q_t L^r_x}
&\leq \|e^{it \triangle} W^M_n \|^{1-\theta}_{L^{q_1}_t L^{r_1}_x} \| e^{it \triangle} W^M_n \|^{\theta}_{L^{\infty}_t L^{\frac{2 d}{d - 2 s}}_x} \notag
\\&\leq c \| W^M_n \|^{1-\theta}_{\dHs} \| e^{it \Delta} W^M_n \|^{\theta}_{L^{\infty}_t L^{\frac{2 d}{d - 2 s}}_x}.\label{eq:rem}
\end{align}

The goal is to write the profile $\phi_n$  as $\sum^{M}_{j=1} e^{-i t_n^j \Delta } \psi^j(x-x^j_n) + {W}^M_n(x)$ with  $\|W^M_n(x)\|_{\dHs}\leq c_1$, for some constant $c_1$. By \eqref{eq:rem}, it suffices to show 
$$
\lim_{M\to +\infty} \left[ \limsup_{n\to +\infty}\|e^{it\Delta}
W_n^M\|_{L_t^\infty L_x^{\frac{2 d}{d - 2 s}}}\right] = 0\;. 
$$
We have $d\geq2s,$ since we are considering 
\begin{eqnarray}\label{casos dim}
\left\{\begin{array}{cccc}(i) & 0\leq s\leq 1 & \text{in} & d\geq3 \\(ii) & 0<s<1 & \text{in} & d=2 
\\(iii) & 0<s<\frac12 & \text{in} & d=1.
\end{array}\right.
\end{eqnarray}

 \noindent\emph{ Construction of $\psi^1_n$ }:  

\noindent Let $A_1 = \limsup_{n\to+\infty} \| e^{i t \Delta } \phi_n \|_{L^{\infty}_t L_x^{\frac{2 d}{d - 2 s}}}$. 
If $A_1=0$, taking $\psi^j=0$ for all $j$ finishes the construction.  

Suppose that $A_1>0$, and let
$c_1= \limsup_{n\to+\infty}\| \phi_n \|_{H^1}<\infty.$ 
Passing to a subsequence $\phi_n$, we show that there exist sequences $t^1_n$ and $x^1_n$ and a function $\psi^1 \in H^1$, such that  $$e^{it^1_n \Delta } \phi_n( \cdot + x^1_n)  \rightharpoonup \psi^1\quad\mbox{  in }\quad H^1,$$ 
and a constant $K>0,$ independent of all parameters, with
\begin{align}
K c_1^{\frac{d+2s-4s^2}{2s}} \| \psi^1 \|_{\dHs} \geq A_1^{\frac{d+4s-4s^2}{2s}}\label{boundpsi}.
\end{align}
Note that $d+2s-4s^2=d+2s(1-2s)>0$ by \eqref{casos dim}.

Let $\chi_r$ be a radial Schwartz function such that  supp $\chi_r \subset \big[\frac{1}{2r}, 2r\big]$ and $\widehat{\chi}_r(\xi)=1$ for $\frac{1}{r} \leq | \xi | \leq r$. 
Note that   $|1-\hat{\chi_r}| \leq 1$ and 
$\dHs \hookrightarrow L^{\frac{2 d}{d - 2 s}}$ in $\Rn$ with $2s<d$, then
\begin{align*}
\| &e^{it \Delta} \phi_n - \chi_r *  e^{it \Delta} \phi_n \|^2_{L^{\infty}_t L^{\frac{2 d}{d - 2 s}}_x}
 \leq \int |\xi| (1-\hat{\chi_r}(\xi))^2 | \hat{\phi}_n(\xi) |^2 d \xi \\
& \leq \int_{| \xi| \leq \frac{1}{r}} |\xi|  | \hat{\phi}_n |^2 d \xi + \int_{| \xi| \geq r} |\xi|  | \hat{\phi}_n (\xi)|^2 d \xi 
  \leq \frac{\| \phi_n \|_{\Lt_x}^2 + \| \phi_n \|_{\dot{H}^1_x}^2}{r}
\leq \frac{c_1^2}{r}.
\end{align*}
Take $r=\frac{4 c_1^2}{A_1^2}$, then $A_1=\frac{2 c_1}{\sqrt{r}}$. Using the definition of $A_1$, triangle inequality and the previous calculation, for large $n$ we have
\begin{align}
 \frac{A_1}{2}\leq\| \chi_r *  e^{it \Delta} \phi_n \|_{L^{\infty}_t L^{\frac{2 d}{d - 2 s}}_x}  .\label{A1}
\end{align}
Therefore, interpolation implies
\begin{align}
\| \chi_r *  e^{it \Delta} \phi_n \|^d_{L^{\infty}_t L^{\frac{2 d}{d - 2 s}}_x} &\leq
 \| \chi_r *  e^{it \Delta} \phi_n \|_{L^{\infty}_t \Lt_x}^{d-2s}  \|\chi_r *  e^{it \Delta} \phi_n \|_{L^{\infty}_t L_x^{\infty}} ^{2s}  \leq \| \phi_n \|_{\Lt_x}^{d-2s}  \|\chi_r *  e^{it \Delta} \phi_n \|_{L^{\infty}_t L_x^{\infty}}^{2s},\label{estA1}
\end{align}
where the second inequality follows from the fact that $|\widehat{\chi_r}| \leq 1$ and  $\Lt$ isometry property of the linear Schr\"odinger operator.  
Using the definition of $c_1$, combining \eqref{A1} and \eqref{estA1}, we get 
$\Bigg(\frac{A_1}{2 c_1^{\frac{d-2s}{d}}}\Bigg)^{\frac d{2s}}\leq  \|\chi_r *  e^{it \Delta} \phi_n \|_{L^{\infty}_t L_x^{\infty}}.$ 
Thus, there exists a sequence of $(x^1_n,t^1_n) \in \Rn \times \R^1_+$
satisfying 
$\Bigg(\frac{A_1}{2 c_1^{\frac{d-2s}{d}}}\Bigg)^{\frac d{2s}}\leq|\chi_r *  e^{i t^1_n \Delta} \phi_n(x^1_n)| .$
Since  $e^{it \Delta}$ is an $H^1$ isometry and  translation invariant,
 it follows that   $\{ e^{it_n^1\Delta} \phi_n (\cdot+x_n^1)\}$ is uniformly bounded in $H^1$ (with the same constant as $\phi_n$'s) and along a subsequence 
$ \{ e^{it_n^1\Delta} \phi_n (\cdot+x_n^1)\} \rightharpoonup \psi^1$ with 
$\| \psi^1 \|_{H^1} \leq c_1$.

Observe that 
$$
\bigg(\frac{A_1}{2 c_1^{\frac{d-2s}{d}}}\bigg)^{\frac d{2s}}\leq \big|\int_{\mathbb{R}^2} \chi_r(x_n^1-y) \psi^1(y) dy \big| \leq \| \chi_r \|_{\dH^{-s}} \| \psi^1 \|_{\dHs} \leq r^{1-s} \| \psi^1 \|_{\dHs},
$$
 since
$\| \chi_r \|_{\dH^{-s }}^2\lesssim r^{1-s}$ (by converting to radial coordinates) and the  H\"older's inequality produces \eqref{boundpsi} with $K=2^{\frac{d+4s-4s^2}{2s}}.$
  
  Define $W^1_n (x) = \phi_n(x)-e^{-it^1_n \Delta} \psi^1(x-x^1_n)$. Note that $e^{it^1_n \Delta } \phi_n( \cdot + x^1_n)  \rightharpoonup \psi^1$ in $H^1$, therefore, for any $0\leq s \leq1$, we have
  $$
  \langle  \phi_n,e^{-it^1_n \Delta} \psi^1(\cdot-x^1_n)\rangle_{\dH^{s}}=\langle e^{it^1_n \Delta}\phi_n, \psi^1(\cdot-x^1_n) \rangle_{\dH^{s}}\to\|\psi^1\|_{\dH^{s}}^2,
  $$
and since $ \|W_n^1 \|^2_{\dH^{s}}=\langle  \phi_n-e^{-it^1_n \Delta} \psi(\cdot-x^1_n), \phi_n-e^{-it^1_n \Delta} \psi^1(\cdot-x^1_n) \rangle^2_{\dH^{s}},$ we obtain
\begin{align*}
\lim_{n \to \infty} \| W^1_n \|^2_{\dH^{s}} =
\lim_{n \to \infty} \| e^{i t^1_n \Delta} \phi_n \|^2_{\dH^{s}} -\| \psi^1 \|^2_{\dH^{s}}.
\end{align*} 
Taking $s=1$ and $s=0$, yields  $\|W^1_n\|_{H^1}\leq c_1.$
\bigskip

\noindent\emph{ Construction of $\psi^j$ for $j \geq 2$} (
 Inductively we assume that $\psi^{j-1}$ is known 
and construct $\psi^j$):  
 
\noindent Let $M \geq 2$.   Suppose that $\psi^j$, $x_n^j$, $t^j_n$ and $W^j_n$ are known for 
$j \in \{ 1 ,\cdots , M-1 \}$. Consider
$$
A_M= \limsup_n \| e^{it\Delta}W^{M-1}_n \| _{L^{\infty}_t L_x^{\frac{2 d}{d - 2 s}}}.
$$
If $A_M=0$, then  
taking  $\psi^j=0$ for $j \geq M$ will end the construction. 

Assume $A_M>0$, we apply the previous step to $W^{M-1}_n$,  and let
$c_M= \limsup_n \| 
W^{M-1}_n \|_{H^1},$ thus,
obtaining sequences (or subsequences) $x_n^M, t_n^M$ and a function $\psi^M \in H^1$ such that
\begin{align}
e^{it^M_n \Delta } W^{M-1}_n( \cdot + x^M_n)  \rightharpoonup \psi^M \quad \mbox{in} \quad H^1\quad \mbox{and}\quad
  K c_M^{\frac{d+2s-4s^2}{2s}} \| \psi^M \|_{\dot{H}^{s}} \geq  A_M^{\frac{d+4s-4s^2}{2s}}.
 \label{condpsi} 
  \end{align}

Define 
$$
W^M_n(x) = W^{M-1}_n (x) - e^{-i t^M_n \Delta} \psi^M (x-x^M_n).
$$ 
Then \eqref{time-space-seq} and  \eqref{limHs}  follow from induction, i.e., we assume \eqref{limHs} holds at rank $M-1$, then expanding
$$
\|W^M_n(x)\|^2_{\dH^{s}}= \| e^{i t^M_n \Delta} W^{M-1} (\cdot+x^M_n)-\psi^M\|^2_{\dH^{s}},
$$ 
 the weak convergence yields \eqref{limHs} at  rank $M.$ 

In the same fashion, we assume \eqref{time-space-seq} is true for $j,k \in \{ 1, \hdots , M-1\}$ with $j\neq k$, that is $|t^j_n-t^k_n | + |x^j_n - x^k_n|\to+ \infty $ as $n\to \infty$. Take $k \in \{ 1, \hdots , M-1\}$ and show that 
$$
|t^M_n-t^k_n | + |x^M_n - x^k_n|\to+ \infty.
$$

 Passing to a subsequence, assume $t^M_n-t^k_n\to t^{M_1}$ and $x^M_n-x^k_n\to x^{M_1}$ finite, then as $n \to \infty$
 \begin{align*}
e^{i t_n^M \Delta} W^{M-1}_n (x+x_n^M) =& e^{i (t^M_n-t^j_n) \Delta}(e^{i t_n^j \Delta} W^{j-1}_n (x+x_n^j)-\psi^j(x+x_n^j))
\\&
- \sum^{M-1}_{k=j+1} e^{i (t_n^j-t^k_n) \Delta} \psi^{k}_n (x+x_n^j-x_n^k) .
 \end{align*}
The orthogonality condition \eqref{time-space-seq} implies that  the right hand side goes to $0$ weakly in $H^1$, while the left side converges weakly to a nonzero $\psi^M$, which is a contradiction. Note that the orthogonality condition \eqref{time-space-seq} holds for $k=M$, and since \eqref{limHs} holds for all $M,$ we have 
$$
\| \phi_n \|^2_{\dH^{s}} \geq \sum^M_{j=1} \| \psi^j \|^2_{\dH^{s}} + \|\,W^M_n \|^2_{\dH^{s}} 
$$
and  $c_M\leq c_1$. Fix $s$. If for all $M,$ $A_M>0,$ then  \eqref{condpsi}  yields
$$\sum _{M\geq1}\Bigg(\frac{A_M^{{\frac{d+4s-4s^2}{2s}}}}{Kc^{{\frac{d+2s-4s^2}{2s}}}_1}\Bigg)^2\leq \sum _{n\geq1}\|\psi^M\|^2_{\dHs}\leq \limsup_n\|\phi_n\|^2_{\dHs}< \infty,$$
 therefore, $A_M \to 0$ as $M \to \infty,$ and consequently, $ \| e^{i t \Delta}W^M_n \|_{\SHs}\to 0$ as $n\to \infty.$ Finally,  summing over all dyadic $N,$ yields \eqref{smallnessProp}.
\end{proof}
\begin{prop}[Energy Pythagorean expansion]\label{energy Pythagorean expansion} Under the hypothesis of Proposition \ref{nonradialPD}, we have 
\begin{align}
\label{pythagorean}
E[\phi_n]=\sum_{j=1}^M E [e^{-it^j_n\Delta}\psi^j]+E[W^M_n]+o_n(1).
\end{align}
\end{prop}
\begin{proof}
By definition of $E[u]$ and \eqref{limHs} with $s=1$, it suffices to prove that  for all $M\leq1$, we have
\begin{align}
\label{norm6seq}
\|\phi_n\|^{p+1}_{L^{p+1}}=\sum_{j=1}^M\|e^{-it_n^j\Delta}\psi^j\|^{p+1}_{L^{p+1}}+o_n(1).
\end{align}

\noindent\emph{ Step 1. Pythagorean expansion of  a sum of orthogonal profiles.} Fix $M\geq1$. We want to show that the condition \eqref{time-space-seq} yields
\begin{align}
\label{normsum}
\bigg\|\sum_{j=1}^M e^{-it^j_n\Delta}\psi^j(\cdot-x^j_n)\bigg\|^{p+1}_{L_x^{p+1}}=\sum_{j=1}^M\|e^{-it_n^j\Delta}\psi^j\|^{p+1}_{L_x^{p+1}}+o_n(1).
\end{align}
By rearranging and reindexing, we can find $M_0\leq M$ such that 
\begin{enumerate}
\item[(a)]  $t_n^j$ is bounded in $n$ whenever $1\leq j\leq M_0$, 
\item[(b)] $|t_n^j|\to \infty$ as $n\to\infty$  if $M_0+1\leq j\leq M.$
\end{enumerate}

For the case (a) take a subsequence and assume that  for each $1\leq j \leq M_0$, $t_n^j$ converges (in $n$), then adjust the profiles $\psi^j$'s such that  $t^j_n=0$. From \eqref{time-space-seq}  we have
$|x^j_n-x^k_n|\to +\infty$ as $n\to\infty$, which implies
\begin{align}
\label{normsum0}
\bigg\|\sum_{j=1}^{M_0} \psi^j(\cdot-x^j_n)\bigg\|^{p+1}_{L_x^{p+1}}=\sum_{j=1}^{M_0}\|\psi^j\|^{p+1}_{L_x^{p+1}}+o_n(1).
\end{align}

For the case (b), i.e., for $M_0\leq j \leq M$,  $|t_n^j|\to \infty$ as $n\to\infty$ and for $\tilde{\psi}\in \dH^{\frac{p}{p+1}}\cap L^{\frac{p}{p+1}}$, thus, the Sobolev embedding and the $L^p$ space-time decay estimate 
yield 
\begin{align*}
\|e^{-it^k_n\Delta}\psi^k\|_{L_x^{p+1}}\leq c\|\psi^k-\tilde{\psi}\|_{\dH^{\frac{p}{p+1}}}+\dfrac {c}{|t^k_n|^{\frac{d(p-1)}{2(p+1)}}}\|\tilde{\psi}\|_{L^{\frac{p+1}{p}}_x},\end{align*}
and approximating $\psi^k$ by $\tilde{\psi} \in C^\infty_{comp}$ in $\dH^{\frac{p}{p+1}}$,  we have 
\begin{align}
\label{divergent}
\|e^{-it^k_n\Delta}\psi^k\|_{L_x^{p+1}}\to 0 \mbox{    as } n\to\infty.
\end{align}
Thus, combining \eqref{normsum0} and \eqref{divergent}, we obtain \eqref{norm6seq}.

\noindent\emph{ Step  2. Finishing the proof.}  Note that 
\begin{align*}
\|W^{M_1}_n\|_{L^{p+1}_x}&\leq\|W^{M_1}_n\|_{L^\infty_tL^{p+1}_x}
\leq\|W^{M_1}_n\|^{1/2}_{L^\infty_tL^{\frac{2 d}{d - 2 s}}_x}\|W^{M_1}_n\|^{1/2}_{L^\infty_tL^{{\frac{2 d( d + 2 - 2  s)}{d(d-2) + 4s( 1 - s)}}}_x}\\
&\leq\|W^{M_1}_n\|^{1/2}_{L^\infty_tL^{\frac{2 d}{d - 2 s}}_x}\|W^{M_1}_n\|^{1/2}_{L^\infty_t\dH^1_x}
\leq\|W^{M_1}_n\|^{1/2}_{L^\infty_tL^{\frac{2 d}{d - 2 s}}_x}\sup_n\|\phi_n\|^{1/2}_{H^1}.
\end{align*}
By \eqref{smallnessProp} it follows that
\begin{align}
\lim_{M_1\to+\infty}\Big(\lim_{n\to+\infty} \|e^{it\Delta}W^{M_1}_n\|_{L^{p+1}}\Big)=0
\label{limM1}.
\end{align}
Let $M\geq1$ and $\epsilon>0.$ The sequence of profiles $\{\psi^n\}$ is uniformly bounded in $H^1$ and in $L^{p+1}$.  Hence, \eqref{limM1} implies that the sequence of remainders $\{W^M_n\}$ is also uniformly bounded in $L_x^{p+1}$. Pick $M_1\geq M$ and $n_1$ such that for $ n\geq n_1$, we have
\begin{align}
\label{est1}
\Big|\|\phi_n&-W^{M_1}_n\|_{L^{p+1}_x}^{p+1}-\|\phi_n\|_{L^{p+1}_x}^{p+1}\Big|+\Big|\|W^{M}_n-W^{M_1}_n\|_{L^{p+1}_x}^{p+1}-\|W^{M}_n\|_{L^{p+1}_x}^{p+1}\Big|\\
&\leq C\Big(\big(\sup_n\|\phi_n\|^{p}_{L^{p+1}_x}+\sup_n\|W^{M}_n\|^{p}_{L^{p+1}_x}\big)\|W^{M_1}_n\|_{L^{p+1}_x}+\|W^{M_1}_n\|^{p+1}_{L^{p+1}_x}\Big)\leq  \dfrac{\epsilon}{3}.\notag
\end{align}
Choose $n_2\geq n_1$ such that $n\geq n_2$. Then \eqref{normsum} yields
\begin{align}
\label{est2}
\Big|\|\phi_n-W^{M_1}_n\|_{L^{p+1}_x}^{p+1}-\sum^{M_1}_{j=1}\|e^{-it^j_n\Delta}\psi^j\|_{L^{p+1}_x}^{p+1}\Big|\leq \dfrac{\epsilon}{3}.
\end{align}
Since $W^{M}_n-W^{M_1}_n=\sum^{M_1}_{j=M+1}e^{-it^j_n\Delta}\psi^j(\cdot-x_n^j)$, by \eqref{normsum}, there exist $n_3\geq n_2$ such that $n\geq n_3$, 
\begin{align}
\label{est3}
\Big|\|W^{M}_n-W^{M_1}_n\|_{L^{p+1}_x}^{p+1}-\sum^{M_1}_{j=M+1}\|e^{-it^j_n\Delta}\psi^j\|_{L^{p+1}_x}^{p+1}\Big|\leq  \dfrac{\epsilon}{3}.
\end{align}
Thus for $n\geq n_3$, \eqref{est1}, \eqref{est2}, and \eqref{est3} yield
\begin{align}
\Big|\|\phi_n\|_{L^{p+1}_x}^{p+1}-\sum^{M}_{j=1}\|e^{-it^j_n\Delta}\psi^j\|_{L^{p+1}_x}^{p+1}-\|W^{M}_n\|_{L^{p+1}_x}^{p+1}\Big|\leq  \epsilon,
\end{align}
which concludes the proof.
\end{proof}


\begin{prop}[Nonlinear Profile decomposition]\label{nonradialPDNLS}
Let $\phi_n(x)$ be a uniformly bounded  sequence in $H^1(\Rn)$. Then for each $M\in\N$ there exists a subsequence of $\phi_n$, also denoted by  $\phi_n$,  for each $1 \leq j \leq M$, there exist a (same for all n) nonlinear profile $\tpsi^j$ in $H^1(\Rn)$, a sequence  of time shifts $t^j_n$, and a sequence  of space shifts $x^j_n$
and in addition, a sequence (in n) of remainders $\wt^M_n(x)$ in $H^1(\Rn)$,  such that 
\begin{equation}\label{NP}
\phi_n(x)=\sum^{M}_{j=1} \NLS(-t^j_n) \tpsi^j(x-x^j_n) + {\wt}^M_n(x), 
\end{equation}
where (as $n\to \infty$)
\begin{enumerate}
\item[(a)] for each j, either $t^j_n=0, t^j_n\to +\infty$ or $t^j_n\to -\infty$,
\item[(b)] if $t^j_n\to +\infty,$ then  $\|\NLS(-t)\tpsi^j\|_{\dB^0_{S([0,\infty);\dHs)}}<+\infty$ and if $t^j_n\to -\infty$, \\ then
  $\|\NLS(-t)\tpsi^j\|_{\dB^0_{S((-\infty,0];\dHs)}}<+\infty$,
\item[(c)] for $k \neq j$, then  $|t^j_n-t^k_n | + |x^j_n - x^k_n|\to+ \infty.$
\end{enumerate}

The remainder sequence has the following asymptotic smallness property:
\begin{align}
\label{smallnessPropNL}
\lim_{M \to \infty} \big( \lim_{n \to \infty} \| \NLS(t) \wt^M_n \|_{\dB^0_{\SHs}}\big ) = 0.
\end{align}
For fixed $M\in\N$ and any $0 \leq s \leq 1$, we have the asymptotic Pythagorean expansion
\begin{align}
\label{HPdecompNL}
\| \phi_n \|^2_{\dH^{s}} = \sum^M_{j=1} \|\NLS(-t^j_n) \tpsi^j \|^2_{\dH^{s}} + \|\wt^M_n \|^2_{\dH^{s}} +o_n(1)
\end{align}
and the energy Pythagorean decomposition (note that $E[\NLS(-t^j_n)\tpsi^j]=E[\tpsi^j]$):
\begin{align}
\label{pythagoreanNLS}
E[\phi_n]=\sum_{j=1}^M E[\tpsi^j]+E[\wt^M_n]+o_n(1).
\end{align}
\end{prop}
\begin{proof}  From Proposition \ref{nonradialPD}, given that $\phi_n(x)$ is a uniformly bounded  sequence in $H^1$, we have
\begin{align}
\label{linearflowPD}
\phi_n(x)=\sum^{M}_{j=1} e^{-it^j_n\Delta} \psi^j(x-x^j_n) + W^M_n(x)
\end{align}
satisfying \eqref{time-space-seq}, \eqref{smallnessProp}, \eqref{limHs} and \eqref{pythagorean}. We will choose $M\in\N$ later. To prove this proposition, the idea is  to replace a linear flow $e^{it\Delta}\psi^j$ by some nonlinear flow.

For each $\psi^j$ we can apply the wave operator (Proposition \ref{Existence of wave operator}) to obtain a function $\tpsi^j\in H^1,$ which we will refer to as the nonlinear profile (corresponding to the linear profile $\psi^j$) such that the following properties hold:

For a given $j,$ there are two cases to consider: either $t^j_n$ is bounded, or $|t^j_n|\to +\infty.$
\bigskip

\noindent\emph{ Case $|t^j_n|\to +\infty$:} 

If $t^j_n\to +\infty,$  Proposition \ref{Existence of wave operator} Part I equation \eqref{wave1} implies that
\begin{align*}
\|\NLS(-t^j_n)\tpsi^j-e^{-it^j_n\Delta}\psi^j\|_{H^1}\to 0\quad \mbox{    as   } \quad t^j_n\to +\infty
\end{align*}
and so
\begin{align}\label{cotacriticalnorm}
\|\NLS(-t)\tpsi^j\|_{\dB^0_{S([0,+\infty),\dHs)}}<+\infty.
\end{align}
\noindent Similarly, if 
$t^j_n\to -\infty$, by \eqref{wave2} we obtain 
\begin{align*}\|\NLS(-t^j_n)\tpsi^j-e^{-it^j_n\Delta}\psi^j\|_{H^1}\to 0\quad \mbox{    as   } \quad t^j_n\to -\infty,
\end{align*}
and hence,
\begin{align}\label{negcriticalnorm}
\|\NLS(-t)\tpsi^j\|_{\dB^0_{S((-\infty,0],\dHs)}}<+\infty.
\end{align}

\noindent\emph{ Case  $t^j_n$ is bounded} (as $n\to\infty$): Adjusting the profiles $\psi^j$ we reduce it to the  case $t^j_n=0$. Thus, \eqref{time-space-seq}  becomes $|x^j_n-x^k_n|\to +\infty$ as $n\to\infty$, and continuity of the \emph{linear} flow in $H^1$, leads 
to $e^{-t_n^j\Delta}\psi^j \to \psi^j$ strongly in $H^1$ as $n \to \infty$. In this case, we simply let 
\begin{align*}
\tilde{\psi}^j=\NLS(0)e^{-i(\lim_{n\to\infty}t_n^j)\Delta}\psi^j=e^{-i0\Delta}\psi^j=\psi^j. 
\end{align*}

Thus, in either case of sequence $\{t_n^j\}$, we have a new nonlinear profile ${\tpsi^j}$ associated to each original linear profile $\psi^j$ such that 
\begin{equation}\label{flow approx}
\|\NLS(-t^j_n)\tpsi^j-e^{-it^j_n\Delta}\psi^j\|_{H^1}\to 0\quad\text{~~as~~}\quad n\to +\infty.
\end{equation}
Thus, we can substitute $e^{-it^j_n\Delta}\psi^j$ by $\NLS(-t^j_n)\tilde{\psi^j}$ in \eqref{linearflowPD} to obtain
\begin{align}
\phi_n(x)=\sum^{M}_{j=1} \NLS(-t^j_n) \tpsi^j(x-x^j_n) + \wt^M_n(x),\label{profile+residuo}
\end{align}
where
\begin{align}
{\wt}^M_n(x)&={W}^M_n(x)+%
\sum_{j=1}^M\big\{
e^{-it^j_n\Delta}\psi^j(x-x^j_n)-\NLS(-t^j_n)\tpsi^j(x-x^j_n)\big\}\equiv {W}^M_n(x)+%
\sum_{j=1}^M \at^j.
\label{residuo}
\end{align}
The triangle inequality yields 
$
\|e^{it\Delta}{\wt}^M_n\|_{\dB^0_{\SHs}}\leq\|e^{it\Delta}W^M_n\|_{\dB^0_{\SHs}}+
c\sum_{j=1}^M\big\|  \at^j
\big\|_{\dB^0_{\SHs}}.
$
By \eqref{flow approx} we have that
$
\|e^{it\Delta}{\wt}^M_n\|_{\dB^0_{\SHs}}\leq\|e^{it\Delta}W^M_n\|_{\dB^0_{\SHs}}+
c\sum_{j=1}^Mo_n(1),
$
and thus, 
\begin{align*}
\lim_{M \to \infty} \big( \lim_{n \to \infty} \| e^{it\Delta} \wt^M_n \|_{\dB^0_{\SHs}}\big ) = 0.
\end{align*}
Now we are going to apply a nonlinear flow to $\phi_n(x)$ and approximate it by a combination of ``nonlinear bumps" $\NLS(t-t^j_n)\tpsi^j(x-x^j_n),\;$ i.e., 
$$
\NLS(t)\phi_n(x)\approx \sum^{M}_{j=1} \NLS(t-t^j_n) \tpsi^j(x-x^j_n).
$$
Obviously, this can not hold for any bounded in $H^1$ sequence $\{\phi_n\}$, since, for a example, a nonlinear flow can introduce finite time blowup solutions. However, under the proper conditions we can use  the long term perturbation theory (Proposition \ref{longperturbation}) to guarantee that a nonlinear flow behaves basically similar to the linear flow.

To simplify notation, introduce the nonlinear evolution of each separate initial condition $u_{n,0}=\phi_n$:
$u_n(t,x)=\NLS(t)\phi_n(x),\;$
 the nonlinear evolution of each separate nonlinear profile (``bump"):
$\;v^j(t,x)=\NLS(t)\tpsi^j(x), \;$  
and  a linear sum of nonlinear evolutions of ``bumps":
$\tu_n(t,x)=\sum_{j=1}^M v^j(t-t^j_n,x-x^j_n).$ 

Intuitively, we think that $\phi_n=u_{n,0}$ is a sum of bumps $\tpsi^j$ (appropriately transformed) and $u_n(t)$ is a nonlinear evolution of their entire sum. On the other hand, $\tu_n(t)$ is a sum of nonlinear evolutions of each bump so we now want to compare $u_n(t)$ with $\tu_n(t)$. 

Note that if we had just the linear evolutions, then both $u_n(t)$ and $\tu_n(t)$ would be the same. 

Thus, $u_n(t)$ satisfies 
 $$
i\partial_t u_n+\Delta u_n+| u_n|^{p-1} u_n=0,
$$
and $\tu_n(t)$ satisfies 
 $$
i\partial_t \tu_n+\Delta \tu_n+|\tu_n|^{p-1}\tu_n={\te^M_n},
$$
where 
$$
{\te^M_n}=|\tu_n|^{p-1}\tu_n-\sum_{j=1}^M |v^j_n(t-t^j_n,\cdot-x^j_n)|^{p-1} v^j_n(t-t^j_n,\cdot-x^j_n).
$$
{\claim There exists a constant $A$ independent of $M$, and for every $M$, there exists $n_0=n_0(M)$ such that if $n>n_0$,  then
$
\|\tu_n\|_{\dB^0_{\SHs}}\leq A.
$ \label{claim 1} }
{\claim  For  
each $M$ and $\epsilon>0$, there exists  $n_1=n_1(M, \epsilon)$ such that if $n>n_1$, then
$
\|{\te}^M_n\|_{\dB^0_{\SdHs}}
\leq \epsilon.
$ \label{claim 2}}

Note $\tu_n(0,x)-u_n(0,x)=\wt_n^M(x)$. 
Then for any $\tilde \epsilon>0$ there exists $M_1=M_1(\tilde \epsilon)$ large enough such that for each $M>M_1$ there exists $n_2=n_2(M)$ with $n>n_2$ implying 
$$
\|e^{it\Delta}(\tu_n(0)-u_n(0))\|_{\dB^0_{\SHs}}\leq\tilde \epsilon.
$$
Therefore, for $M$ large enough and $n=\max(n_0,n_1,n_2)$, since
 $$
 e^{it\Delta}(\tu_n(0))=e^{it\Delta}\Bigg(\sum_{j=1}^M v^j(-t^j_n,x-x^j_n)\Bigg),
 $$ 
 which are scattering by \eqref{flow approx},
Proposition \ref{longperturbation} implies 
$\|u_n\|_{\dB^0_{\SHs}}< +\infty$, a  contradiction. 

Coming back to the nonlinear remainder $\wt^M_n,$ we estimate its nonlinear flow as follows (recall the notation of $\wt^M_n$, $\; W^M_n$ and $\at^j$ in \eqref{residuo}): 

By Besov Strichartz estimates \eqref{eq:In-Besov-Strichartz} and by the triangle inequality, we get 
\begin{align}
\|\NLS(t)\wt^M_n\|_{\dB^0_{\SHs}}&
\leq\|e^{it\Delta}\wt^M_n\|_{\dB^0_{\SHs}}+\left\|\left|\wt^M_n\right|^{p-1}\wt^M_n\right\|_{\dB^0_{\SdHs}}\notag\\
&\leq\|e^{it\Delta}\wt^M_n\|_{\dB^0_{\SHs}}+c\sum_{j=1}^M \|\at^j\|_{\dB^0_{\SHs}}^{p-1} \|\at^j\|_{\dB^{s}_{\SLt}}\label{pb1}\\
&\leq\|e^{it\Delta}\wt^M_n\|_{\dB^0_{\SHs}}+c\sum_{j=1}^M \|\at^j\|_{\dB^0_{\SHs}}^{p-1} \|\at^j\|_{\dB^0_{S(\dH^1)}}\label{pb2}.
\end{align}
We used \eqref{small besov} to obtain \eqref{pb1} and since $s<1$ we have $\;\dH^1\hookrightarrow\dH^s,\;$ so it yields \eqref{pb2}.
Hence,
\begin{align*}
\|\NLS(t)\wt^M_n\|_{\dB^0_{\SHs}}&
\leq\|e^{it\Delta}\wt^M_n\|_{\dB^0_{\SHs}}+c\sum_{j=1}^M \big\|
e^{-it^j_n\Delta}\psi^j-\NLS(-t^j_n)\tpsi^j\big\|_{H^1}^{p}
\end{align*}
and by \eqref{flow approx} and then applying \eqref{smallnessProp}, we obtain 
$
\lim_{n \to \infty} \| e^{i t \Delta}W^M_n \|_{\dB^0_{\SHs}}\to 0\quad \text{~~ as~~}\quad M\to \infty.
$
Thus we proved \eqref{profile+residuo}, \eqref{smallnessPropNL}. This also gives \eqref{HPdecompNL}.

Next,  we substitute the linear flow in  Lemma \ref{energy Pythagorean expansion} by the nonlinear  
and repeat the above long term perturbation argument
to obtain
\begin{align}
\label{aproxL6NL}
\| \phi_n \|^{p+1}_{L^{p+1}} = \sum^M_{j=1} \|\NLS(-t^j_n) \psi^j \|^{p+1}_{L^{p+1}} + \|\wt^M_n \|^{p+1}_{L^{p+1}} +o_n(1),
\end{align} 
which yields the energy Pythagorean decomposition \eqref{pythagoreanNLS}. The proof will be concluded after we prove the Claims \ref{claim 1} and \ref{claim 2}.

\noindent\emph{ Proof of Claim \ref{claim 1}}.
We show that for a large constant $A$ independent of $M$ and if $n>n_0=n_0(M)$, then 
\begin{align}\label{boundsol}
\|\tu_n\|_{S(\dHs)}\leq A.
\end{align}
 
 Let $M_0$ be a large enough such that
$\|e^{it\Delta}\wt^{M_0}_n\|_{\SHs}\leq \delta_{sd}.$   
Then, by \eqref{residuo}, for  each $j>M_0,$ we have $\|e^{it\Delta} \psi^j\|_{\SHs}\leq \delta_{sd},$ thus, Proposition \ref{Existence of wave operator} yields  
$
\|v^j\|_{\SHs}\leq 2\|e^{it\Delta} \psi^j\|_{\SHs}$     for  $ j>M_0.
$

Assume both $s\neq \frac12$ and $d\neq2$,  the pairs $\;\left(\profq,\profq\right)$, $\;\left(\infty,\frac{2 d}{d - 2 s}\right)$, $\;\left(\quhtq, \ruhr\right)\;$ and $\;\left(\qdhtq, \rdhr\right)$, are $\dH^s$ admissible.
Hence, we have 
\begin{align}
\notag
\|\tu_n&\|^{\profq}_{L_t^{\profq}L^{\profq}_x}=\\
 =&\sum_{j=1}^{M_0}\|v^j\|^{\profq}_{L_t^{\profq}L^{\profq}_x}
 +\sum_{j=M_0+1}^{M}\|v^j\|^{\profq}_{L_t^{\profq}L^{\profq}_x} + \mbox{cross terms} 
 \notag
 \\
\leq&\sum_{j=1}^{M_0}\|v^j\|^{\profq}_{L_t^{\profq}L^{\profq}_x}
+2^{\profq}\sum_{j=M_0+1}^{M}\|e^{it\Delta}  \psi^j\|^{\profq}_{L_t^{\profq}L^{\profq}_x} + \mbox{cross terms,}
\label{sumsNLS}
\end{align}
note that by \eqref{linearflowPD} we have 
\begin{align}
\|&e^{it\Delta}  \phi_n\|^{\profq}_{L_t^{\profq}L^{\profq}_x}=\notag
\\&
=\sum_{j=1}^{M_0}\|e^{it\Delta}  \psi^j\|^{\profq}_{L_t^{\profq}L^{\profq}_x} +2^{\profq}\sum_{j=M_0+1}^{M}\|e^{it\Delta}  \psi^j\|^{\profq}_{L_t^{\profq}L^{\profq}_x} + \mbox{cross terms.} \label{sumlinear}
\end{align}
Observe that by \eqref{time-space-seq} and taking $n_0=n_0(M)$ large enough, we can consider $\{u_n\}_{n>n_0}$ and thus, make ``the cross terms" $\leq1$. Then \eqref{sumlinear} and
 $\|e^{it\Delta}  \phi_n\|_{L_t^{\profq}L^{\profq}_x}\leq c \|\phi_n\|_{\dHs}\leq c_1$ imply $\sum_{j=M_0+1}^{M}\|e^{it\Delta}  \psi^j\|^{\profq}_{L_t^{\profq}L^{\profq}_x} $ is bounded independent of $M$ provided $n>n_0$. If $n>n_0,$ then $\|\tu_n\|_{L_t^{\profq}L^{\profq}_x}$ is also bounded independent of $M$ by \eqref{sumsNLS}. 

 In a similar fashion, one can prove that $\|\tu_n\|_{L_t^{\infty}L^{\frac{2 d}{d - 2 s}}_x}$ is bounded independent of $M$ provided $n>n_0$. Interpolation between $\|\tu_n\|_{L_t^{\profq}L^{\profq}_x}$ and $\|\tu_n\|_{L_t^{\infty}L^{\frac{2 d}{d - 2 s}}_x}$ gives $\|\tu_n\|_{\quht \ruh}$ and $\|\tu_n \|_{\qdht\rdh}$   are both  bounded independent of $M$ for $n>n_0$. 

 When $s=\frac12$ and $d=2$, the previous argument takes the pair $(2,\infty)$ which is not an admissible pair in dimension 2. Instead  we estimate $\|\tu\|_{L^8_x,L^8_x}$ and $\|\tu\|_{L^\infty_x,L^4_x}$,   and interpolate between them to get that $\|\tu\|_{L^{12}_x,L^6_x}$ is bounded independent of $M$ provided $n>n_0.$
 
  To close the argument, we  apply Kato estimate \eqref{eq:Kato-Strichartz} to the integral equation of  
 $$
i\partial_t \tu_n+\Delta \tu_n+|\tu_n|^{p-1}\tu_n={\te^M_n}. 
$$
Claiming   $\|\te_n^M\|_{\dB^0_{S'(\dH^{-s})}}\leq 1$ (see Claim \ref{claim 2}) , as in Proposition \ref{longperturbation}, we obtain that $\|\tu_n\|_{\dB^0_{\SHs}}$ is as well bounded independent of $M$ provided $n>n_0.$ Thus, Claim \ref{claim 1} is proved.  
\\

\noindent\emph{ Proof of Claim \ref{claim 2}.} Note that the pairs $(\frac{6}{1-s},\frac{6 d}{3 d-4s-2})$, $(\frac{4}{1- s},\frac{2 d}{d -s-1})$ are $\Hs$ admissible and the pair $(\frac{12 (d - 2 s)}{(8 + 3 d - 6 s) (1 - s)},\frac{6 d (d - 2 s)}{  3 (d^2+ 2s^2)+ 9 d (1 - s) - 2(5 s + 4)})$ is $\dH^{-s}$ admissible.
Recall the elemental inequality: for $a_j, a_k \in \C,$ 
$$\Bigg|\bigg|\sum^M_{j=1}a_j\bigg|^{p-1}\sum^M_{k=1}a_k- \sum^M_{j=1}|a_j|^{p-1}a_j\Bigg| \leq c_{p,M} \sum_{j=1}^M\sum _{\substack{k=1\\k\neq j}}^M |a_k|^{p-1}|a_j|,$$
 which combined with the H\"older's inequality, for each dyadic number $N\in 2^{\Z}$, leads to
\begin{align*}
\|\te_n^M&\|_{S'(\dH^{-s})}\leq \|{\te}^M_n\|_{\qzht \rzh}\\
&\leq \sum_{j=1}^M\sum _{\substack{k=1\\k\neq j}}^M \|v^{k}(t-t^{k}_n,x-x^{k})\|^{p-1}_{\quht \ruh}\|v^{j}(t-t^{j}_n,x-x^{j})\|_{\qdht\rdh}.
\end{align*}
Here, we used the following H\"older splits: 
$$\frac{(p-1)(1-s)}{6}+\frac{1-s}{4}=\frac{(8 + 3 d - 6 s) (1 - s)}{12 (d - 2 s)},$$
$$\frac{(p-1)(3 d-4s-2)}{6 d}+\frac{d -s-1}{2 d}=\frac{  3 (d^2+ 2s^2)+ 9 d (1 - s) - 2(5 s + 4)}{6 d (d - 2 s)}.$$
 Note that either   $\{t_n^{k}\} \to \pm \infty$ or $\{t_n^{k}\}$ is bounded. 

If $\{t_n^{j}\} \to \pm \infty$, without loss of generality assume $|t_n^{k}-t_n^{j}|\to\infty$ as $n\to \infty$ and by adjusting the profiles that $|x_n^{k}-x_n^{j}|\to0$  as $n\to \infty$. Since $v^{k}\in{\quht \ruh}$ and $v^{j_2}\in \qdh\rdh$, then
$$\|v^{k}(t-t^{k}_n,x-x^{k})\|^{p-1}_{\quht \ruh}\|v^{j}(t-t^{j}_n,x-x^{j})\|_{\qdht\rdh}\to 0.$$ 

If $\{t_n^{j}\}$ is bounded, without loss of generality, assume $|x_n^{j}-x_n^{k}|\to\infty$ as $n\to \infty,$ then 
$$\|v^{k}(t-t^{k}_n,x-x^{k})\|^{p-1}_{\quht \ruh}\|v^{j}(t-t^{j}_n,x-x^{j})\|_{\qdht\rdh}\to 0.$$ 
Thus, in either case we obtain Claim \ref{claim 2}.

This finishes the proof of Proposition \ref{nonradialPDNLS}
\end{proof}

Observe that \eqref{HPdecompNL} gives $\dH^1$ asymptotic orthogonality at $t=0$ and the following lemma extends it to the bounded NLS flow for $0\leq t\leq T.$ 

\begin{lemma}[$\dH^1$ Pythagorean decomposition along the bounded  NLS flow]\label{HPdecompNLf} Suppose $\phi_n$ is a bounded sequence in $H^1(\Rn)$. Let $T\in (0,\infty)$ be a fixed  time. Assume that $u_n(t)\equiv \NLS(t)\phi_n$ exists up to time $T$ for all $n,\;$ and $\lim_{n\to \infty}\|\nabla u_n(t)\|_{L^\infty_{[0,T]} \Lt_x}<\infty.$  Consider the nonlinear profile decomposition from Proposition \ref{nonradialPDNLS}. Denote $W^M_n(t)\equiv \NLS(t)W^M_n$. Then for all $j$, the nonlinear profiles $v^j(t)\equiv \NLS(t)\tpsi^j$ exist up to time T and for all $t\in[0,T],$
\begin{align}
\label{PdecompNLS}
\|\nabla u_n(t)\|^2_{\Lt}=\sum_{j=1}^{M} \|\nabla v^j(t-t^j_n)\|^2_{\Lt}+\|\nabla W^M_n(t)\|^2_{\Lt_x}+o_n(1),
\end{align}
where $o_n(1)\to 0$ uniformly on $0\leq t\leq T.$
\end{lemma}

\begin{proof}
We use Propositon \ref{nonradialPDNLS} to obtain profiles $\{\tpsi^j\}$ and the nonlinear profile decomposition \eqref{NP}. Note that $\lim_{n\to \infty}\| \NLS(t)W^M_n\|_{\dB^0_{\SHs}}\to0$ as $M\to\infty$, 
so by choosing a large M we can make $\| \NLS(t)W^M_n\|_{\dB^0_{\SHs}}$ small. 

Let $M_0$ be such that for $M\geq M_0$ (and for $n$ large), we have $\| \NLS(t)W^M_n\|_{\dB^0_{\SHs}} \leq \delta_{sd}$ (recall $\delta_{sd}$ from Proposition \ref{small data}). Reorder the first $M_0$ profiles and let $M_2,$  $0\leq M_2\leq M,$ be  such that 
\begin{enumerate}
\item For each $1\leq j\leq M_2,$ we have $t_n^j=0.$ Observe that if $M_2=0,$  there are no $j$ in this case.
\item For each $M_2+1\leq j\leq M_0,$ we have $|t_n^j|\to \infty.$ If $M_2=M_0,$ then it means that  there are no $j$ in this case.
\end{enumerate}
From Proposition \ref{nonradialPDNLS} and the profile decomposition \eqref{NP} we have that $v^j(t)$ for $j>M_0$ are scattering, and  for $M_2+1\leq j\leq M_0$  we have $\|v^j(t-t^j_n)\|_{S(\dHs;[0,T])}\to 0$ as $n\to +\infty$. 

  In fact, taking $t_n^j \to +\infty$ and
$\|v^j(-t)\|_{S(\dot H^{s}; [0,+\infty))}<\infty$, dominated convergence leads to $\|v^j(-t)\|_{L_{[0,+\infty)}^q L_x^r}<\infty$, for $q<\infty$,  where $(r,q)$ is an $\dHs$ admissible pair, and consequently,
$\|v^j(t-t_n^j)\|_{L_{[0,T]}^q L_x^r} \to 0$ as $n\to\infty$.  As   $v^j(t)$ has been constructed via the existence of wave operators to converge in $H^1$ to a linear flow,   the $L^r_x$  decay of the linear flow 
$$\|v^j(t-t^j)\|_{L^\infty_{[0,T]}L^{r}_x}\to 0,$$
 with 
$$r=\left\{\begin{array}{cccc}
\frac{2d}{d-2s} &  & d\geq3& 
\\\frac{2}{1-s} &  & d=2 &\\\frac{2}{1-2s} &  & d=1&
\end{array}\right. \quad\quad \text{and~~} s \text{~~as in~~} \eqref{casos dim}.$$

Let $B=\max\{1,\lim_{n}\|\nabla u_n(t)\|_{L^\infty_{[0,T]}\Lt_x}\}<\infty$. For each $1\leq j\leq M_2$, 
let $T^j\leq T$ be the maximal forward time  such that $\|\nabla v^j\|_{L^\infty_{[0,T^j]}\Lt_x}\leq 2B$,
 and $\tilde T=\min_{1\leq j\leq M_2}T^j$ or  $\tilde T=T$ if $M_2=0.$  It is sufficient to prove that 
 \eqref{PdecompNLS} holds for $\tilde T=T$, since for each $1\leq j \leq M_2$, we have $T^j=T,\;$ 
 and  therefore, $\tilde T=T.$ Thus, let's consider $[0,\tilde T]$. For each $1\leq j\leq M_2$,  we have
 \noindent for $d\geq3$:
\begin{align}
\|v^j(t)\|_{S(\dHs;[0,\tilde T])}
&\lesssim \|v^j\|_{L^{\frac {2}{1-s}}_{[0,\tilde T]}L^{\frac{2d}{d-2}}_x}
+\|v^j\|_{L^{\infty}_{[0,\tilde T]}L^{\frac{2 d}{d - 2 s}}_x}\label{line1}
\\
&\lesssim \|v^j\|_{L^{\frac {2}{1-s}}_{[0,\tilde T]}L^{\infty}_x}
\|v^j\|_{L^{\infty}_{[0,\tilde T]}L^{\frac{2d}{d-2}}_x}
+\|v^j\|^{1-s}_{L^{\infty}_{[0,\tilde T]}L^{2}_x}
\|v^j\|^s_{L^{\infty}_{[0,\tilde T]}L^{\frac{2d}{d-2}}_x}\label{line2}
\\
&\lesssim \big(\tilde T^{\frac{1-s}{2}}+c^{1-s}\big)\|\nabla v^j\|_{L^{\infty}_{[0,\tilde T]}L^{2}_x}
\lesssim \langle\tilde T^{\frac{1-s}{2}}\rangle B \label{line4},
\end{align}
note that  \eqref{line1} comes from  the ``end point" admissible $S(\dHs)$ Strichartz norms ($L^{\frac{2}{1-s}}_{t}L^{\frac{2d}{d-2s}}_x$ and $L^{\infty}_{t}L^{\frac{2 d}{d - 2 s}}_x$), since  all other $S(\dHs)$ norms will be bounded by interpolation; the H\"older's inequality  yields \eqref{line2} and  the Sobolev's embedding $\dH^1(\Rn)\hookrightarrow L^{\frac{2 d}{d - 2 s}}(\Rn)$ together with $\|v^j\|_{L^{\infty}_{[0,\tilde T]}L^{2}_x}=\|\psi^j\|_{L^{2}_x}\leq  \|\phi_n\|_{L^{2}}$, from \eqref{HPdecompNL} with $s=0$, 
gives \eqref{line4}.

\noindent For  $d=2$:  
\begin{align}
\|v^j(t)\|&_{S(\dHs;[0,\tilde T])}
\lesssim \|v^j\|_{L^{\infty}_{[0,\tilde T]}L^{\frac{2}{1-s}}_x}+
\|v^j\|_{L^{\frac {2}{1-s}}_{[0,\tilde T]}L^{r}_x}\label{linea1}
\\
&\lesssim \|v^j\|^{1-s}_{L^{\infty}_{[0,\tilde T]}L^{2}_x}
\|v^j\|_{L^{\infty}_{[0,\tilde T]}L^{\infty}_x}+
\|v^j\|^{1-s}_{L^{2}_{[0,\tilde T]}L^{\infty}_x}
\|v^j\|_{L^{\infty}_{[0,\tilde T]}L^{r
}_x}\label{linea2}
\\
&\lesssim \|v^j\|^{1-s}_{L^{\infty}_{[0,\tilde T]}L^{2}_x}
\|\nabla v^j\|_{L^{\infty}_{[0,\tilde T]}L^{2}_x}+
\|v^j\|^{1-s}_{L^{2}_{[0,\tilde T]}L^{\infty}_x}
\|v^j\|_{L^{\infty}_{[0,\tilde T]}\dH^{1-
\frac 2{r}
}_x}\label{linea2a}
\\
&\lesssim \left(\|v^j\|^{1-s}_{L^{\infty}_{[0,\tilde T]}L^{2}_x}
+
\|v^j\|^{1-s}_{L^{2}_{[0,\tilde T]}L^{\infty}_x}
\right)\|\nabla v^j\|_{L^{\infty}_{[0,\tilde T]}L^{2}_x}
\label{linea3}
\\
&\lesssim \big(\tilde T^{\frac{1-s}{2}}+c^{1-s}\big)\|\nabla v^j\|_{L^{\infty}_{[0,\tilde T]}L^{2}_x}\lesssim \langle\tilde T^{\frac{1-s}{2}}\rangle B\label{linea4},
\end{align}
where $r=\big( \big(\frac{2}{1-s}\big)^+\big)'$.
Note that  \eqref{linea1} comes from  the ``end point" admissible  Strichartz norms ($L^{\infty}_{t}L^{\frac{2}{1-s}}_x$ 
and $L^{\frac {2}{1-s}}_{t}L^{r}_x$); 
 H\"older's inequality yields \eqref{linea2};  the Sobolev's embeddings $\dH^1(\R^2)\hookrightarrow L^\infty(\R^2)$  and $\dH^{1-\frac{2}{r}}(\R^2)\hookrightarrow L^{r}(\R^2)$ leads to \eqref{linea2a}; since $r$ is large we have the Sobolev's embedding $\dH^{1}(\R^2)\hookrightarrow \dH^{1-\frac{2}{r}}(\R^2),\;$ which implies \eqref{linea3}, and finally,
since  $\|v^j\|_{L^{\infty}_{[0,\tilde T]}L^{2}_x}=\|\psi^j\|_{L^{2}_x}\leq  \|\phi_n\|_{L^{2}}$ by \eqref{HPdecompNL} with $s=0$ 
we get \eqref{linea4}.

\noindent For  $d=1$:  
\begin{align}
\|v^j(t)\|&_{S(\dHs;[0,\tilde T])}
\lesssim \|v^j\|_{L^{\infty}_{[0,\tilde T]}L^{\frac{2}{1-2s}}_x}+
\|v^j\|_{L^{\frac {4}{1-2s}}_{[0,\tilde T]}L^{\infty}_x}\label{linea1d1}
\\
&\lesssim \|v^j\|^{1-2s}_{L^{\infty}_{[0,\tilde T]}L^{2}_x}
\|v^j\|_{L^{\infty}_{[0,\tilde T]}L^{\infty}_x}+
\|v^j\|^{\frac{1-2s}{2}}_{L^{2}_{[0,\tilde T]}L^{\infty}_x}
\|v^j\|_{L^{\infty}_{[0,\tilde T]}L^{\infty
}_x}\label{linea2d1}
\\
&\lesssim \left(\|v^j\|^{1-2s}_{L^{\infty}_{[0,\tilde T]}L^{2}_x}
+
\|v^j\|^{\frac{1-2s}2}_{L^{2}_{[0,\tilde T]}L^{\infty}_x}
\right)\|\nabla v^j\|_{L^{\infty}_{[0,\tilde T]}L^{2}_x}\notag
\\
&\lesssim \big(\tilde T^{\frac{1-2s}{2}}+c^{\frac{1-2s}{2}}\big)\|\nabla v^j\|_{L^{\infty}_{[0,\tilde T]}L^{2}_x}\lesssim \langle\tilde T^{\frac{1-2s}{2}}\rangle B\label{linea4d1},
\end{align}
note that  \eqref{linea1d1} comes from  the ``end point" admissible  Strichartz norms ($L^{\infty}_{t}L^{\frac{2}{1-2s}}_x$ 
and $L^{\frac {4}{1-2s}}_{t}L^{\infty}_x$); 
 H\"older's inequality yields \eqref{linea2d1};  the Sobolev's embeddings $\dH^1(\R^1)\hookrightarrow L^\infty(\R^1)$  implies \eqref{linea2d1}, and finally,   
  $\|v^j\|_{L^{\infty}_{[0,\tilde T]}L^{2}_x}=\|\psi^j\|_{L^{2}_x}\leq  \|\phi_n\|_{L^{2}}$ leads to  \eqref{linea4d1}.

As in the proof of Proposition \ref{nonradialPDNLS}, set  
$\tu_n(t,x)=\sum_{j=1}^M v^j(t-t^j_n,x-x^j_n)$
 and, a linear sum of nonlinear flows of nonlinear profiles $\tpsi^j,\;$  $\te^M_n=i\partial_t \tu_n+\Delta \tu_n+|\tu_n|^{p-1}\tu_n.$ Thus,  for $M>M_0$ we have
 
 \noindent\emph{ Claim \ref{claim 1}}: There exist a constant $A=A(\tilde T)$ independent of $M$, and for every $M$, there exists $n_0=n_0(M)$ such that if $n>n_0$,  then
$
\|\tu_n\|_{\dB^0_{\SHs}}\leq A.
$ 

\noindent\emph{ Claim \ref{claim 2}}: For each $M$ and $\epsilon>0$, there exists  $n_1=n_1(M, \epsilon)$ such that if $n>n_1$, then
$\|{\te}^M_n\|_{\dB^0_{\SdHs}}.
$

\begin{rmk}\label{adapting long time}
Note  since $u(0)-\tu_n(0)=\wt^M_n$, there  exists $M'=M'(\epsilon)$ large enough so that for each $M>M'$  there exists  $n_2=n_2(M)$ such  that $n>n_2$ implies 
$$\|e^{it\Delta}(u(0)-\tu_n(0))\|_{\dB^0_{S(\dHs;[0,\tilde T])}}\leq \epsilon.$$
\end{rmk}

We will next apply the long term perturbation argument (Proposition \ref{longperturbation}); note that in Proposition \ref{longperturbation}, $T=+\infty$, while here, it is not necessary. However,  $T$ does not form part of the parameter
dependence, since $\epsilon_0$ depends only on $A=A(T)$, not on $T$, that is, there will be dependence on $T$, but it is only through $A$.

Thus, the long term perturbation argument  (Proposition \ref{longperturbation}) gives us $\epsilon_0=\epsilon_0(A).$
 Selecting an arbitrary $\epsilon\leq \epsilon_0,$ and from Remark \ref{adapting long time} take $M'=M'(\epsilon)$.  Now select an arbitrary $M>M'$ and take $n'=\max(n_0,n_1,n_2)$. Then combining claims \ref{claim 1} - \ref{claim 2}, Remark \ref{adapting long time} and Proposition \ref{nonradialPDNLS}, we obtain that for $n>n'(M,\epsilon)$ with $c=c(A)=c(\tilde T)$ we have
 \begin{align}\label{aprox}
\|u_n-\tu_n\|_{S(\dHs;[0,\tilde T])}\leq c(\tilde T)\epsilon.
\end{align}

We will next prove \eqref{PdecompNLS} for $0\leq t\leq\tilde T$.
Recall that for each dyadic number $N\in 2^{\Z},\;$ $\|v^j(t-t^j_n)\|_{S(\dHs;[0,\tilde T])}\to 0$ as $n\to \infty$ and  for each $1\leq j\leq M_2$, we have $\|\nabla v^j\|_{L^\infty_{[0,T^j]}\Lt_x}\leq 2B$.  Strichartz estimates imply 
$\|\nabla v^j(t-t^j_n)\|_{L^{\infty}_{[0,\tilde T])}\Lt_x}\lesssim \|\nabla v^j(-t^j_n)\|_{L^{\infty}_{[0,\tilde T])}\Lt_x},$
then
\begin{align*}
\|\nabla \tu(t)\|_{L^{\infty}_{[0,\tilde T]}\Lt_x}^2&
=\sum_{j=1}^{M_2}\|\nabla v^j(t)\|_{L^{\infty}_{[0,\tilde T]}\Lt_x}^2
+\sum_{j=M_2+1}^{M}\|\nabla v^j(t-t^j_n)\|_{L^{\infty}_{[0,\tilde T]}\Lt_x}^2+o_n(1)\\
&\lesssim M_2B^2+\sum_{j=M_2+1}^{M}\|\nabla \NLS(-t^j_n)\psi^j\|_{\Lt_x}^2+o_n(1)\\
&\lesssim M_2B^2+\|\nabla\phi_n\|_{\Lt_x}^2+o_n(1)
\lesssim M_2B^2+B^2+o_n(1).
\end{align*}
Using  \eqref{aprox}, we obtain
\noindent for  $d\geq3$:
\begin{align}
\|u_n-\tu_n\|_{L^\infty_{[0,\tilde T]}L^{p+1}_x}&
\lesssim \|u_n-\tu_n\|_{L^\infty_{[0,\tilde T]}L^{\frac{2 d}{d - 2 s}}_x}^{\frac{2}{d - 2 s+2}}
\|u_n-\tu_n\|_{L^\infty_{[0,\tilde T]}L^{\frac{2d}{d-2}}_x}^{\frac{d - 2 s}{d - 2 s+2}}
\label{aproxlinea1a}
 \\
 &\lesssim \|u_n-\tu_n\|_{S(\dHs;[0,\tilde T])}
 ^{\frac{2}{d - 2 s+2}}
\|\nabla(u_n-\tu_n)\|_{L^\infty_{[0,\tilde T]}L^{2}_x}^{\frac{d - 2 s}{d - 2 s+2}}
\label{aproxlinea2a}\\
&\lesssim c(\tilde T)^{\frac{2}{d - 2 s+2}}(M_2B^2+B^2+o(1))^{\frac{d - 2 s}{d - 2 s+2}}\epsilon^{\frac{2}{d - 2 s+2}}, \notag
\end{align}
in this case, we used H\"older's inequality to get \eqref{aproxlinea1a} and the Sobolev embedding $\;\dH^1(\Rn)\hookrightarrow L^{\frac{2 d}{d - 2 }}(\Rn)\;$  to obtain \eqref{aproxlinea2a}.

\noindent For  $d=2$:  
\begin{align}
\|u_n-\tu_n\|_{L^\infty_{[0,\tilde T]}L^{p+1}_x}
&\lesssim 
\|u_n-\tu_n\|_{L^\infty_{[0,\tilde T]}L^{\frac{2}{1-s}}_x}^{\frac{1}{2-s}}
\|u_n-\tu_n\|_{L^\infty_{[0,\tilde T]}L^{\infty}_x}^{\frac{1-s }{2-s}}\label{aproxlinea1} \\
&\lesssim 
\|u_n-\tu_n\|_{S(\dHs;[0,\tilde T])}
^{\frac{1}{2-s}}
\|\nabla(u_n-\tu_n)\|_{L^\infty_{[0,\tilde T]}L^{2}_x}^{\frac{1-s }{2-s}}\label{aproxlinea2} \\
&\lesssim c(\tilde T)^{\frac{1}{2-s}}
(M_2B^2+B^2+o(1))^{\frac{1-s }{2-s}}
\epsilon^{\frac{1}{2-s}},
\notag
\end{align}
here, we used H\"older's inequality to get \eqref{aproxlinea1} and the Sobolev embedding $\dH^1(\R^2)\hookrightarrow L^\infty(\R^2)\;$ to obtain \eqref{aproxlinea2}.

\noindent For  $d=1$:  
\begin{align}
\|u_n-\tu_n\|_{L^\infty_{[0,\tilde T]}L^{p+1}_x}
&\lesssim 
\|u_n-\tu_n\|_{L^\infty_{[0,\tilde T]}L^{\frac{4}{1-2s}}_x}^{\frac{2}{3-2s}}
\|u_n-\tu_n\|_{L^\infty_{[0,\tilde T]}L^{\infty}_x}^{\frac{1-2s }{3-2s}}\label{aproxlinea1d1} \\
&\lesssim 
\|u_n-\tu_n\|_{S(\dHs;[0,\tilde T])}
^{\frac{2}{3-2s}}
\|\nabla(u_n-\tu_n)\|_{L^\infty_{[0,\tilde T]}L^{2}_x}^{\frac{1-2s }{3-2s}}\label{aproxlinea2d1} \\
&\lesssim c(\tilde T)^{ \frac{2}{3-2s}}
(M_2B^2+B^2+o(1))^{\frac{1-2s }{3-2s}}
\epsilon^{\frac{2}{3-2s}},
\notag
\end{align}
here, we used H\"older's inequality to get \eqref{aproxlinea1d1} and the Sobolev embedding $\dH^1(\R^1)\hookrightarrow L^\infty(\R^1)\;$ to obtain \eqref{aproxlinea2d1}.

Similar to the argument in the proof of \eqref{aproxL6NL}, we establish that for $0\leq t \leq\tilde T$
\begin{align}
\label{aproxuL6NL}
\| u_n(t) \|^{p+1}_{L^{p+1}} = \sum^M_{j=1} \|v^j(t-t^j_n) \|^{p+1}_{L^{p+1}} + \|W^M_n(t) \|^{p+1}_{L^{p+1}} +o_n(1).
\end{align} 
Energy conservation and \eqref{pythagoreanNLS} give us
\begin{align}
E[u_n(t)]&=\sum_{j=1}^M E[v^j(t-t^j)]+E[W^M_n]+o_n(1)=\sum_{j=1}^M E[\psi^j]+E[W^M_n]+o_n(1).
\label{energyNLF} 
\end{align}
Combining \eqref{aproxuL6NL} and \eqref{energyNLF}, completes the proof of \eqref{PdecompNLS}.
\end{proof}

We now have all the profile decomposition tools to apply to our particular situation in part I (a) of Theorem A*.

\begin{prop}[Existence of a critical solution.]\label{existence of u_c}
There exists a global   $(T=+\infty)$   $H^1$ solution $u_\crit(t)\in H^1(\Rn)$ with initial datum $u_{c,0}\in H^1(\Rn)$ such that 
$$\|u_{\crit,0}\|_{\Lt}=1, \quad \quad
E[u_\crit]^s=(ME)_\crit<M[\uQ]^{1-s} E[\uQ]^s, $$
$$
\g_{u_c}(t)<1\; \;\text{~~~ for all~~~ }\; \;0\leq t<+\infty,$$
\begin{equation}\label{critical-besov}
\|u_\crit\|_{\dB^0_{\SHs}}=+\infty.
\end{equation}
\end{prop} 

\begin{proof}
Consider a sequence of  solutions $u_n(t)$ to $\NLSf$ with corresponding initial data $u_{n,0}$ such that $\g_{u_n}(0)<1$
and $M[u_n]^{1-s} E[u_n]^s \searrow (ME)_\crit$ as $n\to +\infty$, for which $SC(u_{n,0})$ does not hold for any $n$. 

Without lost of generality, rescale the solutions so that $\|u_{n,0}\|_{\Lt}=1$, thus,
$$\|\nabla u_{n,0}\|_{\Lt}^s<\|\uQ\|^{1-s}_{\Lt}\|\nabla \uQ\|^{s}_{\Lt} \quad\mbox{ and }\quad E[u_n]^{s}\searrow (ME)_\crit.$$ 
By construction, $\|u_n\|_{\dB^0_{\SHs}}=+\infty$. Note that  the sequence $\{u_{n,0}\}$ is uniformly bounded on $H^1$. Thus, applying the nonlinear profile decomposition (Proposition \ref{nonradialPDNLS}), we have
\begin{align}\label{profileinitial}
u_{n,0}(x)=\sum_{j=1}^M\NLS(-t^j_n)\tpsi^j(x-x_n^j)+\wt^M_n(x).
\end{align}
Now we will refine the profile decomposition property (b) in Proposition \ref{nonradialPDNLS} by using part II of Proposition \ref{Existence of wave operator} (wave operator), since it is specific to our particular setting here. 

Recall that in nonlinear profile decomposition we consider 2 cases when $|t_n^j|\to \infty$ and $|t_n^j|$ is bounded. In the first case, we can refine it to the following.

First note that we can obtain $\tpsi^j$ (from linear $\psi^j$) such that 
$$\|\NLS(-t^j_n)\tpsi^j-e^{-it^j_n\Delta}\psi^j\|_{H^1}\to 0\quad\text{~~as~~}\quad n\to +\infty$$ 
with properties \eqref{wave meg}, since the linear profiles $\psi^j$'s satisfy 
$$\| \psi \|^{2(1-s)}_{\Lt} \| \nabla \psi \|^{2s}_{\Lt}  < \sigma^2 \left(\frac {d  }{s}\right)^sM[\uQ]^{1-s} E[\uQ]^{s}.$$
We also have,
$$\sum_{j=1}^{M}M[e^{-it^j\Delta}\psi^j]+\lim_{n\to+\infty}M[W^M_n]=\lim_{n\to+\infty}M[u_{n,0}]=1.$$
$$\sum_{j=1}^M\lim_{n\to+\infty} E[e^{-it^j_n\Delta}\psi^j] +\lim_{n\to+\infty}E[ W^M_n]=\lim_{n\to+\infty}E[u_{n,0}]=(ME)_\crit,$$
thus,
$\frac12\|\psi^j\|_{\Lt}^{1-s}\|\nabla\psi^j\|^{s}_{\Lt}\leq(ME)_\crit.$

The properties \eqref{wave meg} for $\tpsi^j$ imply that  $\ME[\tpsi^j]<(ME)_\crit,$ and thus, we get that 
\begin{equation}\label{nonlinear Besov}
\|\NLS(t)\tpsi^j(\cdot-x^j_n)\|_{\dB^0_{\SHs}}<+\infty.
\end{equation}

This fact will be essential for case 1 below. Otherwise, in nonlinear decomposition \eqref{profileinitial} we also have the Pythagorean decomposition for mass and energy:
$$
\sum_{j=1}^M\lim_{n\to+\infty} E[\tpsi^j] +\lim_{n\to+\infty}E[ \wt^M_n]=\lim_{n\to+\infty}E[u_{n,0}]=(ME)_\crit^{\frac1s}.
$$
Since each energy is greater than 0 (Lemma \ref{equivalence grad and energy}), for all $j$ we obtain
\begin{align}
\label{energy initial}
E[\tpsi^j]^{s}\leq(ME)_\crit.
\end{align}

Furthermore, 
 $s=0$ in (\ref{HPdecompNL}) imply
\begin{align}\label{masainitial}
\sum_{j=1}^{M}M[\tpsi^j]+\lim_{n\to+\infty}M[\wt^M_n]=\lim_{n\to+\infty}M[u_{n,0}]=1.
\end{align}

We show that  in the profile decomposition \eqref{profileinitial}  either 
more than one profiles $\tpsi^j$ are non-zero, or 
only one profile $\tpsi^j$ is non-zero  and the rest ($M-1$) profiles are zero. 
The first case will give a contradiction to the fact that each $u_n(t)$ does not scatter, consequently, only  the second possibility holds. That non-zero profile $\tpsi^j$ will be the initial data $u_{c,0}$
 and will produce the critical solutiton $u_\crit(t)=\NLS(t)u_{c,0},$ such that $\|u_\crit\|_{\dB^0_{\SHs}}=+\infty.$
\bigskip

\noindent\emph{ Case 1:} More than one $\tpsi^j\neq0.$  For each $j,$ \eqref{masainitial}  gives $M[\tpsi^j]<1$   and for a large enough $n$,  \eqref{energy initial} and \eqref{masainitial} yield
\begin{align*}
M[\NLS(t)\tpsi^j] ^{1-s} E[\NLS(t)\tpsi^j]^s=M[\tpsi^j] ^{1-s}E[\tpsi^j]^s<(ME)_\crit.
 \end{align*}
 Recall \eqref{nonlinear Besov}, we have
 $$
\|\NLS(t-t^j)\tpsi^j(\cdot-x^j_n)\|_{\dB^0_{\SHs}}<+\infty, \quad\quad \text{~ for large enough~} n,
$$
and thus, the right hand side in \eqref{profileinitial} is finite in $S(\dHs),\;$ since \eqref{smallnessPropNL} holds for the remainder $\wt_n^M(x).\;$ This contradicts the fact that $\|\NLS(t)u_{n,0}\|_{\dB^0_{\SHs}}=+\infty$.

\noindent\emph{ Case 2:} Thus, we have that only one profile $\tpsi^j$ is non-zero, renamed to be $\tpsi^1$, 
\begin{equation}\label{initial approx}
u_{n,0}=\NLS(-t^1_n)\tpsi^1(\cdot-x^1_n)+\wt^1_n,
\end{equation} 
with 
$$M[\tpsi^1]\leq 1,\quad\;E[\tpsi^1]^s\leq(ME)_\crit\quad\; \text{~~and~~} \quad\;
\lim_{n \to+\infty} \|\NLS(t)\wt^1_n\|_{\dB^0_{\SHs}}=0.
$$

Let $u_\crit$ be the solution to $\NLSf$ with the initial condition $u_{c,0}=\tpsi^1$. Applying $\NLS(t)$ to both sides of \eqref{initial approx} and estimating it in ${\dB^0_{\SHs}}$, we obtain (by the nonlinear profile decomposition Proposition \ref{nonradialPDNLS}) that
\begin{align*}
 \|u_\crit\|_{\dB^0_{\SHs}} 
&
=\|\NLS(t-t_n^1)\tpsi^1\|_{\dB^0_{\SHs}}
=\lim_{n\to\infty}\|\NLS(t)u_{n,0}\|_{\dB^0_{\SHs}}
\\&=\lim_{n\to\infty}\|u_{n}(t)\|_{\dB^0_{\SHs}}=+\infty,
\end{align*}
since by construction $\|u_n\|_{\dB^0_{\SHs}}=+\infty,$ completing the proof.
\end{proof}

\begin{lemma}[ Precompactness of the flow of the critical solution] \label{precompact} Assume $u_\crit$ as in Proposition \ref{existence of u_c}. Then there exists a continuous path $x(t)$ in $\Rn$ such that  
$$K=\{u_\crit(\cdot-x(t),t)|t\in[0,+\infty)\} $$
is precompact in $H^1(\Rn)$.
\end{lemma}
\begin{proof}
Let  a sequence $\tau_n \to
+\infty$ and $\phi_n = u_\crit(\tau_n)$ be a uniformly bounded sequence in $H^1$; we want to show that $u_\crit(\tau_n)$ has a convergent subsequence in $H^1$.

The nonlinear profile decomposition (Proposition \ref{nonradialPDNLS}) implies the existence of  profiles
$\tpsi^j,$ the time and space sequences $\{t^j_n\},\; \{x^j_n\}\;$ and an error $\wt_n^M$ such that
\begin{equation}\label{critical profile}
u_\crit(\tau_n) = \sum^{M}_{j=1} \NLS(-t^j_n)\tpsi^j(x-x^j_n) + {\wt}^M_n(x),
\end{equation}
with $|t_n^j-t_n^k|+ |x_n^j-x_n^k| \to +\infty$ as $n\to +\infty$ for fixed $j\neq
k$. In addition,
\begin{align*}
\sum_{j=1}^M E[\tpsi^j]+E[\wt^M_n]=E[u_\crit]=(ME)_\crit,
\end{align*}
since each energy is nonnegative, and we have
$$
\lim_{n\to\infty}E[\NLS(
-t^j_n)\tpsi^j(x-x^j_n) ]\leq(ME)_\crit.
$$
Taking $s=0$ in \eqref{HPdecompNL}
\begin{align*}
 \sum^M_{j=1}M[\tpsi^j(x-x^j_n)] +\lim_{n\to\infty} \|\wt^M_n \|^2_{\Lt} =M[u_{c}]=1.
\end{align*}
Note that, in the decomposition \eqref{critical profile} either we have more than one $\tpsi^j\neq0$ or only one $\tpsi^1\neq0$ and $\tpsi^j=0$ for all $2\leq j<M$. Following the argument of Proposition
\ref{existence of u_c}, we show that only the second case occurs:
\begin{equation}
\label{E:case2expand}
u_\crit(\tau_n) = \NLS(-t^1_n)\tpsi^1(x-x^1_n)+\wt^1_n(x)
\end{equation}
such that
$$M[\tpsi^1]=1,\quad\quad
\lim_{n\to\infty}E[\NLS(-t^1_n)\tpsi^j(x-x^1_n) ]=(ME)_\crit, $$
$$
\lim_{n\to\infty} M[\wt^M_n] =0\quad\text{~~and~~}\quad\lim_{n\to\infty} E[\wt^M_n]=0. 
$$
Lemma \ref{equivalence grad and energy} implies
\begin{equation}\label{E:H^1tozero}
\lim_{n\to\infty} \|\wt^M_n\|_{H^1}=0.
\end{equation}
The sequence $x^1_n$ will create a path $x(t)$ by continuity.
We now show that $t^1_n$ has a convergence
subsequence $\widetilde{t^1_n}.$

Assume that $\widetilde{t_n^1} \to
-\infty$, apply $\NLS(t)$ to \eqref{E:case2expand}  implies then  triangle inequality yields
\begin{align*}
\|\NLS(t)u_\crit(\tau_n)\|_{\dB^0_{S(\dHs; [0,+\infty))}} \leq &\|
\NLS(t-\widetilde {t^1_n})\tpsi^1(x-x^1_n)\|_{\dB^0_{S(\dHs; [0,+\infty))}}+ \|\NLS(t) \wt_n^M(x)\|_{\dB^0_{S(\dHs; [0,+\infty))}}.
\end{align*}
Note
\begin{align*}
\lim_{n\to +\infty} \| \NLS(t-\widetilde{t_n^1})&\tpsi^1(x-x^1_n)\|_{\dB^0_{S(\dHs; [0,+\infty))}} = \lim_{n\to +\infty} \|\NLS(t)\tpsi^1(x-x^1_n)\|_{\dB^0_{S(\dHs; [t^1_n,+\infty))}} = 0,
\end{align*}
and 
$$\|\NLS(t)\wt_n^M\|_{\dB^0_{S(\dHs)}} \leq \frac12 \,\delta_{\text{sd}},$$ 
thus, taking $n$ sufficiently large, the small data
scattering theory (Proposition \ref{small data}) implies
$\|u_\crit\|_{\dB^0_{S(\dHs;(-\infty,\tau_n))}} \leq \delta_{\text{sd}}$ a contradiction.  

In a similar fashion, assuming that 
$\widetilde{t_n^1} \to
+\infty$,  we obtain that for $n$ large,

$
\|\NLS(t)u_\crit(\tau_n)\|_{\dB^0_{S(\dHs; (-\infty,0])}} \leq
\frac12\delta_{\text{sd}},
$
and thus, the small data scattering theory (Proposition \ref{small data})
shows that
\begin{equation}\label{scatcrit}
\|u_\crit\|_{\dB^0_{S(\dHs;(-\infty,\tau_n])}} \leq
\delta_{\text{sd}}.
\end{equation}
Taking $n\to +\infty$ implies $\tau_n\to +\infty$, thus \eqref{scatcrit} becomes
 $\|u_\crit\|_{\dB^0_{S(\dHs;(-\infty,+\infty))}} \leq
\delta_{\text{sd}}$, a contradiction. 
Thus,
$\widetilde{t_n^1}$ must converge to some finite $t^1$.

 Since  \eqref{E:H^1tozero} holds and
$\NLS(\tilde{t_n^1})\tpsi^1 \to
{\NLS(t^1)}\tpsi^1$
  in  $H^1,$  \eqref{E:case2expand} implies  $u_\crit(\tau_n)$ converges in $H^1$.
\end{proof}

\begin{cor}\label{precompact-localization}\textrm{( Precompactness of the flow implies uniform localization.)} Assume $u$ is a solution to \eqref{eq:NLS}  such that
$$K=\{u
(\cdot-x(t),t)|t\in[0,+\infty)\} $$
is precompact in $H^1(\Rn)$. Then for each $\epsilon>0$, there exists $R>0$, so that for all $0\leq t<\infty$
\begin{align}\label{norm-H1-energy}
\int_{|x+x(t)|>R}|\nabla u(x,t)|^2+|u(x,t)|^2+|u(x,t)|^{p+1}dx<\epsilon.
\end{align}
Furthermore,
$\|u(t,\cdot-x(t))\|_{H^1(|x|> R)}< \epsilon.$
\end{cor}

\begin{proof}
Assume \eqref{norm-H1-energy} does not hold, i.e., there exists $\epsilon>0$ and a sequence of times $t_n$
such that for any $R>0, $ we have
$$
\int_{|x+x(t_n)|>R} |\nabla u(x,t_n)|^2 + |u(x,t_n)|^2 +|u(x,t_n)|^{p+1}\, dx \geq
\epsilon.
$$
Changing variables, we get
\begin{equation}
\label{int>epsilon}
\int_{|x|>R} |\nabla u(x - x(t_n),t_n)|^2 + |u(x - x(t_n),t_n)|^2+ |u(x - x(t_n),t_n)|^{p+1}\,
dx \geq \epsilon.
\end{equation}
Note that since $K$ is precompact, there exists $\phi \in H^1$ such that,
passing to a subsequence of $t_n$, we have  $u(\cdot - x(t_n),t_n)
\to \phi$ in $H^1$. For all $R>0$,   \eqref{int>epsilon} implies
$$
\forall R>0,\quad \int_{|x|>R} |\nabla \phi(x)|^2 + |\phi(x)|^2 +|\phi(x)|^4\geq \epsilon,
$$
which is a contradiction with the fact that $\phi\in H^1$. Thus,  \eqref{norm-H1-energy} and $\|u(t,\cdot-x(t))\|_{H^1(|x|> R)}< \epsilon$ hold.
\end{proof}

\begin{lemma}
\label{spacial translation}
Let $u(t)$ be a solution of $\NLSf$ defined on $[0,+\infty)$ such that $P[u]=0$ and either
\begin{enumerate}[(a)]
\item $K=\{u(\cdot-x(t),t) | t \in[0,+\infty)\}$ is precompact in $H^1(\Rn)$, 
or
\item for all $0<t$,
\begin{align} \label{E:ground-close}
\|u(t)-e^{i\theta(t)}\uQ(\cdot-x(t))\|_{H^1}\leq \epsilon_1
\end{align}
\end{enumerate}
for some continuous function $\theta(t)$ and $x(t).$ Then 
\begin{align}\label{E:x-control}
\lim_{t\to+\infty}\frac{x(t)} {t}  = 0.
\end{align} 
\end{lemma}

Proof of this Lemma can be found in \cite{guevara} or adjusted from \cite{DuHoRo08}.

\begin{thm}\textrm{ (Rigidity Theorem.)}\label{rigidity}
Let $u_0 \in H^1$   satisfy $P[u_0]=0$,  $\ME[u_0] <1$ and $\g_u(0)<1$. 
Let $u$ be the global $H^1(\Rn)$ solution of $\NLS_{p}(\Rn)$ with initial data $u_0$ and 
suppose that $K=\{u_\crit(\cdot-x(t),t)|t\in[0,+\infty)\} $ is precompact in $H^1,$ then $u_0\equiv0$.
\end{thm}

\begin{proof}
Let $\phi \in C^{\infty}_0$ be radial, with
$$
\phi(x) =
\begin{cases}
|x|^2 & \mbox{ for } |x| \leq 1\\
0 & \mbox{ for } |x| \geq 2.
\end{cases}
$$
For $R>0$ define
\begin{align}
z_R(t) = \int R^2 \phi(\frac{x}{R}) |u(x,t)|^2 dx.
\label{localization_z}
\end{align} 
Then 
\begin{align}
z'_R(t) =2\im \int R \nabla \phi(\frac{x}{R})\cdot \nabla u(t)\bar u(t) dx,
\label{derivative_z}
\end{align} 
 and H\"older's inequality yields
\begin{align}
|z_R'(t)| \leq c R \int_{ \{ |x| \leq 2R \} } | \nabla u(t) | |u(t)| dx \leq c R \| u(t)\|_{\Lt}^{2(1-s)} \| \nabla u(t) \|_{\Lt}^{2s}
\label{localization_zder}.
\end{align}
Note that, 
\begin{align} \label{virialz}
\notag z''_R(t)  = 4 \sum_{j,k}\int \frac{\partial^2\phi}{\partial x_j\partial x_k}& \bigg( \frac{|x|}{R} \bigg)\frac{\partial u}{\partial x_j}\frac{\partial \bar u}{\partial x_k} - \frac{1}{R^2} \int \Delta^2 \phi \bigg ( \frac{|x|}{R} \bigg) |u|^2\\& -4\left( \frac{1}{2}-\frac{1}{p+1} \right) \int \Delta \phi  \bigg( \frac{|x|}{R} \bigg) |u|^{p+1}.
\end{align}
Since $\phi$ is radial, we have
\begin{align}\label{ridigity-eq1}
z''_R(t)  = 8 \int | \nabla u |^2 - \frac{4d({p-1})}{p+1}  \int |u|^{p+1}+A_R(u(t)),
\end{align}
where 
\begin{align}\label{ARvar}
A_R&(u(t)) = 4 \sum_{j}\int \Bigg(\partial^2_{x_j}\phi \bigg( \frac{|x|}{R} \bigg)-2\Bigg) \Bigg |  \frac{\partial u}{\partial x_j} \Bigg |^2 
+4 \sum_{j\neq k}\int_{R\leq|x|\leq 2R} \frac{\partial^2\phi}{\partial x_j\partial x_k} \bigg( \frac{|x|}{R} \bigg)\notag\\
& - \frac{1}{R^2} \int \Delta^2 \phi \bigg ( \frac{|x|}{R} \bigg) |u|^2- 4\left( \frac{1}{2}-\frac{1}{p+1} \right)\int \Bigg(\Delta \phi  \bigg( \frac{|x|}{R} \bigg)-2d\Bigg) |u|^{p+1}.
\end{align}
Thus, 
\begin{align}
\label{no radial remainder}
\big|A_R(u(t))\big|  = c \int_{|x|\geq R} \bigg(| \nabla u(t) |^2 + \frac{1}{R^2}  |u(t)|^2 + |u(t)|^{p+1}\bigg)dx.
\end{align}
Choosing $R$ large enough, over a suitably chosen time interval $[t_0,t_1]$, with $0\ll t_0\ll t_1<\infty$, combining \eqref{ridigity-eq1} and \eqref{lower}, we obtain
\begin{align}
\label{second derivative} 
|z''_R(t)|\geq 16(1-\omega^{p-1})E[u]-|A_R(u(t))|.
\end{align}
From Corollary \ref{precompact-localization} , letting $\epsilon=\frac{1-\omega^{p-1}}{c}$, with $c$ as in \eqref{no radial remainder}, we can obtain $R_0\geq0$ such that for all $t$,
\begin{align}
\label{integral}
\int_{|x+x(t)|>R_0}\big(|\nabla u(t)|^2+|u(t)|^2+|u(t)|^{p+1}\big)\leq\frac{1-\omega^{p-1}}{c}E[u].
\end{align}
Thus combining (\ref{no radial remainder}), (\ref{second derivative}) and (\ref{integral}), and  taking $R\geq R_0+\sup_{t_0\leq t\leq t_1}|x(t)|,$ gives that for all $t_0\leq t\leq t_1$, 
 \begin{align}
\label{derivada2}
|z''(t)|\geq 8(1-\omega^{p-1})E[u].
\end{align}
By Lemma \ref{spacial translation}, there exists $t_0\geq0$ such that for all $t\geq t_0$, we have $|x(t)|\leq \gamma t.$ Taking $R=R_0+\gamma t_1,$ we have that (\ref{derivada2}) holds for all $t\in[t_0,t_1]$. Thus, integrating it over this interval, we obtain
\begin{align}\label{energy1}
|z'_R(t_1)-z'_R(t_0)| \geq 8(1-\omega^{p-1}) E[u](t_1-t_0).
\end{align}
In addition, for all $t\in[t_0,t_1]$,  combining \eqref{localization_zder}, $\g_u(0)<1$, and Lemma \ref{Lower-bound-convexity} we have
\begin{align}
|z_R'(t)|&\leq cR\|u(t)\|^{2(1-s)}_{\Lt}\|\nabla u(t)\|^{2s}_{\Lt}\leq2 cR\|\uQ\|^{2(1-s)}_{\Lt}\|\nabla \uQ\|^{2s}_{\Lt}\notag\\
&\leq c\|\uQ\|^{2(1-s)}_{\Lt}\|\nabla \uQ\|^{2s} _{\Lt}(R_0+\gamma t_1). \label{energy2}
\end{align}
Combining \eqref{energy1} and \eqref{energy2} yields 
\begin{equation}\label{rigidez}
8(1-\omega^{p-1}) E[u](t_1-t_0)\leq 2 c\|\uQ\|^{2(1-s)}_{\Lt}\|\nabla \uQ\|^{2s} _{\Lt}(R_0+\gamma t_1).
\end{equation}
Observe that, $\omega$,  and $R_0$ are constants depending on $\ME[u]$,  and $t_0=t(\gamma)$. Let  $\gamma=\frac{(1-\omega^{p-1})E[u]}{c\|\uQ\|^{2(1-s)}_{\Lt}\|\nabla \uQ\|^{2s}_{\Lt}}>0$.Then \eqref{rigidez} yields
\begin{equation}\label{rigidez2}
6(1-\omega^{p-1}) E[u]t_1\leq 2 c\|\uQ\|^{2(1-s)}_{\Lt}\|\nabla \uQ\|^{2s} _{\Lt}R_0+8(1-\omega^{p-1}) E[u]t_0.
\end{equation}
Now sending $t_1\to +\infty,$ implies that the left hand side of \eqref{rigidez2} goes to $\infty$ and  the right hand side is bounded, which is a contradiction, unless $E[u]=0$ which implies $ u\equiv 0$.
\end{proof}

\section{Weak blowup via Concentration Compactness}\label {wbup}

In this section, we complete the proof of  Theorem A* part II  (b), i.e.,
if under the mass-energy threshold  $\ME[u]<1$, a solution $u(t)$ to $\NLSf$ with the initial condition $u_0\in H^1$ such that $\g_{u}(0)>1$ exists globally for all positive time, then  there exists a sequence of times $t_n\to +\infty$ such that $\g_u(t_n)\to +\infty$. We call this solution a ``weak blowup" solution.

\begin{definition}\label{me_line}  Let $\lambda>0$.  The horizontal line for which 
$$M[u]=M[u_Q]\quad \text{~~and~~}\quad \frac{E[u]}{E[u_Q]}=\frac{d}{2s}\lambda^{\frac{2}{s}}\bigg(1-\frac{\lambda^{p-1}}{\alpha^2}\bigg)$$ 
is called the ``\emph{ mass-energy}" line for $\lambda.$ 
\end{definition}
Notice that in Definition \ref{me_line}, the renormalized energy definition comes naturally by expressing the energy in terms of the gradient which is assumed to be $\lambda$.   We illustrate the mass-energy line notion 
in Figure \ref{fig2}.
\begin{figure}[htbp]
\includegraphics[width=160mm, height=120mm]{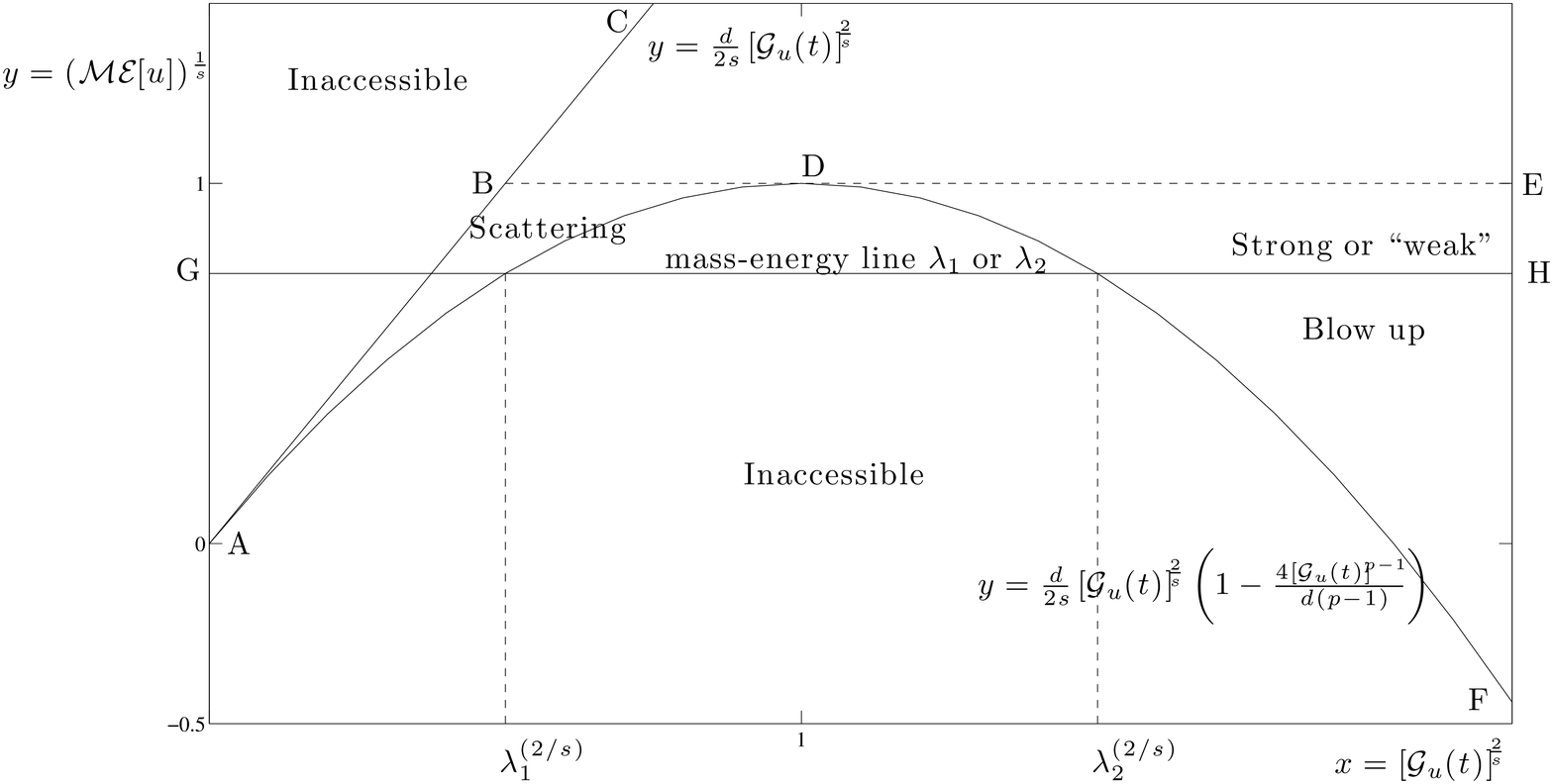}
\caption[Mass energy line for $\lambda>0$]{This is a graphical representation of restrictions on energy and gradient.  
 For a given $\lambda>0,$ the horizontal line GH is referred to as the ``mass-energy" line for this $\lambda$. Observe that this horizontal line can intersect the parabola
  $y=\frac{d}{2s} \left[\g_u(t )\right]^{\frac2s}\left(1 -\frac{4\left[\g_u(t )\right]^{p-1}}{d(p-1)} \right)$
twice, i.e., it can be a ``mass-energy" line for $0<\lambda_1<1$ and $1<\lambda_2<\infty$, the first case produces solutions which are global and are scattering (by Theorem A* part I) and the second case produces solutions which either blow up in finite time or diverge in infinite time (``weak blowup") as shown in Section \ref{wbup}. }
 \label{fig2}
\end{figure}

\subsection{ Outline for Weak blowup via Concentration Compactness}\label{outline weak}

Suppose that there is no finite time blowup for a nonradial and infinite variance solution (from Theorem A* part II),  thus, the existence on time (say, in forward direction) is infinite ($T^*=+\infty$). Now, under the assumption of global existence,  we study the behavior of $\g_u(t)$ as $t\to+\infty$, and use a concentration compactness type argument for establishing the divergence of $\g_u(t)$ in $H^1-$norm  as it was developed in \cite{HoRo09}, note that the concentration compactness  and the rigidity argument is not used here to prove scattering but to prove for a blowup property. The description of this argument is in steps 1, 2 and 3.

\noindent\emph{ Step 1:} \emph{ Near boundary behavior.}

\begin{figure}[htbp]
\includegraphics[width=150mm]{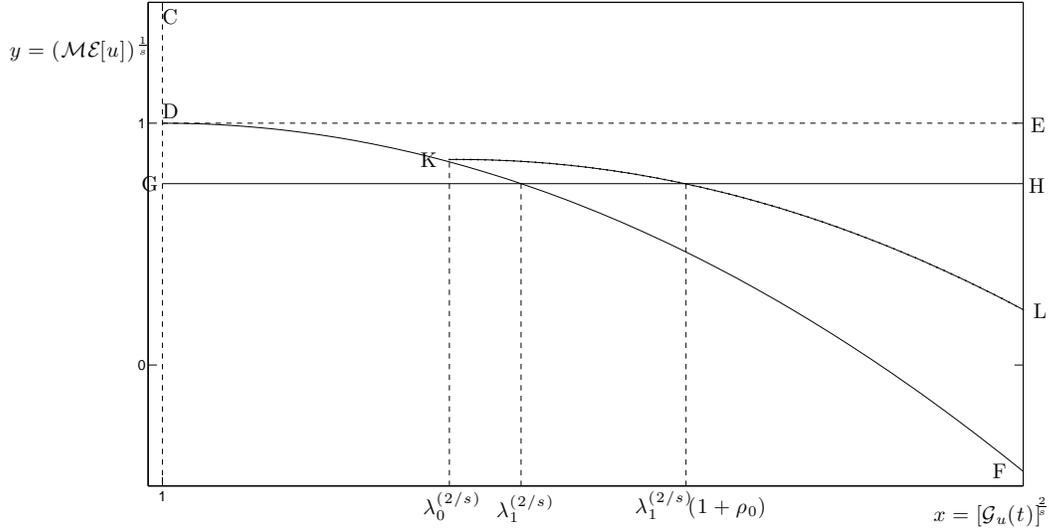}
\caption[Near  boundary behavior of  $\g(t)$.]{Near  boundary behavior of  $\g(t)$. We investigate whether the solution can
remain close to the boundary (see the dash dot line KL) for all time. }
 \label{fig3}
\end{figure}

\noindent Theorem A* II part (a) yields $\g_u(t)>1$ for all $t\in (T_*,T^*)$ whenever $\g_u(0)>1$  on the ``mass-energy" line for some $\lambda>1$. We illustrate this in Figure \ref{fig2}: given $u_0\in H^1$, we first determine $M[u_0]$ and $E[u_0]$ which specifies the ``mass-energy" line  GH. Then the gradient $\g_u(t)$ of a solution $u(t)$ lives on the line GH.  Note that $\g_u(t)>\lambda_2>1$ if   $\g_u(0)>1$. A natural question is whether $\g_u(t)$  can be, with time, much larger than 1 or $\lambda_2$. Proposition \ref{near boundary case} shows that it can not. Thus, we prove that the renormalized gradient $\g_u(t)$ can not forever remain near the boundary if originally $\g_u(0)$ is very close to it, that is,  
 if $\lambda_0>1$, there exists  
$ \rho_0(\lambda_0)>0$ such that for all $ \lambda>\lambda_0$ 
there is NO solution at the ``mass-energy" line for $\lambda$ satisfying 
\begin{equation*}
\lambda\leq\g_u(t)\leq\lambda(1+\rho_0).\label{eq:gradbound}
\end{equation*}
\noindent Using the Figure \ref{fig3}, this means that the solution $u(t)$ would have a gradient $\g_u(t)$ very close to the boundary DF (for all times), i.e., between the boundary DF and the dashed line KL.
\noindent  We will show that $\g_u(t)$ on any ``mass-energy" line with $\ME[u]<1$ and $\g_u(0)>1$ will escape to infinity (along this line). By contradiction,  assume that  all solutions  (starting from some mass-energy line corresponding to the initial renormalized gradient $\g_u(0)=\lambda_0>1$) are bounded in renormalized gradient for all $t>0.$

Step 1 gives the basis for induction,  giving that when $\lambda>1$, any  solution $u(t)$ of $\NLSf$ at the ``mass-energy" line for this $\lambda$ can not have a renormalized gradient $\g_u(t)$ bounded near the boundary $DF$ for all time (see Figure \ref{fig3}). We will show that  $\g_u(t)$, in fact, will tend to $+\infty$ (at least along an infinite time sequence). 

\begin{definition}\label{GBG}
Let $\lambda>1.$ We say  \emph{the property}  $\GB(\lambda,\sigma)$ \emph{holds}\footnote[13]{\normalsize $\GB$ stands for \emph{globally bounded gradient.}\\} if there exists a solution $u(t)$ of $\NLSf$ at the mass-energy line $\lambda$ (i.e., $M[u]=M[u_Q]$ and $\frac{E[u]}{E[u_Q]}=\frac{d}{2s}\lambda^{\frac{2}{s}}\left(1-\frac{\lambda^{p-1}}{\alpha^2}\right)$) such that 
$\lambda\leq\g_u(t)\leq\sigma$  for all $ t\geq 0$.
Figure \ref{fig4} illustrates this definition.
\end{definition}
\begin{figure}[htbp]
\includegraphics[width=150mm]{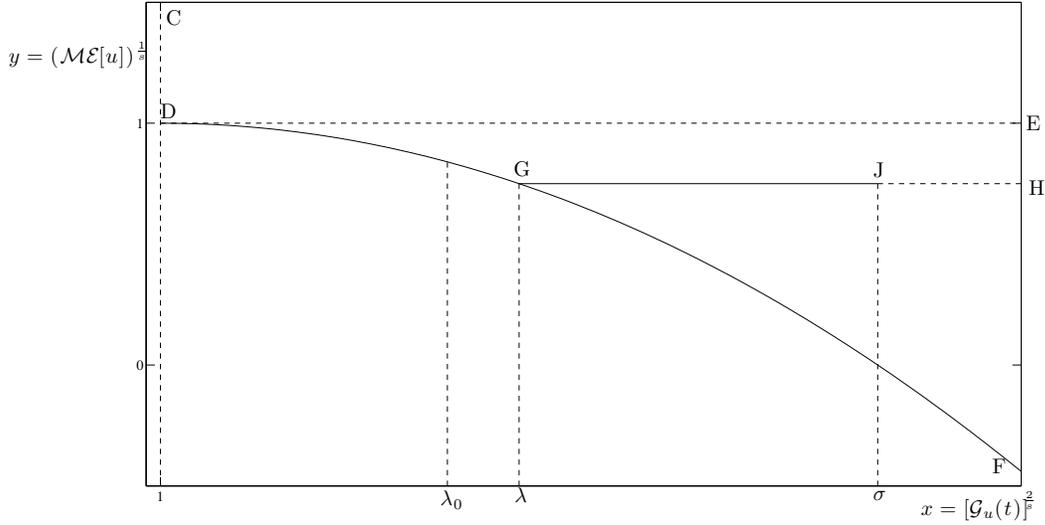}
\caption[Globally Bounded Gradient]{On this graph the statement ``$\GB(\lambda,\sigma)$ \emph{holds}" shoes that  $\g(t)$ is only on the segment GJ.}
 \label{fig4}
\end{figure}
  In other words, $\GB(\lambda, \sigma)$ is not true if for  every solution  $u(t)$ of $\NLSf$ at the ``mass-energy" line for $\lambda$, such that $\lambda\leq\g_u(t)$ for all $t>0,$   there exists $t^*$ such that $\sigma<\g_u(t^*)$. Iterating,
we conclude that, there exists a sequence $\{t_n\}\to \infty$ with $\sigma<\g_u(t_n)$ for all $n$.
  
Note that, if $\GB(\lambda,\sigma)$ does not hold, then for any $\sigma'<\sigma,\;$ $\GB(\lambda,\sigma')$ does not hold either. This will allow us induct on the $\GB$ notion.

\begin{definition}\label{thresholdGB}
Let $\lambda_0>1$.   We define \textrm{ the critical threshold} $\sigma_c$ by
$$
\sigma_c=  \sup\big\{\sigma | \sigma>\lambda_0 \mbox{ and } \GB(\lambda,\sigma) \mbox{ does NOT hold for all } \lambda \mbox{ with } \lambda_0\leq\lambda\leq\sigma\big\}.
$$  
Note that  $\sigma_c=\sigma_c(\lambda_0)$  stands for ``$\sigma$-critical".
\end{definition}

From the step 1 (Proposition \ref{near boundary case}) 
we have that $\GB(\lambda,\lambda(1+\rho_0(\lambda_0))$ does not hold for all $\lambda\geq\lambda_0$.

\noindent\emph{ Step 2: Induction argument.} 

\noindent Let $\lambda_0>1$ . We would like to show that $\sigma_c(\lambda_0)=+\infty$. Arguing by contradiction, we assume $\sigma_c(\lambda_0)$ is finite. 

\noindent Let $u(t)$ be a solution to $\NLSf$ with initial data $u_{n,0}$ at the ``mass-energy" line for $\lambda>\lambda_0,$ i.e.,$\frac{E[u]}{E[\uQ]}=\frac{d}{2s}\lambda^{\frac{2}{s}}\bigg(1-\frac{\lambda^{p-1}}{\alpha^2}\bigg)$,  $\;M[u]=M[\uQ]$
 and $\g_{u}(0)>1.$  We want to show that there exists a sequence of times $\{t_n\}\to+\infty$ such that $\g_u (t_n)\to\infty$. Suppose the opposite, that is,  such sequence of times does not exist.
 
\noindent Then  there exists $\sigma<\infty$ satisfying  $\lambda\leq \g_{u}(t)
\leq \sigma$ for all $t\geq0$, i.e., $\GB(\lambda,\sigma)$ holds with  $\sigma_c(\lambda_0)\leq\sigma<\infty.$ At this point we can apply Proposition \ref{nonradialPDNLS} (the nonlinear profile decomposition).

\noindent The nonlinear profile decomposition of the sequence $\{u_{n,0}\}$ and profile reordering will allow us to construct  a ``\emph{critical threshold solution}"  $u(t)=u_c(t) $ to $\NLSf$ at the ``mass-energy" line $\lambda_c,$ where $\lambda_0<\lambda_c<\sigma_c(\lambda_0)$ and $\lambda_c<\g_{u_c}(t)<\sigma_c(\lambda_0)$ for all $t>0$ (see Existence of threshold solution  Lemma \ref{existence threshold}).
 
 \noindent\emph{ Step 3:}  \emph{ Localization properties of critical threshold solution.} 

\noindent By construction, the critical threshold solution $u_c(t)$ will have the property  that the set 
$K=\{u(\cdot-x(t),t)|t\in[0,+\infty)\}$ has a compact closure in $H^1$ (Lemma \ref{compactH1}). Thus, we will have uniform concentration of $u_c(t)$ in time, which together with the localization property  (Corollary \ref{precompact-localization}) implies that for a given $\epsilon >0$, there exists an $R>0$ such that
$
\|\nabla u(x,t)\|_{\Lt(|x+x(t)|>R)}^2\leq\epsilon
$   
uniformly in $t\;$; as a consequence, $u_c(t)$ blows up in finite time (Lemma \eqref{blowup localization}), that is, $\sigma_c=+\infty,$ which contradicts the fact that $u_c(t)$ is bounded  in $H^1.$ Thus, $u_c(t)$ can not exist since  our assumption that $\sigma_c(\lambda_0)<\infty$ is false,  and this ends the proof of the ``weak blowup".

 In the rest of this chapter we proceed with the proof of claims described  in Step 1, 2 and 3.


First, recall variational characterization of the ground state.

\subsection{ Variational Characterization of the Ground State}

Propositon \ref{lion}  is a restatement of  Proposition 4.4 \cite{HoRo09} adjusted for our  general case, and shows that if a solution $u(t,x)$ is close to $u_Q(t,x)$ in mass and energy, then it is close to $u_Q$ in $H^1(\Rn),$ up to a phase and shift in space. The proof is identical so we omit it.
\begin{prop} \label{lion}
There exists a function $\epsilon(\rho)$ defined for small $\rho>0$ with $\lim_{\rho\to 0}\epsilon(\rho)=0$, such that for all $u\in H^1(\Rn)$ with
\begin{align*}
\big| \|u\|_{L^{p+1}}-\|u_Q\|_{L^{p+1}}\big|+\big| \|u\|_{L^2}-\|u_Q\|_{L^2}\big|+\big| \|\nabla u\|_{L^2}-\|\nabla u_Q\|_{L^2}\big|\leq\rho,
\end{align*}
there is $\theta_0\in \R$ and $x_0\in \Rn$ such that 
\begin{align}\label{sca_lion}
\|u-e^{i\theta_0}u_Q(\cdot-x_0)\|_{H^1}\leq\epsilon(\rho).
\end{align}
\end{prop}

The Proposition \ref{newchar} is a variant of Proposition 4.1 \cite{HoRo09}, rephrased for our case.
\begin{prop} \label{newchar} 
There exists a function $\epsilon(\rho)$ such that $\epsilon(\rho)\to 0$ as $\rho\to 0$  satisfying the following:  Suppose there exists $\lambda >0$ such that 
\begin{align}\label{bound1la}
\bigg|\left(\ME[u]\right)^{\frac{1}{s}}-\frac{d}{2s}\lambda^{\frac{2}{s}}\bigg(1-\frac{\lambda^{p-1}}{\alpha^2}\bigg)\bigg|\leq \rho  \lambda^{\frac{2(p-1)}{s}}
\end{align}
and
\begin{align}\label{bound2la}
\big|[\g_u(t)]^{\frac{1}{s}}
-\lambda\big|\leq \rho\left\{\begin{array}{c}\lambda^{\frac2s} \quad \text{~~ if}\quad \lambda\leq1 \\\lambda\quad \text{~~ if}\quad \lambda\geq1\end{array}\right. .
\end{align}
Then there exist $\theta_0 \in \R$ and $x_0\in \Rn$ with $\kappa=\Big(\frac{M[u]}{M[u_Q]}\Big)^{\frac{1-s}{s}}$ such that 
\begin{align*}
\Big\|u(x)-e^{i\theta_0}\lambda \kappa^{-\frac{s}{1-s}} u_Q\big(\lambda( \kappa^{-\frac{3s}{d(1-s)}}x-x_0)\big)\Big\|_{L^2}\leq \kappa^{\frac{s}{2(1-s)}}\epsilon(\rho),
\end{align*}
and 
\begin{align*}
\Big\|\nabla\Big[u(x)-e^{i\theta_0}\lambda \kappa^{-\frac{s}{1-s}} u_Q\big(\lambda( \kappa^{-\frac{3s}{d(1-s)}}x-x_0) \Big]\Big\|_{L^2}\leq\lambda \kappa^{-\frac{s}{2(1-s)}}\epsilon(\rho).
\end{align*}
\end{prop}
\begin{proof}
Set $v(x)=\kappa^{\frac{s}{1-s}} u(\kappa^{\frac{3s}{(1-s)d}}x)$, hence $M[v]=\kappa^{-\frac{s}{1-s}}M[u]$. Assume $M[v]=M[u_Q].$ Then there exists $\lambda>0$ such that (\ref{bound1la}) and (\ref{bound2la}) become
\begin{equation}
\label{bound1sc}
\bigg|\dfrac{E[v]}{E[u_Q]}-\frac{d}{2s}\lambda^{\frac{2}{s}}\bigg(1-\frac{\lambda^{p-1}}{\alpha^2}\bigg)\bigg|\leq \rho_0 \lambda^{\frac{2(p-1)}{s}},
\end{equation}
and
\begin{equation}\label{bound2sc}
\bigg|\dfrac{\|\nabla v\|_{L^2}}{\|\nabla u_Q\|_{L^2}}-\lambda\bigg|\leq 
\rho_0 \left\{\begin{array}{c}\lambda^{\frac{2}{s}} \mbox{  if  } \lambda\leq1 \\\lambda \mbox{   if  } \lambda\geq1\end{array}\right. .
\end{equation}
Letting $\tilde u(x)=\lambda^{-\frac{2((p-1)(d-2)-ds)}{((p-1)(d-2)-4)s}}v(\lambda^{-\frac{2((p-1)(2-s)+2s)}{((p-1)(d-2)-4)s}}x)$,  we have
\begin{equation}\label{bound2sc1}
\bigg|\dfrac{\|\nabla \tilde u\|_{L^2}}{\|\nabla u_Q\|_{L^2}}-1\bigg|\leq \rho_0 \left\{\begin{array}{c}\lambda^{\frac{2}{s}-1} \mbox{  if  } \lambda\leq1 \\1\quad \mbox{   if  } \lambda\geq1\end{array}\right.
\leq \rho_0.
\end{equation}
Combining Pohozhaev identities, (\ref{bound1sc}) and (\ref{bound2sc}), gives
\begin{eqnarray*}
\frac{d}{2s\alpha^2}\bigg|\dfrac{\|\ v\|_{L^{p+1}}^{p+1}}{\| u_Q\|^{p+1}_{L^{p+1}}}-\lambda^{\frac{2(p-1)}{s}}\bigg|
&\leq&
\Bigg|\dfrac{E[v]}{E[u_Q]}-
\Bigg(\frac{d}{2s}\lambda^{\frac 2s}\bigg(1 -\frac{\lambda^{p-1}}{\alpha^{2}} \bigg)\Bigg)\Bigg|
+\frac{d}{2s}\bigg|\dfrac{\|\nabla v\|^2_{L^2}}{\|\nabla u_Q\|^2_{L^2}}-\lambda^2\bigg|\\
&\leq&\rho_0
\left(\lambda^{\frac{2(p-1)}{s}}+
\frac{d}{2s}\left\{\begin{array}{c}\lambda^{\frac{2(p-1)}{s}} \mbox{  if  } \lambda\leq1 \\\lambda^{\frac{2}{s}} \mbox{   if  } \lambda\geq1\end{array}\right.\right)\leq\frac{d+2s}{2s}\rho_0 \lambda^{\frac{2(p-1)}{s}}.
\end{eqnarray*}
This yields 
\begin{equation}
\label{bound3sc}
\bigg|\dfrac{\| \tilde u\|_{L^{p+1}}^{p+1}}{\| u_Q\|^{p+1}_{L^{p+1}}}-1\bigg|\leq \frac{\alpha^2(d+2s)}{d}\rho_0.
\end{equation}
From (\ref{bound2sc1}) and (\ref{bound3sc}) we have
\begin{equation*}
\bigg|{\| \tilde u\|_{L^{p+1}}}-{\| u_Q\|_{L^{p+1}}}\bigg|+\bigg|{\| \tilde u\|_{L^2}}-{\|u_Q\|_{L^2}}\bigg|+\bigg|{\| \nabla \tilde u\|_{L^2}}-{\| \nabla u_Q\|_{L^2}}\bigg|\leq C(\|u_Q\|_{L^2}) \rho_0.
\end{equation*}

Let $\rho=\frac{\rho_0}{C(\|u_Q\|_{L^2})},\;$ then by Proposition \ref{lion} there exist $\theta \in \R$ and $x_0 \in \Rn$ such that (\ref{sca_lion}) holds for $\tilde u$.  Rescaling to $v$ and then to $u$, completes the proof.
\end{proof}
Next proposition  is ``close to the boundary" behavior.

\begin{prop}\label{near boundary case} Fix $\lambda_0>1.$ There exists $\rho_0=\rho_0(\lambda_0)>0$ (with the property that $\rho_0\to0$ as $\lambda_0\searrow1$) such that for any $\lambda\geq \lambda_0,$  there is NO solution $u(t)$ of $\NLSf$ with P[u]=0 satisfying $\|u\|_{L^2}=\|u_Q\|_{L^2}$, and
$\frac{E[u]}{E[u_Q]}=\frac{d}{2s}\lambda^{\frac{2}{s}}\bigg(1-\frac{\lambda^{p-1}}{\alpha^2}\bigg)$
(i.e., on any ``mass-energy" line corresponding to $\lambda\geq\lambda_0$ and $\ME<1$)
with  $\lambda\leq\g_u(t)\leq\lambda(1+\rho_0)$ for all $t\geq0.$ 
A similar statement holds for $t\leq 0.$
\end{prop}

\begin{proof}
To the contrary, assume that there exists a solution $u(t)$ of (\ref{eq:NLS}) with $\|u\|_{L^2}=\|u_Q\|_{L^2}$,
$\frac{E[u]}{E[u_Q]}=\frac{d}{2s}\lambda^{\frac{2}{s}}\bigg(1-\frac{\lambda^{p-1}}{\alpha^2}\bigg)$  and $\g_u(t)\in[\lambda, \lambda(1+\rho_0)]$.

By continuity of the flow $u(t)$ and Proposition \ref{newchar}, there are  continuous $x(t)$ and $\theta(t)$ such that 
\begin{align}
\Big\|u(x)-e^{i\theta_0}\lambda 
u_Q\big(\lambda
(x-x_0)\big)\Big\|_{L^2}\leq 
\epsilon(\rho),\label{galieanQR}
\end{align}
and 
\begin{align}
\Big\|\nabla\Big[u(x)-e^{i\theta_0}\lambda 
 u_Q\big(\lambda( 
 x-x_0) \Big]\Big\|_{L^2}\leq\lambda
 \epsilon(\rho). 
\label{galieangradQR}
\end{align}

Define $R(T ) = \max\Big\{\max_{0\leq t\leq T} |x(t)|,\log \epsilon(\rho)^{-1}\Big\}$.  Consider the localized variance (\ref{localization_z}).  Note 
$$\frac{d}{s}\lambda^{\frac{2}{s}}E[u_Q]= \lambda^{\frac{2}{s}}{\|\nabla u_Q\|^2_{L^2}}\leq{\|\nabla u(t)\|^2_{L^2}},$$
then,
\begin{align*}
z''_R&= 4d(p-1)E[u]-\big(2d(p-1)-8\big)\|\nabla u\|_{L^2}^2+A_R (u(t))\\
&=16\alpha^2E[u]-8\big(\alpha^2-1\big)\|\nabla u\|_{L^2}^2+A_R (u(t))\leq -8\frac{d}{s} \lambda^{\frac2s}(\lambda^{p-1}-1)E[u_Q]+A_R (u(t)),
\end{align*}
where $A_R(u(t))$ is given by \eqref{ARvar}.

Let $T > 0$ and  for the local virial identity (\ref{virialz}) assume $R = 2R(T )$. Therefore, (\ref{galieanQR}) and (\ref{galieangradQR}) assure that there exists $c_1 > 0$ such that 
$$|A_R (u(t))| \leq  c_1 \lambda^2\big( \epsilon(\rho) + e^{-R(T )}\big)^2 \leq \tilde c_1 \lambda^2 \epsilon(\rho)^2.$$ 
Taking a suitable $\rho_0$ small (i.e. $\lambda > 1$ is taken closer to 1), such that for $0\leq t\leq T$, $\epsilon(\rho)$ is small enough, we get 
$$z''_ R (t) \leq -8\frac{d}{s} \lambda^{\frac2s}(\lambda^{p-1}-1)E[u_Q].$$
Integrating $z''_R(t)$ in time over $[0, T]$ twice, we obtain 
$$\frac{z_R (T )} {T^ 2} \leq \frac{z_R(0)} {T^2} + \frac{z' _R (0)}{ T} -8\frac{d}{s} \lambda^{\frac2s}(\lambda^{p-1}-1)E[u_Q].$$

Note $\sup_{x\in \Rn} \phi(x)$ from (\ref{localization_z}), is bounded, say by $c_2>0$. Then from \eqref{localization_z} we have 
$$|z_R(0)|\leq c_2 R^2 \|u_0\|^2_ {L^2} =  c_2 R^2 \|u_Q\|^2_ {L^2},$$ 
and by (\ref{localization_zder})
$$|z'_ R (0)| \leq c_3 R\|u_0\|^{2(1-s)}_{L^2}\|\nabla u_0\|^{2s}_{L^2} \leq  c_3 R\|u_Q\|^{2(1-s)}_{L^2}\|\nabla u_Q\|^{2s}_{L^2} \lambda^{\frac1s}(1 + \rho_0 ).$$ 
Taking T large enough so that by Lemma \ref{spacial translation} we have $\frac{R (T)}{ T}<\epsilon(\rho)$, we estimate
\begin{align*}
\frac{z_{2R(T)} (T )} {T^ 2} \leq& c_4 \Big(\frac{R(T)^2} {T^2} + \frac{R (T)}{ T}\Big) - 4\frac{d}{s} \lambda^{\frac2s}(\lambda^{p-1}-1)E[u_Q]\\
\leq& C (\epsilon(\rho)^2+\epsilon(\rho)) -4\frac{d}{s} \lambda^{\frac2s}(\lambda^{p-1}-1)E[u_Q] .
\end{align*}
We can initially choose $\rho_0$ small enough (and thus, $\epsilon(\rho_0)$) such that 
$C (\epsilon(\rho)^2+\epsilon(\rho)) < 4\frac{d}{s} \lambda^{\frac2s}(\lambda^{p-1}-1)E[u_Q].$ We obtain 
$0 \leq z_{2R(T )} (T ) < 0,\;$ which is a contradiction, showing that our initial assumption about the existence of a solution to (1.1) with bounded $\g_u(t)$ does not hold.
\end{proof}

Before we exhibit the existence of a critical element/solution, we return to the nonlinear profile decomposition (Proposition \ref{nonradialPDNLS}) and introduce reordering.

\begin{lemma}[Profile reordering]\label{reordering} Suppose $\phi_n=\phi_n(x)$ is a bounded sequence in $H^1(\Rn)$. Let $\lambda_0>1.$ Assume that $M[\phi_n]=M[u_Q]\;$ and  $\;\frac{E[\phi_n]}{E[u_Q]}=\frac{d}{2s}\lambda_n^{\frac{2}{s}}\big(1-\frac{\lambda_n^{p-1}}{\alpha^2}\big)$ such that  $1<\lambda_0\leq\lambda_n$ and $\lambda_n\leq \g_{\phi_n}(t)$
for each n. Apply Proposition \ref{nonradialPDNLS} to the sequence $\{\psi_n\}$ and obtain nonlinear profiles $\{\tpsi^j\}$.   Then, these profiles $\tpsi^j$ can be reordered so that there exist $1\leq M_1\leq M_2\leq M$ and 
\begin{enumerate}
\item For each $1\leq j \leq M_1,$  we have $t^j_n=0$ and $v^j(t)\equiv \NLS(t)\tpsi^j$ does not scatter as $t\to +\infty.$ (In particular, there is at least one such $j$)
\item For each $M_1+1\leq j\leq M_2,$  we  have $t^j_n=0$ and $v^j(t)$ scatters as $t\to +\infty.$ (If $M_1=M_2,$ there are no $j$ with this property.)
\item For each $M_2+1\leq j\leq M,$  we  have $|t^j_n|\to\infty$ and $v^j(t)$ scatters as $t\to +\infty.$ (If $M_2=M,$ there are no $j$ with this property.) 
\end{enumerate}
\begin{proof}
Pohozhaev identities 
\eqref{poz3} and energy definition yield
\begin{align*}
\Bigg(\frac{\|\phi_n\|_{L^{p+1}}}{\|u_Q\|_{L^{p+1}}}\Bigg)^{p+1}=\frac{d}{d-2s}[\g_{\phi_n}(t)]^{\frac{2}{s}}
-\frac{2s}{d-2s}\frac{E[\phi_n]}{E[u_Q]}
\geq\lambda^{\frac{2(p-1)}{s}}_n\geq \lambda_0^{\frac{2(p-1)}{s}}>1.
\end{align*}
Notice that if $j$ is such that $|t^j_n|\to \infty,$ then  $\|\NLS(-t^j_n)\tpsi^j\|_{L^{p+1}}\to0,$ and by  (\ref{aproxL6NL}) we have that 
$\frac{\|\phi_n\|_{L^{p+1}}}{\|u_Q\|_{L^{p+1}}}
\to 0$. Therefore, there exists at least one $j$ such that $t^j_n$ converges. Without loss of generality, assume $t^j_n=0,$ and reorder the profiles such that for $1\leq j\leq M_2$, we have $t^j_n=0$ and for $M_2+1\leq j\leq M$, we have $|t^j_n|\to0$.

It is left to prove that there exists at least one $j$, $1\leq j\leq M_2$ such that $v^j(t)$ is not scattering. Assume that  for all $1\leq j\leq M_2$ we have that all $v^j$ are scattering, and thus, $ \|v^j(t)\|_{L^{p+1}} \to 0$ as ${t\to+\infty}.$ Let  $\epsilon > 0$ and $t_0$ large enough such that for all $1\leq j\leq M_2$ we have $ \|v^j(t)\|_{L^{p+1}}^{p+1}\leq \epsilon/M_2.$ Using $L^{p+1}$ orthogonality (\ref{aproxuL6NL}) along the NLS flow, and letting $n\to +\infty,$ we obtain
\begin{align*}
\lambda_0^{\frac{2(p-1)}{s}}\|u_Q&\|_{L^{p+1}}^{p+1}\leq \| u_n(t) \|^{p+1}_{L^{p+1}} \\
&= \sum^{M_2}_{j=1} \|v^j(t_0) \|^{p+1}_{L^{p+1}} +\sum^{M}_{j=M_2+1} \|v^j(t_0-t^j_n) \|^{p+1}_{L^{p+1}} + \|W^M_n(t) \|^{p+1}_{L^{p+1}} +o_n(1)\\
&\leq \epsilon+ \|W^M_n(t) \|^{p+1}_{L^{p+1}} +o_n(1).
\end{align*}
The last line  is obtained since $\sum^{M}_{j=M_2+1} \|v^j(t_0-t^j_n) \|^{p+1}_{L^{p+1}}\to 0$ as $n\to \infty$, and gives a contradiction.
\end{proof}
\end{lemma}
 Recall that we have a fixed $\lambda_0>1.$
 
\begin{lemma}[Existence of  a threshold solution] \label{existence threshold}
There exists initial data $u_{c,0}\in H^1(\Rn)$ and $1<\lambda_0\leq\lambda_c\leq\sigma_c(\lambda_0)$ such that $u_c(t)\equiv\NLS(t)u_{c,0}$ is a global solution with $M[u_c]=M[u_Q],$ \,$\frac{E[u_c]}{E[u_Q]}=\frac{d}{2s}\lambda_c^{\frac{2}{s}}\bigg(1-\frac{\lambda_c^{p-1}}{\alpha^2}\bigg)$ and, moreover, $\lambda_c\leq\g_{u_c}(t)
\leq \sigma_c$ for all $t\geq0.$
\end{lemma}

\begin{proof}
Definition of $\sigma_c$ implies the existence of sequences $\{\lambda_n\}$ and $\{\sigma_n\}$ with $\lambda_0\leq\lambda_n\leq \sigma_n$  and $\sigma_n \searrow\sigma_c$ such that $\GB(\lambda_n,\sigma_n)$ is false.  This means that there exists $u_{n,0}$ with $M[u]=M[u_Q]$, $\frac{E[u_{n,0}]}{E[u_Q]}=\frac{d}{2s}\lambda^{\frac{2}{s}}\bigg(1-\frac{\lambda^{p-1}}{\alpha^2}\bigg)$ and $\lambda_c\leq\frac{\|\nabla u\|_{L^2}}{\|\nabla u_Q\|_{L^2}}=[\g_u(t)]^{\frac1s}\leq \sigma_c,$
 such that $u_n(t)=\NLS(t)u_{n,0}$ is global.

Note that the sequence $\{\lambda_n\}$ is bounded, thus, we pass to a convergent subsequence $\{\lambda_{n_k}\}$. Assume $\lambda_{n_k}\to \lambda'$  as ${n_k}\to\infty$,  thus $\lambda_0\leq\lambda'\leq\sigma_c.$

We apply the nonlinear profile decomposition (Proposition \ref{nonradialPDNLS}) and reordering (Lemma \ref{reordering}).

In Lemma \ref{reordering}, let $\phi_n=u_{n,0}$. Recall that $v^j(t)$ scatters as $t\to \infty$ for $M_1+1\leq j\leq M_2,$  and by Proposition \ref{nonradialPDNLS}, $v^j(t)$ also scatter in one or the other time direction  for $M_2+1\leq j\leq M$ and  $E[\tpsi^j]=E[v^j]\geq0.$ Thus, by the Pythagorean decomposition for the nonlinear flow (\ref{pythagoreanNLS}) we have
$$\sum_{j=1}^{M_1}E[\tpsi^j]\leq E[\phi_n]+o_n(1).$$
For at least one $1\leq j\leq M_1$, we have $E[\tpsi^j]\leq\max\{\lim_n E[\phi_n],0\}$. Without loss of generality, we may assume $j=1$. Since $1=M[\tpsi^1]\leq\lim_nM[\phi_n]=M[u_Q]=1$, it follows
$
\left(\ME[\tpsi^1]\right)^{\frac1s}\leq \max\bigg(\lim_n\dfrac{E[\phi_n]}{E[u_Q]}\bigg),
$
thus, for some $\lambda_1\geq\lambda_0,$ we have
$
\left(\ME[\tpsi^1]\right)^{\frac1s}=\frac{d}{2s}\lambda_1^{\frac{2}{s}}\bigg(1-\frac{\lambda^{p-1}}{\alpha^2}\bigg).
$

Recall $\tpsi^1$ is a nonscattering solution, thus $[\g_{\psi^1}(t)]^{\frac1s}>\lambda$, otherwise it will contradict Theorem A* Part I (b). We have two cases:  either $\lambda_1\leq\sigma_c$ or $\lambda_1>\sigma_c$.

\noindent
\emph{ Case 1. } $\lambda_1\leq\sigma_c.$ Since the statement ``$\GB(\lambda_1,\sigma_c-\delta)$ is false" implies for each $\delta>0$, there is a nondecreasing sequence $t_k$ of times such that 
$
\lim[\g_{v^1}(t_k)]^{\frac1s}\geq \sigma_c,
$ 
thus,
\begin{eqnarray}
\sigma_c^2-o_k(1)&\leq&\lim[\g_{v^1}(t_k)]^{\frac2s}
\leq\dfrac{\|\nabla v^1(t_k)\|^2_{L^2}}{\|\nabla u_Q\|^2_{L^2}}\notag\\
&\leq&\dfrac{\sum_{j=1}^M\|\nabla v^1(t_k-t_n)\|^2_{L^2}+\|W_n^M(t_k)\|^2_{L^2}}{\|\nabla u_Q\|^2_{L^2}}\label{desigu}\\
&\leq&\dfrac{\|\nabla u_n(t)\|^2_{L^2}}{\|\nabla u_Q\|^2_{L^2}}+o_n(1)
\leq\sigma_c^2+o_n(1)\notag.
\end{eqnarray}
Taking $k\to \infty$, we obtain $\sigma_c^2-o_n(1)=\sigma_c^2+o_k(1)$. Thus, $\|W_n^M(t_k)\|_{H^1}\to 0$  and $M[v^1]=M[u_Q].$ Then, Lemma \ref{HPdecompNLf} yields that for all $t,$
$$
\dfrac{\|\nabla v^1(t)\|^2_{L^2}}{\|\nabla u_Q\|^2_{L^2}}\leq\lim_n\dfrac{\| u_n(t)\|^2_{L^2}}{\|\nabla u_Q\|^2_{L^2}}\leq \sigma_c.
$$
Take $u_{c,0}=v^1(0)(=\psi^1),$ and $\lambda_c=\lambda_1.$

\noindent \emph{ Case 2. } $\lambda_1\geq\sigma_c.$ Note that 
\begin{equation}
\lambda_1^2\leq\lim[\g_{v^1}(t_k)]^{\frac2s}.
\label{reepdesi}
\end{equation}
Replacing the first line of \eqref{desigu} by \eqref{reepdesi}, taking $t_k=0$ and sending $n\to+\infty$, we obtain
\begin{eqnarray*}
\lambda_1^2&\leq&\dfrac{\|v^1(t_k)\|^2_{L^2}\|\nabla v^1(t_k)\|^2_{L^2}}{\| u_Q\|^2_{L^2}\|\nabla u_Q\|^2_{L^2}}
\leq\dfrac{\|\nabla v^1(t_k)\|^2_{L^2}}{\|\nabla u_Q\|^2_{L^2}}\notag\\
&\leq&\dfrac{\sum_{j=1}^M\|\nabla v^1(t_k-t^j_n)\|^2_{L^2}+\|W_n^M(t_k)\|^2_{L^2}}{\|\nabla u_Q\|^2_{L^2}}\\
&\leq&\dfrac{\|\nabla u_n(t)\|^2_{L^2}}{\|\nabla u_Q\|^2_{L^2}}+o_n(1)
\leq\sigma_c^2+o_n(1).\notag
\end{eqnarray*}
Thus, we have  $\lambda_1\leq\sigma_c,$ which is a contradiction. Thus, this case cannot happen.
 \end{proof}

\begin{lemma}{\label{compactH1}}
Assume $u(t)=u_c(t)$ to be the critical solution provided by Lemma \ref{existence threshold}. Then there exists a path $x(t)$ in $\Rn$ such that 
$$
K=\{u(\cdot-x(t),t)|t\geq 0\}
$$
has a compact closure in $H^1(\Rn)$.
\end{lemma}

\begin{proof}
As we proved in Lemma \ref{precompact-localization}, it suffices to show that for each sequence of times $t_n\to\infty$, passing to a subsequence, there exists  a sequence $x_n$ such that $u(\cdot-x_n,t_n)$ converges in $H^1$. Let   $\phi_n = u(t_n)$ as in Proposition \ref{reordering}, and apply  the proof of Lemma \ref{existence threshold}. It follows for $j\geq2$ we have $\psi_j =0$ and  $\wt_n^M \to 0$ in $H^1$ as $n\to \infty$. And thus, $u(\cdot-x_n,t_n)\to \psi^1$ in $H^1$.
\end{proof}

\begin{lemma} [Blow up for \emph{a priori} localized solutions]\label{blowup localization} 
Suppose $u$ is a solution of the $\NLSf$ at the mass-energy line $\lambda>1$, with $\g_u(0)>1$. Select $\kappa$ such that
$0<\kappa<\min(\lambda-1,\kappa_0)$, where $\kappa_0$ is an absolute
constant. Assume that there is a radius $R\gtrsim \kappa^{-1/2}$
such that for all $t$, we have 
$\g_{u_{R}}(t):=\dfrac{\|u\|_{\Lt(|x|\geq R)}^{1-s}\|\nabla u(t)\|_{\Lt(|x|\geq R)}^{s}}{\|\uQ\|_{\Lt(|x|\geq R)}^{1-s}\|\nabla \uQ\|_{\Lt(|x|\geq R)}^{s}}\lesssim \kappa.$
Define $\tilde r(t)$ to be the scaled local variance:
$
 r(t) =  \frac{z_R(t)}{32 \alpha^2 E[\uQ] \left(\frac{d}{2s}\lambda^{\frac2s}
\left(1-\frac{\lambda^{p-1}}{\alpha^2}-\kappa\right)\right)} \,.
$

Then blowup occurs in forward time before $t_b$ (i.e., $T^* \leq
t_b$), where
$
t_b = r'(0) + \sqrt{ r'(0)^2 + 2 r(0)} \,.
$
\end{lemma}

\begin{proof}
By the local virial identity \eqref{ridigity-eq1},
$$
r''(t) =
\frac{16\alpha^2E[u]-8(\alpha^2-1)\|\nabla u\|_{L^2}^2+{A_R(u(t))}}{16 \alpha^2 E[\uQ] \left(\frac{d}{2s}\lambda^{\frac2s}
\left(1-\frac{\lambda^{p-1}}{\alpha^2}\right)-\kappa\right)},
$$
where  
$$
\big|A_R(u(t))\big|  = \| \nabla u(t) \|^2_{\Lt(|x|\geq R)}  + \frac{1}{R^2}  \|u(t)\|^2_{\Lt(|x|\geq R)} + \|u(t)\|^{p+1}_{L^{p+1}(|x|\geq R)}.
$$

Note that, $E[\uQ]=\frac ds\|\nabla \uQ\|_{L^2}^2$ and definition of the mass-energy line yield
\begin{align}
\frac{16\alpha^2E[u]-8(\alpha^2-1)\|\nabla u\|_{L^2}^2}{16 \alpha^2 E[\uQ] }
&=\frac{E[u]}{E[\uQ]}-\frac{d\|\nabla u\|_{L^2}^2}{sE[\uQ]}\notag\\
&=\frac{E[u]}{E[\uQ]}-\frac{\|\nabla u\|_{L^2}^2}{\|\nabla \uQ\|_{L^2}^2}\label{P2c}\\
&\leq \frac{d}{2s}\lambda^{\frac2s}
\left(1-\frac{\lambda^{p-1}}{\alpha^2}\right)-[\g_u(t)]^2\label{P3c}.
\end{align}
In addition, we have the following estimates 
$$\| \nabla u(t) \|^2_{\Lt(|x|\geq R)}\lesssim \kappa,\quad\quad\quad\quad 
\frac{\| u(t) \|^2_{\Lt(|x|\geq R)}}
{R^2}=\frac{\| \uQ \|^2_{\Lt}}
{R^2}\lesssim \kappa,$$
\begin{align} \label{magia3}
 \|u(t)\|^{p+1}_{L^{p+1}(|x|\geq R)}&\lesssim
 \|\nabla u\|^{\frac{d(p-1)}{2}}_{L^{2}(|x|\geq R)}\| u\|^{2-\frac{(d-2)(p-1)}{2}}_{L^{2}(|x|\geq R)}\\
&\lesssim[\g_{u_{R}}(t)]^{2}
\left(\|\nabla \uQ\|^{s}_{L^{2}}\| \uQ\|^{1-s}_{L^2}\right)^{p-1}\lesssim \kappa\notag.
\end{align}
We used the Gagliardo-Nirenberg to obtain \eqref{magia3} and noticing that $\|\nabla \uQ\|^{s}_{L^{2}}$ and $\| \uQ\|^{1-s}_{L^2}$ are constants, the last expression is estimated by $\kappa$ (up to a constant). In addition, $\g_u(t)>1,$ then $\kappa\lesssim \kappa[\g_u(t)]^2$. Applying the above estimates, it follows
$$
r''(t) \lesssim
\frac{\frac{d}{2s}\lambda^{\frac2s}
\left(1-\frac{\lambda^{p-1}}{\alpha^2}\right)-[\g_u(t)]^2(1-\kappa)}
{\frac{d}{2s}\lambda^{\frac2s}
\left(1-\frac{\lambda^{p-1}}{\alpha^2}\right)-\kappa}.$$
Since $\g_u(t)\geq\lambda$, we obtain
$
\tilde r''(t) \leq -1  \,.
$
which is a contradiction.  Now integrating in time twice gives
$
r(t) \leq -\frac12 t^2 + r'(0)t + r(0) \,.
$

The positive root of the polynomial on the right-hand side is $t_b=r'(0) + \sqrt{ r'(0)^2 + 2 r(0)}.$
\end{proof}
This concludes all the claims in steps 1, 2 and 3 in subsection \ref{outline weak} and finishes the proof of Theorem A* part II  (b).

\bibliographystyle{amsalpha}

\end{document}